%% file: main.tex
\documentclass[11pt,reqno]{amsart}

\usepackage[UKenglish]{babel}
\usepackage{lmodern}
\usepackage{amsmath}
\usepackage{amssymb}
\usepackage{amsthm}
\usepackage{geometry}
\usepackage{tikz, tikz-cd}
\usepackage{enumitem}
\usepackage{sseq}
\usepackage[utf8]{inputenc}
\usepackage[T1]{fontenc}
\usepackage{csquotes}
\usepackage[
backend=biber,
style=alphabetic,
sorting=nty
]{biblatex}
\usepackage[pdfencoding=auto, psdextra]{hyperref}

\newtheorem{theorem}{Theorem}[section]
\newtheorem{proposition}[theorem]{Proposition}
\newtheorem{lemma}[theorem]{Lemma}
\newtheorem{corollary}[theorem]{Corollary}
\theoremstyle{definition}
\newtheorem{definition}[theorem]{Definition}
\newtheorem{example}[theorem]{Example}
\newtheorem{remark}[theorem]{Remark}
\newtheorem{claim}[theorem]{Claim}

\newcommand{\set}[1]{\left\{ #1 \right\}}
\newcommand{\ls}[2]{#1_{1},\dots,#1_{#2}}
\newcommand{\abs}[1]{\left | #1 \right |}
\newcommand{\mc}[1]{\ensuremath{\mathcal{#1}}}

\newcommand{\et}{\mathrm{\acute{e}t}}
\newcommand{\tikzpullback}{\arrow[phantom,dr,"\ulcorner",very near start]}
\newcommand{\BC}{\ensuremath{\mathcal{BC}}}
\newcommand{\BT}{\ensuremath{\mathcal{BT}}}
\newcommand{\powerseries}[1]{[\! [ #1 ]\! ]}

\newcommand{\smallvector}[2]{{\begin{pmatrix} #1 \\ #2\end{pmatrix}}}
\newcommand{\inv}{{-1}}
\DeclareMathOperator{\Q}{\mathbb{Q}}

\DeclareMathOperator{\Z}{\mathbb{Z}}

\DeclareMathOperator{\F}{\mathbb{F}}

\DeclareMathOperator{\D}{\mathbb{D}}
\DeclareMathOperator{\Proj}{\mathbb{P}}
\DeclareMathOperator{\annu}{\mathbb{A}}
\DeclareMathOperator{\Aut}{Aut}

\DeclareMathOperator{\Bun}{Bun}
\DeclareMathOperator{\Perf}{Perf}
\DeclareMathOperator{\Nrd}{Nrd}
\DeclareMathOperator{\GL}{GL}
\DeclareMathOperator{\SL}{SL}

\DeclareMathOperator{\Gal}{Gal}

\DeclareMathOperator{\Hom}{Hom}
\DeclareMathOperator{\Ext}{Ext}
\DeclareMathOperator{\spa}{Spa}
\DeclareMathOperator{\spd}{Spd}
\DeclareMathOperator{\rad}{rad}
\DeclareMathOperator{\ind}{Ind}
\DeclareMathOperator{\cind}{c-Ind}
\DeclareMathOperator{\res}{Res}
\DeclareMathOperator{\st}{St}
\DeclareMathOperator{\Pic}{Pic}
\DeclareMathOperator{\isoc}{Isoc}
\DeclareMathOperator{\Div}{Div}

\title{Gluing sheaves along Harder-Narasimhan strata of \texorpdfstring{$\Bun_G$}{BunG}}

\author{Jon Miles}

\date{}

\begin{document}

\begin{abstract}
	We describe how to glue prime-to-$p$ torsion sheaves along Harder-Narasimhan strata of Fargues-Scholze's $\Bun_G$ in terms of the cohomology of locally closed strata inside the smooth charts $\mc M_b \to \Bun_G$ constructed in \cite{fargues2021geometrization}, which are moduli of certain split parabolic bundles. Our computations for $G=\GL_2$ explicitly describe images of geometric constant term functors when restricted to a distinguished point inside $\Bun_{T_{\{b\}}}$, where $T_{\{b\}}$ is the split inner form of the Levi quotient of the corresponding parabolic subgroup $P_b$. 
\end{abstract}

\maketitle

\tableofcontents

\section{Introduction}

\input{Introduction.tex}

\section{Vector Bundles on the Curve}
\input{VB_on_curve}

\section{Gluing sheaves along strata of \texorpdfstring{$\Bun_G$}{BunG}}
\input{Semi-Orth.tex}

\section{Explicit Computations for \texorpdfstring{$\Bun_2$}{Bun2} in small codimension}
\input{Some_Examples.tex}

\appendix
\section{Smooth Representations}
\input{Smooth_Reps.tex}

\printbibliography

\end{document}

%% file: Introduction.tex
\subsection{Motivation and overview}
The geometrization of the local Langlands correspondence by certain sheaf categories on the moduli of analytic vector bundles on relative Fargues-Fontaine curves has proved to be a fundamental advancement in geometric representation theory. The geometry of Fargues-Scholze's $\Bun_G$ is interesting in its own right within the context of rational $p$-adic Hodge theory of perfectoid spaces. In particular its geometric points correspond to isomorphism classes of $G$-isocrystals, whose slope decompositions induce a canonical Harder-Narasimhan stratification of the stack. In this paper, we study the \'etale cohomology of the $\mc M_b$ coverings of open strata $\Bun_{G}^{\leq b}$ \cite[\hphantom{}V.3]{fargues2021geometrization} and their role in understanding the smooth representation theory of $p$-adic reductive groups in terms of (prime-to-$p$ torsion) sheaves $D_\et(-)$ restricted to locally closed strata.

Classically, there are functors between smooth representation categories of Levi subgroups $M$ of a fixed reductive group $G$ given by parabolic induction and Jacquet restriction that are crucial in classifying and understanding smooth $G$-representations. These are categorified by geometric Eisenstein series and constant term functors between $\Bun_M$ and $\Bun_G$ (cf. \cite{hamann2024geometriceisensteinseriesi}), and it was shown in loc. cit. that the cohomology of $\mc M_b$ can be identified with certain stalks of constant term functors. In this way, the gluing data between sheaves along Harder-Narasimhan strata of $\Bun_G$ can be thought of as a geometric analogue of Jacquet restrictions, which are most interesting when restricted to subquotients of parabolically induced representations. 

There are two main purposes of this article: the first is to formalize the role of certain strata $\mc M_{b_2}^{b_1}\hookrightarrow \mc M_{b_2}$ and their cohomology in the understanding of the infinite semi-orthogonal decomposition of $D_\et(\Bun_G)$ with prime-to-$p$ torsion coefficients. The second is to be a very explicit introduction to objects appearing in the automorphic side of the Fargues-Scholze program and how smooth representations appear in the étale cohomology of diamonds in the $\GL_2$ case. A large part of the article is to compute some examples of gluing data along the first specialization from basic strata for $\Bun_2$. Notably, we only use methods from the automorphic side. We hope that the computations appearing here provide some insight to the connections between constant term functors and more classical categories of smooth representations.

\subsection{Main results for general reductive groups}
Let $E$ be a local field with residue characteristic $p$, $G$ a connected reductive group over $E$, and $\Lambda$ a torsion ring killed by an integer $n$ that is prime to $p$. The goal of this paper is to understand the infinite semi-orthogonal decomposition of $D_\et(\Bun_G):=D_\et(\Bun_G,\Lambda)$ induced by excision triangles along Harder-Narasimhan strata in terms of the cohomology of certain diamonds. More precisely, for $b\in B(G)$ one considers the quasicompact open substack $\Bun_G^{\leq b}$ consisting of points that specialize to $b$ (i.e. ``more semistable'' than $b$), which has $b$ as its unique closed point. Let $i:\Bun_G^b \hookrightarrow \Bun_G^{\leq b}$ be the inclusion with open complement $j:\Bun_G^{<b}\hookrightarrow \Bun_G^{\leq b}$. Then there are the usual functorial excision triangles \[Ri_* i^! A\to A\to Rj_*j^* A\] for any $A$ in $D_\et(\Bun_G^{\leq b})$, so that one can recover the original triangle by a map $i^* R j_*B\to C$ in $D_\et(\Bun_G^b)$, up to adjunction. Let $\underline{G_b(E)}$ denote the group of connected components of the stabilizer of $\Bun_G^b$, so that $D_\et(\Bun_{G}^b)$ is canonically equivalent to the derived category of smooth $G_b(E)$-representations in $\Lambda$-modules. Then since homomorphisms are well understood in this category, the problem of which sheaves glue non-trivially is essentially reduced to understanding the functor $i^* Rj_*$.

\subsubsection{Cohomology of \texorpdfstring{$\mc M_b$}{Mb} and relation to constant term functors}
Our first result is constructing natural equivalences between $i_b^*$ and the cohomology of sheaves on the smooth covers $\mathcal{M}_{b}\to \Bun_G$, defined in \cite[Chapter~V.3]{fargues2021geometrization} as the moduli stack of $G$-bundles with ascending separated and exhaustive filtrations on fibre functors indexed by $\lambda\in \Q$ and with associated graded that is  $v$-locally isomorphic to $\mc E_b$. Then there is a canonical $\underline{G_b(E)}$-torsor $\widetilde{\mc M_{b}}\to \mc M_b$ given by fixing $v$-local trivializations of the associated graded bundle. Let $\phi: \widetilde{\mc M_{b}} \to \Bun_G$ denote the composition.
\begin{proposition}[Proposition \ref{action-on-cohomology}]
For any $A\in D_\et(\Bun_G)$, there is a natural isomorphism of underlying $\Lambda$-modules \[(i_{b}^* A)_\Lambda \cong R\Gamma (\widetilde{\mc M_{b}}, \phi^*A)\] in the derived category of smooth $\Lambda$-modules $D(\Lambda)$ which is equivariant for the natural $G_{b}(E)$-structure on the right. 
\end{proposition}
This is an immediate consequence of (\cite[Theorem~V.4.2]{fargues2021geometrization}), which identifies the cohomology of $\widetilde{\mc M_b}$ with the pullback to the special point corresponding to the bundle with everywhere split filtration. In this sense, Fargues-Scholze call $\widetilde{\mc M_{b}}$ the ``strict henselization'' of the point $i_b:\Bun_G^b\hookrightarrow \Bun_G$. As the space $\widetilde{\mc M_b}$ is isomorphic to an iterated extension of negative Banach-Colmez spaces, this is one of several instances where it is useful to think of Banach-Colmez spaces as playing the role of henselian local rings in the geometry of $v$-stacks (at least on the level of $\ell$-adic \'etale sheaves). Morally speaking, the equivariant version of our computation can be seen as a sort of nearby cycles functor for local systems over the punctured space $\widetilde{\mc M_b^\circ} \to *$.

We remark that the stack $\mc M= \bigsqcup_{b\in B(G)}\mc M_b$ seems to be closely related to an analogue of the moduli of Simpson's mixed twistor structures over the twistor projective line, which serve as a generalization of mixed real Hodge structures (cf. \cite{simpson1997mixed}). In a slightly different direction, Hamann and Imai identify $\mc M_b\cong \Bun_{P_b}^{b'}$ as a certain basic point of the moduli of parabolic bundles for the dynamic parabolic $P_b$ with stabilizer isomorphic to $\underline{G_b(E)}$\footnote{The notation $b'$ is our own; for more details see \cite[Corollary~1.7]{hamann2024dualizingcomplexesmoduliparabolic}.} . Recently, Hamann, Hansen, and Scholze used this identification to identify the renormalized pullback $i_{b}^{\mathrm{ren}*}$ with a stalk\footnote{The stalk is supported at a semistable point of $\Bun_{G_{\{b\}}}$ given by a certain basic reduction of $b$ to $B(G_{\{b\}})$, where $G_{\{b\}}$ is the Levi quotient of $P_{\{b\}}$; furthermore $G_{\{b\}}$ is a quasi-split inner form of $G_b$.} of the constant term functor $\mathrm{CT}_{P_{\{b\}}^{-}!}\cong \mathrm{CT}_{P_{\{b\}}^{}*}$ \cite[Corollary~2.2.5]{hamann2024geometriceisensteinseriesi} after identifying the two under second adjointness \cite[Theorem~3.2.1]{hamann2024geometriceisensteinseriesi}. The renormalization does not fundamentally affect the image, since $i_b^*$ and $i_b^{\mathrm{ren}*}$ differ only by a shift by the codimension of $\Bun_G^b\hookrightarrow \Bun_G$ and twist by a square root of the modulus character $\delta_b$ of $P_{\{b\}}$. Our result can thus be seen as giving an explicit tower of small $v$-stacks $\widetilde{\mc M_{b}}\cong \varprojlim_K \widetilde{\mc M_{b}}/\underline{K}$ as $K$ ranges through the open pro-$p$ subgroups of $G_b(E)$, whose cohomology realizes a distinguished ``classical'' part of the constant term functors.

\subsubsection{Cohomology of \texorpdfstring{$\mc M_{b_2}^{b_1}$}{Mb1b2} and geometric Eisenstein series}
We mostly restrict to the setting where the sheaf on $\Bun_G^{\leq b}$ arises as the restriction of $Ri_{b_1,*} A$ for some $\Bun_G^{b_1}\hookrightarrow \Bun_G$ factoring through $\Bun_G^{\leq b}\setminus \Bun_G^{b}$; in this setting, we relabel $b_2=b$. Using the six functor formalism for $D_\et(-,\Lambda)$, it is possible to reduce $i_{b_2}^*Ri_{b_1,*}$ to the cohomology of a locally closed stratum inside of $\widetilde{\mc M_{b_2}}$. Namely, we define via pullback 
\begin{equation*}
\begin{tikzcd}
\widetilde{\mc M^{b_1}_{b_2}}\arrow[d,"\widetilde{i_{b_1}}",swap] \arrow[r,"g_{b_1}^{}"] \tikzpullback & \mc M^{b_1}_{b_2}\arrow[d,swap,"i_{b_1}' "]\arrow[r,"f_{b_1}"]\tikzpullback & {[*/ \mc{G}_{b_1}]}\arrow[d, "i_{b_1}"]\\
\widetilde{\mc M_{b_2}}\arrow[r,"g"] & \mc M_{b_2} \arrow[r,"f"]& \Bun_G,
\end{tikzcd}
\end{equation*}
and note that there is a corresponding tower $\widetilde{\mc M_{b_2}^{b_1}}\cong \varprojlim_K \widetilde{\mc M_{b_2}^{b_1}}/\underline{K}$ as above with each of the finite steps forming an \'etale covering of $\mc M_{b_2}^{b_1}$. We use the notation $\left(R\Gamma(\widetilde{\mc M_{b_2}^{b_1}},g_{b_1}^*\mc F)\right)^\mathrm{sm}:=\varinjlim_K R\Gamma(\widetilde{\mc M_{b_2}^{b_1}}/\underline{K},\mc F)$ for any sheaf $\mc F$ on $\mc M_{b_2}^{b_1}$, which we call the smooth cohomology of $\mc F$.

\begin{theorem}[Theorem \ref{Smooth-cohomology-of-tower}]
There is a natural isomorphism of functors $D_\et([*/\mc G_{b_1}])\to D_{\et}(*)\cong D(\Lambda)$
\[R\Gamma(\widetilde{\mc M_{b_2}},-) \circ (fg)^*R i_{b_1,*} \xrightarrow{\sim} \left(R\Gamma(\widetilde{\mc M_{b_2}^{b_1}},g_{b_1}^*(-))\right)^\mathrm{sm}\circ f_{b_1}^*.\]
Furthermore, the induced smooth $G_{b_2}(E)$-actions on either side agree, and remembering the equivariant structure computes $i_{b_2}^* Ri_{b_1,*}$.
\end{theorem}

Also in \cite[Corollary~2.2.5]{hamann2024geometriceisensteinseriesi}, the authors identify the similarly renormalized $i_{b_1,*}^\mathrm{ren}$ as the $*$-pushforward along the same distinguished point of $\Bun_{G_{\{b_1\}}}$ followed by the geometric Eisenstein series $\mathrm{Eis}_{P_{\{b_1\}}^- *}$, which can be thought of as a geometric analogue of parabolic induction. Insomuch as there is an adjunction $\mathrm{CT}_{P_{\{b\}}^- !}\dashv \mathrm{Eis}_{P_{\{b\}}^- *}$ \cite[Lemma~2.1.8]{hamann2024geometriceisensteinseriesi}, one can think of the constant term functor as a geometric analogue of the Jacquet restriction functor, and this analogy can be seen explicitly in our computations for $G=\GL_2$ (see for example the appearance of the Jacquet module in the proof of Proposition \ref{non-LS-type}). Then Theorem \ref{Smooth-cohomology-of-tower} can be loosely interpreted as giving a tower whose cohomology approximates the image of geometric Eisenstein series applied to complexes of smooth $G_{b_1}(E)$-representations, but only at the stalks $\Bun_G^{b_2}$ that are specializations of $b_1$. 

\subsubsection{Restriction to admissible representations and Verdier duality}
Finally, there is a similar $\underline{G_{b_1}(E)}$-torsor $\widetilde{\widetilde{\mc M_{b_2}^{b_1}}}\to \widetilde{\mc M_{b_2}^{b_1}}$ given by partially uniformizing\footnote{Note that for basic $b_1$, $\mc G_{b_1}=\underline{G_{b_1}(E)}$ and the distinction is irrelevant. In general, the fiber groupoids of the partial uniformization can be thought of as $G$-bundles everywhere locally of type $\mc E_{b_1}$ equipped with trivializations of the associated graded bundle with respect to the Harder-Narasimhan filtration for any representation of $G$.} the $\mc G_{b_1}$-action on $\widetilde{\mc M_{b_2}^{b_1}}$ induced by automorphisms of the total bundle. After basechange to a geometric point $\spa (C, C^+)$, the double-uniformized space often has a concrete description in terms of positive Banach-Colmez spaces, whose prime-to-$p$ cohomology is very simple; more precisely, the cohomology can be described in terms of excision and cohomology of polydisks. If we restrict to $\mathrm{ULA}$ sheaves (for example $Ri_{b_1,*}[\sigma]$ for an admissible smooth representation $\sigma$ of $G_{b_1}(E)$), we can apply Verdier duality on $\Bun_G^{\leq b_2}$ to rather compute the compactly supported cohomology of $\widetilde{\mc M_{b_1}^{b_2}}$, which can be obtained from the compactly supported cohomology of $\widetilde{\widetilde{\mc M_{b_1}^{b_2}}}$ by a homological Hochschild-Serre descent spectral sequence for $\underline{G_{b_1}(E)}$ (see Proposition \ref{homological-Hochschild-Serre}). We make extensive use of tools that are available for compactly supported cohomology, such as the general basechange, Künneth, and projection formulas, as well as the excision triangles for $!$-pushforwards. In particular, there is the following standard lemma in equivariant geometry (descent along torsors) which reduces the computation even further to compactly supported cohomology of the trivial sheaf.
\begin{lemma}[Lemma \ref{equivariant-projection-formula}]
Let $p:Y\to X$ be a $\underline{G}$-torsor in small $v$-stacks for a totally disconnected Hausdorff group $G$. Then there is a cartesian diagram
\begin{center}
\begin{tikzcd}
Y\arrow[r,"g"]\arrow[d,"p",swap] & *\arrow[d,"q"]\\
X\arrow[r,"f",swap] & {[*/\underline{G}]}.
\end{tikzcd}
\end{center}
For any smooth $G$-representation $(\sigma,V)$ corresponding to a sheaf $ [\sigma] \in D_\et([*/\underline{G}])$, there is an isomorphism of underlying $\Lambda$-modules
\[R\Gamma_c(Y,p^*f^*[\sigma])\cong V \otimes^{\mathbb{L}}_\Lambda R\Gamma_c(Y,\Lambda).\]
Furthermore, $Rf_! f^*[\sigma]\cong [\sigma]\otimes^{\mathbb L}_{\Lambda} f_! \Lambda \in D(G)$, identifying the $G$-equivariant structure of $R\Gamma_c(Y, p^* f^*[\sigma])$ with the equivariant structure of $V\otimes_{\Lambda}^{\mathbb{L}} R\Gamma_c(Y,\Lambda)$ obtained from the canonical descent datum via pullback along $q$.
\end{lemma}
Applying the lemma to $p: \widetilde{\widetilde{\mc M_{b_2}^{b_1}}}\to \widetilde{\mc M_{b_2}^{b_1}}$ gives an identification (with slight abuse of notation for the sheaf $[\sigma]$ pulled back to $\widetilde{\mc M_{b_2}^{b_1}}$)
\[R\Gamma_c(\widetilde {\mc M_{b_2}^{b_1}},[\sigma])\cong \left (\sigma \otimes_{\Lambda}^{\mathbb{L}} R\Gamma_c(\widetilde{\widetilde{ \mc M_{b_2}^{b_1}}},\Lambda)\right )\otimes^{\mathbb{L}}_{\mc H(G_{b_1}(E))} \Lambda,\]
where $-\otimes^{\mathbb{L}}_{\mc H(G_{b_1}(E))} \Lambda$ is the smooth group homology functor.
The identification is furthermore $G_{b_2}(E)$-equivariant, and the righthand-side (up to Verdier duality) will be our main approach to understanding the functor $i_{b_2}^*Ri_{b_1,*}[\sigma]$.

\subsection{Main results and methods when \texorpdfstring{$G=\GL_2$}{G=GL2}}
A large part of the paper is concerned with computing the images of sheaves along immediate specializations from the semistable locus of $\Bun_2$. The geometry of the spaces $\mc M_{b_2}^{b_1}$ is already interesting here. 
\subsubsection{Half-integral slope}
Let $b_1\in B(\GL_2)$ be the basic point corresponding to the bundle $\mc O(1/2)$ and $b_2$ be its immediate specialization corresponding to $\mc O\oplus \mc O(1)$, where the stabilizers are the units of quaternions $\mc G_{b_1}\cong \underline{D^\times}$ and a split rank $2$ torus $\pi_0 \mc G_{b_2}=\underline{G_{b_2}(E)}\cong \underline{E_1^\times}\times \underline{E_2^\times}$, respectively\footnote{The notation $E_1^\times=E^\times$ and $E_2^\times=E^\times$ is important to distinguish the two coordinates of the split torus since their geometric actions on moduli spaces are fundamentally different.}. Using ideas related to Hamann's analogue of Drinfeld compactification of $\Bun_{P_{b_2}}$, we construct in Proposition \ref{double-uniformized-slope-1/2-cover} a $\underline{D^\times}\times (\underline{E_1^\times}\times \underline{E_2^\times})$-equivariant isomorphism
\[\widetilde{\widetilde{ \mc M_{b_2}^{b_1}}}\cong(\BC(\mc O(1/2))\setminus{0})\times \underline{E^\times}. \]
Further quotienting the $\underline{D^\times}$-action gives an identification $\widetilde{\mc M_{b_2}^{b_1}} \cong \BC(\mc O(-1)[1])\setminus 0$. Then the cohomology of the negative Banach-Colmez space can be rather understand in terms of equivariant cohomology of a product of positive Banach-Colmez spaces.
\begin{proposition}[Proposition \ref{double-uniformized-half-slope-cs-coh}]
The compactly supported cohomology of $\widetilde{\widetilde{\mc M_{b_2}^{b_1}}}$ as $D^\times$-modules is 
\[H_c^i(\widetilde{\widetilde{\mc M_{b_2}^{b_1}}},\Lambda) =\begin{cases} 
\Lambda\otimes \cind_{\SL_1(D)}^{D^\times}\Lambda,\quad & i=1\\
\abs{\Nrd}^{}\otimes \cind_{\SL_1(D)}^{D^\times}\Lambda,\quad & i=2\\
0,\quad & \text{else.}
\end{cases}
\]
\end{proposition}
This computation only uses the Künneth formula, excision, and the identification of the dualizing sheaf of Banach-Colmez spaces as input. Then for any irreducible smooth representation $\sigma$ of $D^\times$, the image $i_{b_2}^* Ri_{b_1,*}[\sigma]$ is essentially given by the smooth $\SL_1(D)$-homology of $\sigma$, which reduces to the study of orbits since $\SL_1(D)$ has split-rank $0$. In particular, we deduce the following for powers of the reduced norm $\abs{\Nrd}:D^\times \to \Lambda$.
\begin{theorem}
The cohomology of $\abs{\Nrd}^k$ as a sheaf on $\widetilde{\mc M_{b_2}^{b_1}}$ as $E_1^\times \times E_2^\times$-modules is 
\[H^i(\widetilde{\mc M_{b_2}^{b_1}},\abs{\Nrd}^k) =\begin{cases} 
\abs{-}^k,\quad & i=0\\
\abs{-}^k\delta_T,\quad & i=1\\
0,\quad & \text{else.}
\end{cases}
\]
\end{theorem}
The remaining irreducible smooth $D^\times$-modules are finite dimensional and compactly induced from a compact-mod-center open subgroup $K\subset D^\times$, and the images of the gluing functors reduce to $K\cap \SL_1(D)$-orbits of smooth characters of $K$.
\subsubsection{Integral slope}
We can now describe the main computational results of this paper, which concern the first specialization of integral slope components of $\Bun_2$; namely when $b_1$ corresponds to $\mc O^2$ and $b_2$ corresponds to $\mc O(-1)\oplus \mc O(1)$. The relevant computations for integral slope components are smooth group homology of parabolically induced representations, yielding Jacquet modules with their induced torus actions. These representations arise geometrically from the relative cohomology of sheaves along a fibration $Z\setminus \set{0}\to \underline{\Proj^1(E)}$, where the fiber over $\infty$ is the moduli of certain bounded modifications of line bundles on the curve, equipped with an action of the standard Borel subgroup (see Lemma \ref{geometry-of-Z}). More precisely, $Z$ arises as the boundary of a partial (``toroidal'') compactification of $\widetilde{\widetilde{\mc M_{b_2}^{b_1}}}$ by a positive Banach-Colmez space $\BC(\mc O(1))^2$ (up to a factor of $\underline{E^\times}$ corresponding to the determinant; see Proposition \ref{Equivariant-iso-cover}), which is the preimage of $\underline{\Proj^1(E)}$ for the natural map to the associated projectivized Banach-Colmez space $(\BC(\mc O(1)^2)\setminus 0)/\underline{E^\times} $. Let $U$ denote the open complement of $Z\subset \BC(\mc O(1)^2)$. Then using excision sequences, Künneth formula, and other tools of six functors, one deduces the following.
\begin{proposition}[Proposition \ref{cs-coh-double-tilde-M-trivial-module}]
The compactly supported cohomology of $\widetilde{\widetilde{\mc M_{b_2}^{b_1}}}\cong U \times \underline{E^\times}$ as $G_{b_1}(E)$-modules is 
\[H_c^i(\widetilde{\widetilde{\mc M_{b_2}^{b_1}}},\Lambda) =\begin{cases} \st \otimes \cind_{\SL_2(E)}^{\GL_2(E)}\Lambda ,\quad & i=2\\
\left(\ind_{B}^{\GL_2(E)}\abs{-}\otimes 1\right)\otimes \cind_{\SL_2(E)}^{\GL_2(E)} \Lambda ,\quad & i=3\\
\abs{\det} \otimes \cind_{\SL_2(E)}^{\GL_2(E)} \Lambda,\quad & i=4\\ 
0,\quad & \text{else}.
\end{cases}
\]
\end{proposition}
We also compute the compactly supported cohomology modulo a compact open subgroup $K\subset \GL_2(E)$, which reduces to identifying the orbits of the action of congruence subgroups in $\SL_2(\mc O_E)$ on $\Proj^1(E)=B\backslash \SL_2(E)$ (Proposition \ref{gluing-compact-generators}). Notably, these cohomology groups correspond to the gluing data for the smooth representations $\cind_K^{\GL_2(E)}\Lambda$ and form a collection of compact generators of $D_\et(\Bun_2^1)\cong D(\GL_2(E),\Lambda)$. Furthermore, we compute the images $i_{b_2}^*Ri_{b_1,*}$ for several classes of irreducible smooth representations of $\GL_2(E)$. For unramified characters, we have the following.
\begin{theorem}[Theorem \ref{Purity-for-trivial-rep}]
Let $k$ be an integer (more generally, $k\in \frac{1}{2}\Z$ in case there is a fixed $\sqrt{q}\in \Lambda$). Then 
\[i_{b_2}^* Ri_{b_1,*} \abs{\det}^{k} = \abs{-}^{k}\oplus \abs{-}^{k}\delta_T [-3].\]
In particular, $i_{b_2}^* Ri_{b_1,*} \Lambda = \Lambda[0] \oplus \delta_T [-3]$.
\end{theorem}
We compute directly that the first step in the semi-orthogonal decomposition splits when restricted to the cuspidal irreducible representations of $\GL_2(E)$ (Proposition \ref{vanishing-of-cuspidals}). The crucial input here is the vanishing of smooth $\SL_2(E)$-homology of cuspidal representations since there are no Iwahori fixed vectors. 

Finally, we restrict attention to those representations that are parabolic inductions of unramified characters, and compute directly that the images vanish for generic characters (reproving a result of Hamann) and are nontrivial for non-generic characters. The gluing data for non-generic characters along Harder-Narasimhan strata can be seen as an automorphic incarnation of monodromy of the associated Fargues-Scholze parameter in the moduli of $L$-parameters. 
\begin{proposition}[Proposition \ref{non-LS-type}]
    Let $\chi=\chi_1 \otimes \chi_2$ be an unramified character of $T\subset B$. Then $H_c^i(\widetilde{\mc M_{b_2}^{b_1}}, \ind_B^{\GL_2(E)}\chi\cdot \delta_T^{-1/2})$ vanishes for all $i$ unless $\chi_1 \cdot \chi_2^\inv \cong 1$ or $\abs{-}^{\pm 1}$ as characters $E^\times \to \Lambda$.
    
    Furthermore, if $\chi = \abs{-}_1^a \otimes \abs{-}_2^b: (z_1,z_2)\mapsto \abs{z_1}^a \cdot \abs{z_2}^b$ with $a,b\in \frac{1}{2}\Z$, then \newline
    for $a+1=b$,
    \[H_c^i(\widetilde{\mc M_{b_2}^{b_1}},\ind_B^{\GL_2(E)}(\abs{-}^a\otimes \abs{-}^{a+1})\cdot \delta_T^{-1/2}) =\begin{cases} 
    \abs{-}^{a+1/2} & i=1,2\\
    0,\quad & \text{else,}
    \end{cases}
    \]
    for $a=b$,
    \[H_c^i(\widetilde{\mc M_{b_2}^{b_1}},\ind_B^{\GL_2(E)}(\abs{-}^a\otimes \abs{-}^a)\cdot \delta_T^{-1/2}) =\begin{cases} 
    \abs{-}^a\delta_T^{1/2} & i=2,3\\
    0,\quad & \text{else,}
    \end{cases}
    \]
    and for $a-1=b$,
    \[H_c^i(\widetilde{\mc M_{b_2}^{b_1}},\ind_B^{\GL_2(E)}(\abs{-}^{a}\otimes \abs{-}^{a-1})\cdot \delta_T^{-1/2}) =\begin{cases} 
    	\abs{-}^{a-1/2}\delta_T & i=3,4\\
    	0,\quad & \text{else,}
    \end{cases}
    \]
    and the cohomology vanishes for any other pair $(a,b)$.
\end{proposition}
Finally, Verdier duality yields the following cohomologies of the images of the gluing functors. 
\begin{corollary}[Corollary \ref{gluing-functors-for-unram-principal-series}]
    Let $m,n\in \frac{1}{2}\Z$. Then 
    $i_{b_2}^* Ri_{b_1,*} \ind_B^{\GL_2(E)}\left(\abs{\det}^m\otimes \delta_T^{n}\otimes \delta_T^{-1/2}\right)$ vanishes unless $n\in \set{-1/2,0,1/2},$ where we have
    \begin{align*}
    H^k(i_{b_2}^* Ri_{b_1,*} \ind_B^{\GL_2(E)}\left(\abs{\det}^m\otimes \delta_T^{-1/2}\otimes \delta_T^{-1/2}\right)) &=\begin{cases} 
    \abs{-}^m\delta_T & k=2,3\\
    0,\quad & \text{else,}
    \end{cases}\\
    H^k(i_{b_2}^* Ri_{b_1,*} \ind_B^{\GL_2(E)}\left(\abs{\det}^m \otimes \delta_T^{-1/2}\right)) &=\begin{cases} 
    \abs{-}^m\delta_T^{1/2} & k=1,2\\
    0,\quad & \text{else,}
    \end{cases}\\
    H^k(i_{b_2}^* Ri_{b_1,*} \ind_B^{\GL_2(E)}\left(\abs{\det}^m\otimes \delta_T^{1/2}\otimes \delta_T^{-1/2}\right)) &=\begin{cases} 
    \abs{-}^m & k=0,1\\
    0,\quad & \text{else.}
    \end{cases}
    \end{align*}
\end{corollary}

\subsection{Outline of the sections} The rest of the paper is organized as follows.

§2: we recall the basic properties of the formalism of vector bundles on the Fargues-Fontaine curve, and the associated Banach-Colmez space to a vector bundle with everywhere positive slopes. We also summarize the description of the geometry of the moduli stack of rank $2$ vector bundles $\Bun_2$ in \cite{fargues2021geometrization}, including the smooth charts $\mc M_b \to \Bun_2$ defined by moduli of filtered rank $2$ vector bundles on the curve. Then we remind the reader of two semi-orthogonal decompositions of $D_\et(\Bun_2,\Lambda)$ arising from Harder-Narasimhan strata of $\Bun_2$. 

§3: we describe how to compute the gluing functors for one of the semi-orthogonal decompositions in terms of the cohomology of $\widetilde{\mc M_{b_2}^{b_1}}$, which is a $\underline{G_{b_2}(E)}$-torsor over a locally closed stratum of $\mc M_{b_2}$. In the final part we shift to the compactly supported cohomology of a further trivialization $\widetilde{\widetilde{\mc M_{b_2}^{b_1}}}$ uniformizing the $\mc G_{b_1}$-action, which is much easier to compute in practice, and we explain how compactly supported cohomology descends along torsors by a homological version of Hochschild-Serre spectral sequence (Proposition \ref{homological-Hochschild-Serre}). We also explain how the compactly supported cohomology of these spaces describe the semi-orthogonal decomposition when restricted to dualizable sheaves. Note that the results of this section hold for any reductive group $G/E$.

§4: we compute some examples of the functor $i^*Rj_*$ for $\GL_2/E$. First we collect some lemmas regarding cohomology of Banach-Colmez spaces. The majority of this section is concerned with the immediate specialization in integral slopes, restricted to irreducible smooth $\GL_2(E)$-modules. We also do the same in half-integral slope.

Appendix: we collect some theorems on the smooth (co)homology of $\SL_2(E)$-, $E^\times$-, and $\GL_2(E)$-modules. We provide a brief summary of the smooth projective $\SL_2(E)$-resolution of the trivial representation constructed by Schneider and Stuhler in \cite{StuhlerSchneider+1993+19+32} via sheaves on the Bruhat-Tits building, which we use to compute the $\SL_2(E)$-cohomology of the trivial and Steinberg representations (Example \ref{sl2-cohomology-steinberg}). While these computations are well-known, it is important to reproduce them in order to explicitly understand the $G_{b_2}(E)$-module structure induced on these (co)homology groups. 

\subsection{Notation and conventions}
We adopt the notation used in \cite{fargues2021geometrization}, with very minor exceptions.
\begin{itemize}
    \item $E$ is a nonarchimedean local field with integers $\mc O_E$, uniformizer $\pi$, and residue field $\F_q=\mc O_E/\pi$ of characteristic $p$.
    \item $G$ is a reductive group over $E$.
    \item $k$ is a discrete field containing $\F_q$. Typically, $k=\overline{\F_q}$.
    \item $*$ is the $v$-sheaf $\spd k$, which is the base space of all $v$-stacks we consider. In particular our spaces admit an arithmetic $q$-Frobenius automorphism.
    \item $C$ is an algebraically closed complete non-archimedean field containing $k$. In particular it is perfectoid with pseudo-uniformizer $\varpi$.
    \item We write $\D$ for the perfectoid open unit disk over $C$.
    \item $\Lambda$ is a self-injective ring killed by an integer $n$ which is prime to $p$. There is also sometimes a standing assumption that the torsion of $\Lambda$ is ``sufficiently large'' (e.g. $\ell$ is a banal prime for $\GL_2$, so that $\ell$ does not divide the pro-order of any compact open subgroup); more precisely, the primes dividing $n$ are distinct from those dividing $(q-1)^2(q+1)$, so that the modular cohomology for certain quotients by arithmetic subgroups is trivial (this is further explained in the Appendix). The cases of primary interest are $\Lambda = \overline{\F_\ell}$ and $\Lambda = \Z/\ell^i\Z$.
    \item For any locally profinite group $G$, we write $D(G)$ for the derived category of smooth representations of $G$ in $\Lambda$-modules.
    \item $B$ is the group of $E$-rational points of the standard Borel subgroup of $\GL_2(E)$ or $\SL_2(E)$ depending on context.
    \item $\ind_B^G \sigma$ is the non-smooth and non-normalized induction (i.e. $\sigma$ is not first twisted by the character $\delta_B^{-1/2}$ as in the principal series representations).
    \item $\delta_T: E^\times \times E^\times \to \Lambda$ is the character $(e_1,e_2)\mapsto \abs{e_2/e_1}$, where $\abs{e_2/e_1}\in q^{\Z}\subset \Q$ is viewed as an element of $\Lambda$ by considering the image of $1$ under the multiplication-by-$\abs{e_2/e_1}$ map, which is well-defined since $q\in \Lambda^\times$. We warn that this is the convention used in \cite{Bushnell2006}, and the opposite sign can be found in many places in the literature. We often write $\delta_T$ even for the inflated character to $B$ in order to emphasize this discrepancy.
    \item If $S$ is a compact Hausdorff space, then $\underline{S}$ is the diamond taking $X\in \Perf$ to the continuous functions
    \[\underline{S}(X):= C^0(\abs{X},S).\]
    More generally, if $S$ is locally compact Hausdorff then the same functor only produces a $v$-sheaf.
\end{itemize}

It will often be the case that we implicitly base change along $\spa C\to *=\spd k$, which induces equivalences on \'etale sheaves and their cohomology if $k$ is algebraically closed (and in particular, equipped with its $q$-Frobenius automorphism). The symbol $\Lambda$ is used interchangeably as the constant sheaf, the abstract ring, the trivial representation, or a complex supported in degree $0$ depending on context. Note that since all of the groups we work with are locally pro-$p$ and since $\Lambda$ is discrete, self-injective, and $p^\inv \in \Lambda$, we have in our setup exactness for the functors $(-)^{\text{sm}}:(\rho,V)\mapsto (\rho, V^{\text{sm}}:=\varinjlim_K V^K)$ and $(-)^\vee:(\rho,V)\mapsto (\rho^{\vee},V^\vee:=\Hom_{G}(V,\Lambda)^{\text{sm}})$. Then since taking smooth vectors is exact and preserves projectives, it does not affect computations of cohomology, so we will also be quite inexplicit about taking smooth subrepresentations at some points. We also omit a lot of details of the perfectoid and $v$-geometry, such as remarking when maps are (partially) compactifiable, of finite geometric transcendence degree, representable in locally spatial diamonds, etc. We fix a Haar measure $\mu_{\GL_2(E)}$ so that the Hecke algebra is identified with the algebra of compactly supported locally constant functions on $\GL_2(E)$. We write $R i_{b,?}: D_\et(\Bun_G^b)\to D_\et(\Bun_G)$ for $?\in \set{\sharp,!,*}$ instead of $i_{b,?}$ in contrast to other places in the literature since it is convenient for denoting the degree of cohomology groups.

\subsection*{Acknowledgements}
I would like to thank Peter Scholze for suggesting this topic, namely the computation for unramified characters of $\GL_2(E)$, as my master's thesis in 2022/23 and for his helpful remarks, suggestions, and corrections. Thanks also to Johannes Anschütz, Ali Barkhordarian, Lucas Gerth, Linus Hamann, David Hansen, Ben Heuer, Bence Hevesi, Thibaud van den Hove, Katharina Hübner, Jordan Levin, Timo Richarz, Ruth Wild, Mingjia Zhang, and Konrad Zou for helpful conversations, their interest in this project, and their support.

During the completion of this article, the author was partially supported by the Deutsche Forschungsgemeinschaft (DFG, German Research Foundation) through
the Collaborative Research Centre TRR 326 GAUS, project number 444845124.

%% file: VB_on_curve.tex
In this section we recall many relevant definitions and theorems involving vector bundles on the Fargues-Fontaine curve. We give special attention to the moduli of rank $2$ bundles for use in section 4. This section may be skipped and is only included for convenience of the reader.
\subsection{The Fargues-Fontaine curve and Banach-Colmez spaces}
We briefly recall the formalism of vector bundles over the curve, starting with vector bundles over absolute curves $X_{C,E}$ for geometric points $\spa C$. 
\begin{definition}
For any $S=\spa(R,R^+)\in \Perf$ an affinoid perfectoid space over $\F_q$ with pseudo-uniformizer $\varpi\in R^+$, consider the locus $Y_{S,E}\subset \spa W_{\mc O_E}(R^+)$ where $\pi$ and $[\varpi]$ are invertible. Note that $W_{\mc O_E}(R^+):= W(R^+)\otimes_{W(\F_q)}\mc O_E$ are the ramified Witt vectors, which inherit a Frobenius automorphism via the first factor. Then $Y_{S,E}$ is something like the ``generic fiber of $S\times_{\F_q} \mc O_E$ over $E$'', and it inherits this Frobenius automorphism, denoted $\varphi$ (note that such an analogy is more accurate after tilting; $Y^\diamond \cong S\times \spd E$, and the Frobenius is sent to the partial Frobenius automorphism coming from $S$ \cite[Proposition~II.1.17]{fargues2021geometrization}).
The \emph{Fargues-Fontaine curve for $E$ and $S$} is the adic space over $\spa E$ defined by 
\[X_{S,E}:=Y_{S,E}/\varphi^{\Z}.\]
We will often omit $E$ from the notation, and we let $X:= X_{C,E}$ denote the absolute curve over a fixed geometric point $\spa(C,\mc O_C)$. Note that $X_{S,E}$ does not map to $S$, even in equal characteristic, but nonetheless it serves as a fundamental object in understanding the rational $p$-adic Hodge theory of families parameterized by $S$.
\end{definition}

The key geometric object at play is the moduli stack of vector bundles on the curve. While the theory of quasi-coherent sheaves on adic spaces is necessarily condensed, there is a very classical theory of vector bundles on $X$. Furthermore, vector bundles on $X_{S,E}$ satisfy descent for the $v$-topology on $S$, and hence can be thought of as $p$-adic analytic families over $S$. Observe that $W_{\mc O_E}(\overline{\F_q})[1/\pi]= \Breve{E}$ is the $\pi$-adic completion of the maximal unramified extension of $E$, together with a (Witt vector) Frobenius automorphism $\phi_{\Breve{E}}$. Then one might expect via descent on $\overline{\F_q}\subset C$ and $Y\to X$ that a vector bundle on $X_{C,E}$ is given by an $\Breve{E}$-vector space together with a Frobenius cocycle.

\begin{definition}
An \textit{$E$-isocrystal} is a pair $(V,\phi)$ where $V$ is a finite dimensional $\Breve{E}$-vector space and $\phi:V\xrightarrow{\sim}V$ is a $\phi_{\Breve{E}}$-linear automorphism.
\end{definition}
According to the Dieudonn\'{e}-Manin classification, the category of $E$-isocrystals is semisimple with simple objects $V_\lambda$ corresponding to $\lambda=\frac{d}{r} \in \Q$ (with $d$ and $r$ coprime), defined by 
\[V_\lambda= \Breve{E}^r,\ \phi= \begin{bmatrix}0 & 1 & 0 &  \cdots &\\ & 0 &1 & 0 &\cdots \\ & & \ddots & \ddots & \ddots \\
0 & & & 0 & 1\\ \pi^d & 0 & & & 0\end{bmatrix}.\] Then there is a slope formalism, where orthogonality of the $V_\lambda$ induces unique (split) Harder-Narasimhan filtrations on isocrystals. There is a functor $\isoc_{E}\to \mathrm{VB}(X_{C,E})$ extending $\mc O(\lambda):=\mc E(V_{-\lambda})$ obtained by descending $V_{-\lambda}$ from $Y_{C,E}\to \spa \Breve{E}$ to $X_{C,E}$ (the minus sign occurs so that $\mc O(1)$ is ample). It was shown by  Fargues-Fontaine, Hartl-Pink, and Kedlaya that this functor induces a bijection on isomorphism classes. Furthermore, this identification yields split Harder-Narasimhan filtrations on vector bundles over absolute curves $X_{C,E}$. Originally proven by Kedlaya-Liu, there is a notion of global Harder-Narasimhan (exhausted, separated, and decreasing) filtration on vector bundles over $X_{S,E}$ that exist if the Harder-Narasimhan polygon is constant for every point of $S$ and are split after a pro-\'etale cover \cite[Theorem~II.2.19]{fargues2021geometrization}. 

\begin{remark}
	Note that the splitting of Harder-Narasimhan filtrations for vector bundles over absolute curves happens classically (i.e. for smooth projective curves over $\mathbb{C}$) only for genus $g=0,1$. In varying ways, $X_{C,E}$ shares properties with genus $0$ or genus $1$ curves, and hence can be thought of as having ``genus $1/2$''.
\end{remark}

The following class of spaces are fundamental in constructing representable covers of points in the moduli of vector bundles.

\begin{definition}[{\cite[\hphantom~I.3.5]{fargues2021geometrization}}]
Let $\mc E$ be a vector bundle on $X_S$ with positive slopes at every point of $S$. The \emph{(positive) Banach-Colmez space} $\BC(\mc E)$ associated with $\mc E$ is the $v$-sheaf over $S$ whose $T$-valued points, for $T\in \Perf_S$, are given by \[\BC(\mc E)(T)=H^0(X_T, \mc E|_{X_T}).\]
\end{definition}
\begin{remark}
$\BC(\mc E)$ is in general only a $v$-sheaf if the base $S$ is not perfectoid (e.g. when defined over $\spd k$), but is a locally spatial diamond as soon as the base is perfectoid.
\end{remark}
The long exact cohomology sequences associated to inclusions of vector bundles on $X_S$ provide a very useful tool for understanding Banach-Colmez spaces. For example, the long exact sequence associated to a degree $1$ modification
\[0\to \mc O\rightarrow \mc O(1)\rightarrow i_* \mc O_{S^\sharp}\to 0\]
induces
\[0\to \underline E \to \BC(\mc O(1)) \to (\mathbb{A}^1_{S^\sharp})^\diamond \to 0,\]
which can be thought of as a highly nonsplit extension of an (equal characteristic) affine line over $S$ by the locally archimedean group $E$ (this intuition only fails since $\mathbb{A}^1_{S^\sharp}$ is not perfectoid). In many ways, $\BC(\mc O(n))$ for positive $n$ is the analogue of affine $n$-space in classical geometric representation theory, and the projectivized Banach-Colmez spaces $\underline{E^\times}\backslash [\BC(\mc E) \setminus 0]$ are analogues of projective spaces, which can be seen explicitly in later computations. For example, the projectivized Banach-Colmez space associated to $\BC(\mc O^2)$ is $\underline{\Proj^1(E)}$. Another example is the space of closed effective degree $1$ Cartier divisors on the curve, $\mathrm{Div}^1 \cong \underline{E^\times}\backslash [\BC(\mc O(1)) \setminus 0]$, where the isomorphism is given by an Abel-Jacobi map in the geometric class field theory of Fargues \cite[Corollary~II.2.4]{fargues2021geometrization}.
\begin{proposition}[\cite{fargues2021geometrization} II.2.5(iv)]\label{smalldegBC}
If $0<\lambda\leq [E:\Q_p]$ (resp. for all positive $\lambda$ if $E$ is of equal characteristic), there is an isomorphism
\[\BC (\mc O(\lambda))\cong \spd k\powerseries{x_1^{1/p^\infty},\dots, x_d^{1/p^\infty}}\]
where $\lambda=d/r$ with coprime integers $d,r>0$.
\end{proposition}

\begin{remark}\label{basechange-to-C}\hfill
\begin{itemize}
	\item In many of our computations, $\BC(\mc O(1))$ will be base changed to $\spa C$. Then this space is representable as the perfectoid open unit disk over $C$. More generally ($\lambda$ as above), there is an isomorphism with the perfectoid open disk $\BC(\mc O(\lambda))\cong \D^r_C$ with an action of $\pi\in \underline E^\times$ given by $d$-many iterations of the inverse of geometric Frobenius.
	\item The isomorphism in equal characteristic is explicit: a $\spa (R,R^+)$-point of $\BC(\mc O(1))$ corresponds to a power series $\sum_{n\in \Z} a_n \cdot t^n$, determined by a single term due to the Frobenius equivariance condition (i.e. $a_{n+1}^q=a_n$). That is, a power series $\sum_{n\in \Z} a^{1/q^n}t^n$, which converges if and only if $a\in R^{\circ\circ}$. Then since $\pi:= t$, multiplication by $\pi$ acts as taking the $q^\text{th}$-root on $a$, so the $\underline{E^\times}$-action is carried over to the perfectoid open disk as $\pi=$ inverse geometric $q$-Frobenius. This will induce the same action as an arithmetic Frobenius on \'etale cohomology. 
	\item In the case of $p$-adic $E$, such an isomorphism follows from the Lubin-Tate theory of $E$, identifying $\BC(\mc O(1))$ as (the diamond associated to) the adic generic fiber of the universal cover of the moduli space of $p$-divisible groups of height $1$ and dimension $1$, characterized by their (covariant) Dieudonn\'e modules. The induced $\underline{E^\times}$-action is also $\pi$ acting by inverse geometric Frobenius. See \cite[\hphantom~II.2.1]{fargues2021geometrization} for more details.
\end{itemize}
\end{remark}

Regardless of representability, the cohomology of positive Banach-Colmez spaces was computed in \cite{Hansen_shtukas}, where it was shown that they are ``$\ell$-cohomological disks''.

\begin{proposition}[{\cite[\hphantom~4.7]{Hansen_shtukas}}]\label{Hansen-bc-dualizing}
Let $\lambda>0$ and $f: \BC(\mc O(\lambda))\to *$ be the structure morphism. Then after base-change to $\spa C$, the natural map
\[\Lambda \to R f_{C,*}\Lambda\]
is an isomorphism. Furthermore, $f:\BC(\mc O(\lambda))\to *$ is cohomologically smooth with dualizing sheaf $Rf^!\Lambda \cong \Lambda(d)[2d]$, where $d$ is the degree of $\mc O(\lambda)$.
\end{proposition}
A similar characterization of the cohomology of negative Banach-Colmez spaces was computed in \cite{hamann2023geometric}; see Lemma \ref{negative_bc_dualizing_sheaf}.

\subsection{Geometry of \texorpdfstring{$\Bun_2$}{Bun2}}
This subsection recaps the geometry of the moduli of rank $2$ bundles on the curve and gives some examples relevant to our main computations.

Let $\Bun_2$ be the prestack that assigns to each $S\in \Perf_k$ the groupoid of rank $2$ vector bundles on $X_{S,E}$.
Let $B(\GL_2)$ denote the set of Newton polygons with length $2$. By the Harder-Narasimhan formalism from the previous section, this set is equivalent to the isomorphism classes of $2$-dimensional isocrystals. 
\\
Then since the Harder-Narasimhan filtrations for geometric points $\spa C\to \Bun_2$ are split, the points of the topological space of $\Bun_2$ can be identified with elements of $B(\GL_2)$. The next theorem makes this more precise, and further determines the topology of $\abs{\Bun_2}$.
\begin{theorem}[{\cite[\hphantom~I.4.1]{fargues2021geometrization}}]
The prestack $\Bun_2$ satisfies the following properties.
\begin{enumerate}
    \item The prestack $\Bun_2$ is a stack for the $v$-topology.
    \item The points $\abs{\Bun_2}$ are naturally in bijection with Kottwitz' set $B(\GL_2)$ of $2$-dimensional $E$-isocrystals. 
    \item The Newton map \[\nu: \abs{\Bun_2}\to B(\GL_2)\to (X_*(T)_{\Q}^+)^\Gamma = X_*(T)_{\Q}^+\]
    is semicontinuous, and the Kottwitz map \[\kappa:\abs{\Bun_2}\to B(\GL_2)\to \pi_1(\GL_{2,{\overline{E}}})_{\Gamma}\cong \Z\]
    is locally constant ($\Gamma:=\Gal(\overline{E}/E)$, and the Galois action on coroots is trivial since $\GL_2(E)$ is split). Furthemore, the map $\abs{\Bun_2}\to B(\GL_2)$ is a homeomorphism when $B(\GL_2)$ is equipped with the order topology \cite[Theorem~1.1]{Hansen_degenerating_VB}.
    \item For any $b\in B(\GL_2)$, the corresponding subfunctor 
    \[i_b:\Bun_2^b:=\Bun_2\times_{\abs{\Bun_2}}\set{b}\subset \Bun_2\]
    is locally closed, and isomorphic to $[*/\mc G_b]$, where $\mc G_b$ is a $v$-sheaf of groups such that $\mc G_b\to *$ is representable in locally spatial diamonds with $\pi_0 \mc G_b=G_b(E)$. The connected component $\mc G_b^\circ \subset \mc G_b$ of the identity is cohomologically smooth of dimension $\langle 2\rho, \nu_b\rangle$.
    \item In particular, the semistable locus $\Bun_2^\text{ss}\subset \Bun_2$ is open, and given by \[\Bun_2^\text{ss}\cong \bigsqcup_{b\in B(\GL_2)_{\text{basic}}}[*/\underline{G_b(E)}].\]
    \item The $v$-stack $\Bun_2$ is a cohomologically smooth Artin stack of dimension $0$.
\end{enumerate}
\end{theorem}

In the case of $\GL_2$, the Newton resp. Kottwitz maps on $B(G)$ (the latter of which can be very inexplicit) are given concretely as the slopes of the Harder-Narasimhan filtration resp. the degree of the isocrystal $b\in B(\GL_2)$. Putting everything together, we get the following picture for $\abs{\Bun_2}$. The connected components are arranged vertically by slope with the semistable points sitting in dimension $0$, which specialize to points with the same slope sitting in codimension $\langle 2\rho, \nu_b\rangle$. Note that the connected components are totally ordered under specialization, but this is no longer true even for $\GL_3$, which has more complicated geometry in higher codimension (for example, $\mc O^3$ specializes to $\mc O(-1/2) \oplus \mc O(1)$ and $\mc O(-1) \oplus \mc O(1/2)$, which are not specializations of each other).

\begin{center}
\begin{scriptsize}
\begin{tabular}{|c|}
     \hline \\
     \hfill \\
     \large{$\abs{\Bun_2}$} \\
\begin{tikzcd}
\vdots&\vdots  & & & & \\
2&\mc O(1)^2&  &\mc O\oplus \mc O(2) & &\cdots\\
1&\mc O(1/2) & \mc O\oplus \mc O(1)& &\mc O(-1)\oplus \mc O(2) &\cdots\\
\kappa(b)=0& \mc O^2 &  &\mc O(-1)\oplus \mc O(1) & &\cdots \\
-1&\mc O(-1/2) & \mc O(-1)\oplus \mc O & &\mc O(-2)\oplus \mc O(1) &\cdots\\
\vdots &\vdots & & & &\\
& \dim = 0 & -1 & -2 & -3 & \cdots\\
\end{tikzcd}\\
     \hline
\end{tabular}
\end{scriptsize}
\end{center}

Let $\mc G_b$ denote the stabilizer of a point $b\in \abs{\Bun_2}$. Recall \cite[Proposition~III.5.1]{fargues2021geometrization} that $\mc G_b$ is a split extension of $\underline{G_b(E)}$ by a ``unipotent group diamond" consisting of automorphisms $\gamma:\mc E_b \to \mc E_b$ such that $\gamma -1$ takes sections to a subsheaf of strictly larger slope under the Harder-Narasimhan filtration; 
\[0 \to \mc G_b^{>0}\to \mc G_b\to \underline{G_b(E)}\to 0,\]
where $\mc G_b^{>0}$ is a successive extension of positive Banach-Colmez spaces acting by unipotent transformations. The groups $\underline{G_b(E)}$ are analogous to the Levi quotient of the parabolic reduction of $\mc E_b$ in the geometric Langlands program.

\begin{example}\hfill 
\begin{itemize} 
    \item For basic $b$, already $\mc G_{b}\cong \underline{G_{b}(E)}$ since the corresponding vector bundle $\mc E_b$ is semistable, and hence the Harder-Narasimhan filtration has constant slope. Then $G_b(E)$ is either $\GL_2(E)$ for integral slope basic points or $D_{1/2}^\times$ (units of the unique quaternion algebra over $E$ with invariant $1/2$) for half-integral slope. 
    \item For $b \leftrightarrow \mc O(-1) \oplus \mc O(1)$, $\mc G_b$ is the automorphism group diamond \[\Aut(\mc O(-1)\oplus \mc O(1))\cong \begin{bmatrix} \underline{E^\times} & \BC(\mc O(2))\\ 0 & \underline{E^\times} \end{bmatrix},\]
    so that $\mc G_b^{>0}\cong \BC(\mc O(2))$ are the automorphisms taking the filtration step with slope $-1$ quotient to the filtration step with slope $1$ quotient. Geometrically, this is why the stacky point $[*/\mc G_b]\hookrightarrow \Bun_2$ has dimension $-2$.
    \item As above, $G_b(E)\cong E^\times \times E^\times$ for any non-basic point $b\in B(\GL_2)$. In particular, they are $E$-split tori even in half-integral slope!
\end{itemize}
\end{example}
\begin{proposition}[{\cite[\hphantom~III.5.3]{fargues2021geometrization}}]
Let $b\in B(\GL_2)$ be any element given by a $2$-dimensional isocrystal. The induced map $x_b: *\to \Bun_2^b$ is a surjective map of $v$-stacks, and $*\times_{\Bun_2^b}*\cong \mc G_{b}$, so that 
\[\Bun_2^b\cong [*/\mc G_b]\]
is the classifying stack of $\mc G_b$-torsors. In particular, the map $\mc G_b \to \pi_0 \mc G_b \cong \underline{G_b(E)}$ induces a map \[\Bun_2^b\to [*/\underline{G_b(E)}]\] that admits a splitting. 
\end{proposition}

Finally, we recall the smooth covers constructed in \cite{fargues2021geometrization} that allow us to compute cohomology of sheaves on $\Bun_2$. These covers are defined as moduli of filtered vector bundles, and have much simpler geometry than the covers defined by the affine Schubert cells.
\begin{definition}[{\cite[\hphantom~V.3.2]{fargues2021geometrization}}]
The $v$-stack $\mc M$ is the moduli stack taking $S\in \Perf_k$ to the groupoid of rank $2$ vector bundles $\mc E$ on $X_S$ together with an increasing separated and exhaustive (weight) $\Q$-filtration $\mc E^{\leq\lambda}\subset \mc E$, such that (letting $\mc E^{<\lambda}= \bigcup_{\lambda'<\lambda}\mc E^{\leq \lambda'}$), the quotient 
\[\mc E^\lambda = \mc E^{\leq \lambda}/ \mc E^{<\lambda}\]
is a semistable vector bundle of slope $\lambda$, for all $\lambda\in \Q$. 
\end{definition}
Note that the slopes of the associated graded pieces are increasing and hence opposite to the Harder-Narasimhan filtration on $\mc E$; in particular, the slopes in the weight filtration need not agree with the slopes in the Harder-Narasimhan filtration. Our primary example will be analyzing filtrations of the form \[0\to \mc O(-n)\to \mc O^2 \to \mc O(n) \to 0.\]

\begin{remark}[\cite{fargues2021geometrization} V.3]
Passing to the associated graded of the $\Q$-filtration yields a map to the stack of rank 2 vector bundles with split global Harder-Narasimhan filtration, $\Bun_2^{\text{HN-split}}\cong \bigsqcup_{b\in B(G)}[*/\underline{G_{b}(E)}]$ \cite[Proposition~III.4.7]{fargues2021geometrization}. This gives a decomposition $\mc M = \bigsqcup_{b\in B(G)} \mc M_b$. 
\end{remark}

The geometry of $\mc M_b$ is discussed at length in \cite[Chapter~V.3]{fargues2021geometrization}.
To summarize:
\begin{theorem}[{\cite[\hphantom~V.3.5,~V.3.7]{fargues2021geometrization}}]\label{properties-of-M_b}
For any $b\in B(\GL_2)$, the map $f_b:\mc M_b\to \Bun_2$ forgetting the filtration is representable in locally spatial diamonds, partially proper, and cohomologically smooth. There is a factorization of $f_b$ through $\mc M_b\to [*/\underline{G_b(E)}]$, which is representable in locally spatial diamonds, partially proper, and cohomologically smooth of dimension $\langle 2\rho, \nu_b \rangle$.
\end{theorem}

\subsection{Semi-orthogonal decomposition and excision triangles}
We recollect the abstract notion of semi-orthogonal decomposition, along with two infinite semi-orthogonal decompositions of $D_\et(\Bun_G,\Lambda)$ constructed in \cite{fargues2021geometrization}. We also list some properties of sheaves on $\Bun_G$ that correspond to properties of representations on each stratum.

Suppose \mc D is a triangulated category. A full triangulated subcategory $\mc D_1 \subset \mc D$ is \textit{admissible} if the inclusion $Ri_*: \mc D_1 \hookrightarrow \mc D$ admits left and right adjoints $i^* \dashv Ri_* \dashv i^!$. (Our notation is already suggestive of the geometric application to follow.)
Suppose $\mc D_1,\dots, \mc D_m \subseteq \mc D$ are admissible subcategories satisfying $\Hom(\mc D_j,\mc D_i)=0$ whenever $j>i$, and for every $E\in \mc D$, there exists a filtration $0=E_0\to E_1\to \dots \to E_m = E$
and exact triangles $E_{i-1}\to E_i\to A_i$ with $A_i\in \mc D_i$ for each $i\in \set{1,\dots,m}$. Then we say \mc D admits a \textit{semi-orthogonal decomposition} by the $\mc D_i$, and write $\mc D = \langle \ls{\mc D}{m} \rangle.$

For a semi-orthogonal decomposition $\mc D= \langle \mc D_1,\dots, \mc D_m\rangle$, write $Ri_*:\mc D_1\to \mc D$ and $Rj_!:\ ^{\perp}\mc D_1\to \mc D$ for the inclusions, where $^\perp \mc D_1\subset \mc D$ is the full triangulated subcategory \[\set{A\in \mc D\ |\ \Hom(A,Ri_* B)=0\text{ for all }B\in \mc D_1}.\] 
In particular, each $\mc D_2,\dots, \mc D_m\subset {}^{\perp}\mc D_1$ and no object of $\mc D_1$ is in ${}^{\perp}\mc D_1$.
\begin{lemma}
	Suppose $Rj_!$ admits a right adjoint $j^*$. Then there are functorial excision exact triangles \[Rj_! j^* A\to A \to Ri_*i^* A\to Rj_! j^* A[1].\]
\end{lemma}

\begin{proof}
	The map $A\to Ri_* i^*A$ is the unit of the adjunction, and fits in some triangle $K_A \to A\to R i_* i^*A\to K_A[1]$. For any $B\in \mc D_1$, applying $\Hom(-,Ri_*B)$ induces maps 
	\[\Hom(K_A[1],Ri_* B)\to \Hom(Ri_*i^*A, Ri_*B)\to \Hom(A,Ri_*B),\]
	and $Ri_*$ is fully faithful, so that the map on the right identifies with the adjunction isomorphism \[\Hom(i^*A,B)\xrightarrow{\sim}\Hom(A,Ri_*B).\]
	More explicitly: The composition $A\to Ri_*i^*A \to Ri_*B$ corresponds to some map $i^*A\to B$ under the adjunction, and the map $Ri_*i^*A\to Ri_*B$ is the image of some map $i^*A\to B$; then these are necessarily the same map since an adjunction induces \textit{natural} bijections on Hom-sets.\\
	In particular, this shows that $K_A\in {}^\perp \mc D_1$. Then one checks that the assignment $A\mapsto K_A: \mc D\to {}^\perp \mc D_1$ is indeed a right adjoint to $Rj_!$.
\end{proof}

Turning this procedure around, one can obtain all objects in $\mc D$ by completing triangles for morphisms $Ri_* B' \to Rj_! A'[1]$ with $A'\in \mc {}^\perp \mc D_1$ and $B'\in \mc D_1$. By adjunction, this is the same as a morphism $B' \to i^!Rj_! A'[1]$, so that we are interested in analyzing the functors $i^!Rj_!$. An entirely dual situation to that above produces functorial triangles
\[Ri_* i^! A \to A \to Rj_*j^* A \to Ri_* i^! A [1]\]
for the inclusion $Rj_*: \mc D_1^\perp\hookrightarrow \mc D$, whenever it admits a left adjoint $j^* \dashv Rj_*$. For this situation, the gluing functors $i^* Rj_*: \mc D_1^\perp \to \mc D_1$ are of interest. Note however that the functors $i^*Rj_*$ and $i^! Rj_!$ agree up to a shift.
\begin{lemma}
	There is an isomorphism $i^*Rj_*\cong i^! Rj_! [1]$.
\end{lemma}
\begin{proof}
	This lemma is well-known, and follows immediately by considering the triangle
	\[Ri_* i^! (Rj_! A)\to Rj_! A \to Rj_* j^* (Rj_! A)\to Ri_* i^! (Rj_! A)[1],\]
	and postcomposing by the exact functor $i^*$.
\end{proof}

Intuitively, a semi-orthogonal decomposition gives natural filtrations on the objects of $\mc D$ where the associated graded subquotients live in increasingly finer subcategories. The following proposition says that the graded pieces of a complex in $D_\et(\Bun_2)$ are naturally identified with a complex of smooth representations of the locally pro-$p$ group $G_b(E):= \pi_0(\mc G_b)$. 

\begin{proposition}[\cite{fargues2021geometrization} V.2.2]
	For any $b\in B(G)$, the map
	\[i_b:\Bun_2^b=[*/\mc G_b]\to [*/\underline{G_b(E)}]\]
	induces via pullback an equivalence
	\[D(G_b(E),\Lambda)\cong D_\et([*/\underline{G_b(E)}],\Lambda)\cong D_\et([*/\mc G_b],\Lambda).\]
\end{proposition}

\begin{proposition}[{\cite[\hphantom~VII.7.3]{fargues2021geometrization}}]
	For any quasicompact open substack $U \subset \Bun_2$, the category $D_\et (U,\Lambda)$ admits a semi-orthogonal decomposition into the categories $D_\et(\Bun_2^b,\Lambda)\\ \cong D(G_b(E),\Lambda)$ for $b\in \abs{U}\subset B(G)$ via excision triangles. 
\end{proposition}

In other words, the Harder-Narasimhan stratification of $\abs{\Bun_2}$ induces semi-orthogonal decompositions on the restrictions of $D_\et(\Bun_2)$ to finitely many points. The interactions of sheaves supported on different strata are given by the functors $i^*Rj_*$ where $i$ and $j$ are the inclusions of the closed point resp. open complement inside a locally closed substack of $\Bun_2$; these interactions can be computed inductively on the locally closed substacks consisting of one point specializing to the other.
This gives an infinite semi-orthogonal decomposition of $D_\et(\Bun_2)$ when considering the restrictions to each quasicompact open substack. Furthermore, there is a class of compact generators for $D_\et(\Bun_2)$ that are supported on finitely many points. 

\begin{theorem}[{\cite[\hphantom~V.4.1]{fargues2021geometrization}}]
	For any locally closed substack $U\subset \Bun_2$, the triangulated category $D_\et(U,\Lambda)$ is compactly generated. An object $A\in D_\et(U,\Lambda)$ is compact if and only if for all $b\in B(\GL_2)$ contained in $U$, the restriction 
	\[i_b^* A\in D_\et(\Bun_2^b,\Lambda)\cong D(G_b(E),\Lambda)\]
	along $i_b:\Bun_2^b \subset \Bun_2$ is compact, and zero for almost all $b$. Here, compactness in $D(G_b(E),\Lambda)$ is equivalent to lying in the thick triangulated subcategory generated by $\cind_{K}^{G_b(E)}\Lambda$ as $K$ runs over open pro-$p$-subgroups of $G_b(E)$.
\end{theorem}
\begin{remark}[{\cite[177]{fargues2021geometrization}}]
	If $f_K:\widetilde{\mc M_b}/\underline{K}\to \Bun_2$ is the induced map for each open pro-$p$ $K\subset G_b(E)$ (see Definition \ref{tilde-covering}), then the compact generators are given by $A_K^b:=R f_{K!}Rf_K^!\Lambda \in D_\et(\Bun_2,\Lambda)$ for varying $b$ and $K$. Note that the Bernstein-Zelevinsky dual of $A_K^b$ is $i_{b,!}\cind_K^{G_b(E)}\Lambda$ \cite[180]{fargues2021geometrization}.
\end{remark}
\begin{theorem}[{\cite[\hphantom~V.7.1]{fargues2021geometrization}}]
	Let $A\in D_\et(\Bun_2,\Lambda)$. Then $A$ is universally locally acyclic with respect to $\Bun_2\to *$ if and only if for all $b\in B(\GL_2)$, the pullback $i_b^*A$ to $i_b:\Bun_2^b\hookrightarrow \Bun_2$ corresponds under $D_\et(\Bun_2^b,\Lambda)\cong D(G_b(E),\Lambda)$ to a complex $M_b$ of smooth $G_b(E)$-representations for which $M_b^K$ is a perfect complex of $\Lambda$-modules for all open pro-$p$ subgroups $K\subset G_b(E)$.
\end{theorem}

The compact objects resp. ULA objects are geometric analogues of compact complexes resp. admissible $G_b(E)$-representations, and such objects satisfy Bernstein-Zelevinsky duality resp. Verdier duality \cite[Chapter~V]{fargues2021geometrization}. We are mostly interested in computing the functors $i_{b_2}^*R i_{b_1,*}$ on smooth irreducible representations, which are in particular admissible and finite length \cite[\hphantom~9.4~and~10.2]{Bushnell2006}, so that they behave well with respect to Verdier duality.

%% file: Semi-Orth.tex
In this part of the paper, we show how to reduce the functorial gluing data $i^* Rj_*$ to the cohomology of certain locally closed strata of the $\mc M_b$ covers of $\Bun_G^{\leq b}$ for any reductive group $G$. The proof proceeds in essentially three steps: a short diagram chase together with the ``strictly henselian'' property of $\widetilde{\mc M_b}$, a projective limit over smooth base change (which serves as the geometric analogue of passing to smooth vectors), and descent of compactly supported cohomology along torsors.

\subsection{Gluing functors as cohomology of negative Banach-colmez spaces}

This section explains how to reduce the computation of the functors $i_{b_2}^*Ri_{b_1,*}$ to the \'etale cohomology of a certain $v$-sheaf $\widetilde{\mc M_{b_2}^{b_1}}$, which can be thought of as a locally closed subspace of the ``strict henselization at $\Bun_G^{b_2}$''. The core idea is to consider the colimit over smooth basechange isomorphisms along a pro-étale $\underline{G_{b_2}(E)}$-torsor, which produces the smooth subrepresentation of the cohomology of $\widetilde{\mc M_{b_2}^{b_1}}$. If one rather starts with a complex on the connected quasicompact open substack $j:\Bun_G^{\leq b_2}\hookrightarrow \Bun_G$, the gluing functors reduce simply to the cohomology of $\widetilde{\mc M_{b_2}^\circ}$, which is a punctured negative Banach-Colmez space when $G=\GL_2$ (or an iterated extension of negative Banach-Colmez spaces for general $G$). 

We recall the following proposition, which allows us to make all of our computations after changing base to a complete algebraically closed nonarchimedean field $C$.

\begin{proposition}[{\cite[Corollary~V.2.3]{fargues2021geometrization}}]
For any locally closed substack $U\subset \Bun_G$, the functor 
\[D_\et(U,\Lambda)\to D_\et(U \times \spa C, \Lambda)\]
is an equivalence of categories. 
\end{proposition}

\begin{definition}\label{tilde-covering}
Define the cover $\widetilde{\mc M_{b}}$ by the pullback 
\begin{center}
\begin{tikzcd}
\widetilde{\mc M_b}\arrow[r]\arrow[d]\tikzpullback & \mc M_b\arrow[d,"\text{gr}"]\\
*\arrow[r] & {[*/\underline{G_b(E)}]}.
\end{tikzcd}
\end{center}
\end{definition}
In particular, $\widetilde{\mc M_b}\to \mc M_b$ is a $\underline{G_b(E)}$-torsor. Furthermore when $G=\GL_2$, $\widetilde{\mc M_b}$ parameterizes rank $2$ bundles with filtrations corresponding to the point $b$, together with $v$-local trivializations of the graded pieces. In the case of nonbasic $b\in \abs{\Bun_2}$, $\widetilde{\mc M_{b}}$ is a negative Banach-Colmez space that parameterizes extensions of line bundles $\mc O(n)$ by $\mc O(m)$ with $m<n$ (see \cite[Example~V.3.1]{fargues2021geometrization}). For more general $G$, $\widetilde{\mc M_b}$ is an iterated extension of negative Banach-Colmez spaces.

\begin{proposition}[{\cite[\hphantom~V.4.2]{fargues2021geometrization}}]\label{strict_hens}
Let $b\in B(G)$. For any $A\in D_\et(\widetilde{\mc M}_b,\Lambda)$ with stalk $A_0 = i^*A\in D_\et(*,\Lambda)\cong D(\Lambda)$ at the closed point $i: * \subset \widetilde{\mc M}_b$, the map 
\[R\Gamma(\widetilde{\mc M}_b,A)\to A_0\]
induced by $R\Gamma(A \to R i_* i^* A)$ is an isomorphism. In particular, $R\Gamma(\widetilde{\mc M}_b, -)$ commutes with direct sums.
\end{proposition}

In this way, $\widetilde{\mc M}_b$ behaves like the ``strict henselization'' of $\Bun_G$ at $b$. The following proposition is obtained via descent to an open pro-$p$ level subgroup $K\subset G_b(E)$.

\begin{proposition}[{\cite[Corollary~V.4.4]{fargues2021geometrization}}]\label{strict-hens-cor}
Let $b\in B(G)$ and let $K\subset G_b(E)$ be an open pro-$p$-subgroup. Then for any $A\in D_\et (\widetilde{\mc M}_b/\underline{K},\Lambda)$ with pullback $A_0=i^*A \in D_\et([*/\underline{K}],\Lambda)\cong D(K,\Lambda)$ corresponding to a complex $V$ of smooth $K$-representations, the map
\[R\Gamma(\widetilde{\mc M}_b/\underline{K}, A)\to R\Gamma([*/\underline{K}, A_0]) \cong V^K\]
is an isomorphism. In particular, $R\Gamma(\widetilde{\mc M}_b/\underline{K}, -)$ commutes with direct sums.
\end{proposition}

Now we have a diagram for any $b_1\in B(\GL_2)$ with specialization $b_2$:
\begin{equation}
\begin{tikzcd}
 &  & {[*/ \mc{G}_{b_1}]}\arrow[d, "i_{b_1}"]\\
\widetilde{\mc M_{b_2}}\arrow[r,"g"]\arrow[d, dashed, shift right =1]\tikzpullback & \mc M_{b_2} \arrow[r,"f"]\arrow[d, dashed, shift right =1]& \Bun_G \\
* \arrow[r]\arrow[u, shift right =1] & {[*/\underline{G_{b_2}(E)}]}\arrow[r] \arrow[u,shift right=1] \arrow[phantom,ur,"\urcorner",very near start] & {[*/ \mc{G}_{b_2}],} \arrow[u,"i_{b_2}", swap]
\end{tikzcd}
\end{equation}
where the bottom-right horizontal map is induced by the section $\mc G_{b_2}\leftarrow \underline{G_{b_2}(E)}$ of the projection, and the map $[*/\underline{G_{b_2}(E)}]\to \mc M_{b_2}$ takes a $\underline{G_{b_2}(E)}$-torsor to the corresponding split filtered vector bundle on $X_S$.
\begin{proposition}
Pulling back along $[*/\underline{G_{b_2}(E)}]\to [*/\mc G_{b_2}]$ induces an equivalence on $D_\et(-)$.
\end{proposition}
\begin{proof}
Fully faithfulness is shown in the proof of \cite[Proposition~V.2.1]{fargues2021geometrization}. In fact, pulling back induces the identity functor after identifying $D_\et(\mc G_{b_2})\xleftarrow{\cong}D_\et([*/\underline{G_{b_2}(E)}])$ via pullback along $[*/\mc G_{b_2}]\to [*/\underline{G_{b_2}(E)}]$. But the composition $[*/\underline{G_{b_2}(E)}]\to [*/\mc G_{b_2}]\to [*/\underline{G_{b_2}(E)}]$ is pro-\'etale locally the identity map $*\to [*/\BC(\mc L)] \to *$, which must glue to the identity.
\end{proof}
Pulling back along $* \to [*/\underline{G_{b_2}(E)}]$ induces 
\[D(G_{b_2}(E),\Lambda)\cong D_{\et}([*/\underline{G_{b_2}(E)}],\Lambda)\to D_\et(*,\Lambda)\cong D(\Lambda),\]
which sends a complex of smooth $G_{b_2}(E)$-representations to the same complex as their underlying $\Lambda$-modules \cite[Thm~V.1.1]{fargues2021geometrization}. The following lemma allows us to instead pull back through $\widetilde{\mc M_{b_2}}$. 

\begin{lemma}\label{split_point}
The point $*\to \Bun_G$ obtained via composition along the bottom row followed by $i_{b_2}$ is the same point as the composition $*\to \widetilde{\mc M_{b_2}}\xrightarrow{fg} \Bun_G$ in the diagram.
\end{lemma}
\begin{proof}
Either by a simple diagram chase, or by observing that both points correspond to the bundle $\mc E_{b_2}$ (i.e. the trivialized bundle with everywhere split Harder-Narasimhan filtration corresponding to $b_2$).
\end{proof}

Then as a consequence of Proposition \ref{strict_hens} and Lemma \ref{split_point}, we immediately get the following proposition. For convenience, if $C\in D_\et([*/\underline {G_{b}(E)}])$ is a complex of smooth $G_b(E)$-modules, we denote the pullback to $D(\Lambda)$ along $*\to [*/\underline{G_b(E)}]$ by $C_\Lambda$, which we call the underlying $\Lambda$-module.
\begin{proposition}\label{action-on-cohomology}
	For any $A\in D_\et(\Bun_G)$, there is a natural isomorphism \[(i_{b_2}^* A)_\Lambda \cong R\Gamma (\widetilde{\mc M_{b_2}}, (fg)^*A)\] in $D(\Lambda)$ which is equivariant for the natural $G_{b_2}(E)$-structure on the right. 
\end{proposition}
\begin{proof}
	The natural isomorphism as $\Lambda$-modules follows as a chain of natural isomorphisms 
	\[(i_{b_2}^*)_\Lambda\cong s^*(fg)^*\cong R\Gamma(\widetilde{\mc M_{b_2}},(fg)^*(-)),\]
	where the first is given by Lemma \ref{split_point} and the second by Proposition \ref{strict_hens}. 
	
	The equivariance is more or less tautological since both actions are defined via pullbacks along commuting arrows. To see it explictly, note that $\widetilde{\mc{M}_{b_2}}\times_{\mc M_{b_2}}\widetilde{\mc{M}_{b_2}}\to \widetilde{\mc{M}_{b_2}}$ is a trivial $\underline{G_{b_2}(E)}$-torsor, and the diagonal section induces an isomorphism $\widetilde{\mc{M}_{b_2}}\times_{\mc M_{b_2}}\widetilde{\mc{M}_{b_2}}\cong \underline{G_{b_2}(E)}\times \widetilde{\mc{M}_{b_2}}$. Any element $\gamma\in G_{b_2}(E)$ induces points on $\underline{G_{b_2}(E)}$, so that the equivariant structure on cohomology (for any sheaf $\mc F$ on $\mc M_{b_2}$) is given by 
	\[R\Gamma(\widetilde{\mc{M}_{b_2}},g^*\mc F)\to R\Gamma( \widetilde{\mc{M}_{b_2}} \times_{{\mc{M}_{b_2}}}\widetilde{\mc{M}_{b_2}},\mathrm{pr}_1^*g^*\mc F)\xrightarrow{\mathrm{can}}  R\Gamma( \widetilde{\mc{M}_{b_2}} \times_{{\mc{M}_{b_2}}}\widetilde{\mc{M}_{b_2}},\mathrm{pr}_2^*g^*\mc F)\xrightarrow{\gamma^*}R\Gamma(\widetilde{\mc{M}_{b_2}},g^*\mc F),\]
	where $\mathrm{can}: \mathrm{pr}_1^*g^*\mc F\xrightarrow{\cong}\mathrm{pr}_2^*g^* \mc F$ is the canonical descent datum, and $\gamma:\widetilde{\mc{M}_{b_2}}\hookrightarrow \underline{G_{b_2}(E)}\times \widetilde{\mc{M}_{b_2}}\cong \widetilde{\mc{M}_{b_2}}\times_{\mc M_{b_2}}\widetilde{\mc{M}_{b_2}}$ is the section determined by $\gamma$. 
	
	The section $s:*\to \widetilde{\mc{M}_{b_2}}$ induces a $\underline{G_{b_2}(E)}$-equivariant section $\tilde{s}:*\times_{[*/\underline{G_{b_2}(E)}]}*\to \widetilde{\mc{M}_{b_2}}\times_{\mc M_{b_2}}\widetilde{\mc{M}_{b_2}}$. Then one can write the corresponding row on cohomology for the canonical descent datum for $(i_{b_2}^* A)_\Lambda$ with commuting vertical arrows induced by $s$ and $\tilde{s}$; in particular, a commutative square 
	\begin{center}
	\begin{tikzcd}
		R\Gamma( \widetilde{\mc{M}_{b_2}} \times_{{\mc{M}_{b_2}}}\widetilde{\mc{M}_{b_2}},\mathrm{pr}_1^*g^*\mc F) \arrow[r,"\mathrm{can}"]\arrow[d, "\tilde{s}^*"]& R\Gamma( \widetilde{\mc{M}_{b_2}}\arrow[d, "\tilde{s}^*"] \times_{{\mc{M}_{b_2}}}\widetilde{\mc{M}_{b_2}},\mathrm{pr}_2^*g^*\mc F)\\
		R\Gamma(*\times_{[*/\underline{G_{b_2}(E)}]}*,\mathrm{pr}_1^*(\mc F)_\Lambda) \arrow[r,"\mathrm{can}"] & R\Gamma(*\times_{[*/\underline{G_{b_2}(E)}]}*,\mathrm{pr}_2^*(\mc F)_\Lambda),
	\end{tikzcd}
	\end{center}
	after identifying $\mathrm{pr}_k^*(\mc F)_\Lambda$ with $\tilde{s}^*\mathrm{pr}_k^* g^*\mc F$ for $k=1,2$.
\end{proof}
\begin{corollary}
	For any $A\in D_\et(\Bun_G^{b_1})$, there is a natural isomorphism \[(i_{b_2}^* R i_{b_1,*} A)_\Lambda \cong R\Gamma (\widetilde{\mc M_{b_2}}, (fg)^*Ri_{b_1,*}A)\] in $D(\Lambda)$ which is equivariant for the natural $G_{b_2}(E)$-structure on the right. 
\end{corollary}
Now we would like to pass to the cohomology over a locally closed stratum $\widetilde{\mc M_{b_2}^{b_1}}\hookrightarrow \widetilde{\mc{M}_{b_2}}$ by a smooth basechange morphism in the diagram below. Unfortunately, the fibers of $\widetilde{\mc M_{b_2}}\to \mc M_{b_2}$ look like $\underline{G_{b_2}(E)}$, which is not smooth due to the locally profinite nature of the group. Instead, we make use of the following intermediate spaces where smooth base change does apply, and consider the colimit of the system. 

\begin{lemma}\label{intermediate_etale_covers}
$\widetilde{\mc M_{b_2}}\to \mc M_{b_2}$ is a pro-\'{e}tale morphism of small $v$-stacks. More precisely, any open pro-$p$ subgroup $K\subset G$ yields the discrete $G$-set $G/K$, and the pushout along $\underline{G}\to \underline{G/K}$ induces separated \'{e}tale maps (in particular, representable in locally spatial diamonds)
\[\widetilde{\mc M_{b_2}}/ \underline{K}\cong \widetilde{\mc M_{b_2}}\times^{\underline{G}}\underline{G/K}\to \mc M_{b_2},\]
with finite \'{e}tale transition maps $\widetilde{\mc M_{b_2}}/\underline{K'}\to \widetilde{\mc M_{b_2}}/\underline{K}$ whenever $K'\subset K$, and
\[\widetilde{\mc M_{b_2}}\cong \varprojlim_K \widetilde{\mc M_{b_2}}/\underline{K}\to \mc M_{b_2}.\]
\end{lemma}

\begin{proof}
Follows immediately from \cite[Lemma~10.13]{scholze2017etale} after passing to a $v$-cover of $\mc M_{b_2}$ by a perfectoid space, since all claims are $v$-local on the target. 
\end{proof}

\begin{definition}
We define $\mc M_{b_2}^{b_1}$ and the $\underline{G_{b_2}(E)}$-torsor $\widetilde{\mc M_{b_2}^{b_1}}\to \mc M_{b_2}^{b_1}$ to be the pullbacks in the following diagram:

\begin{equation}
\begin{tikzcd}
\widetilde{\mc M^{b_1}_{b_2}}\arrow[d,"\widetilde{i_{b_1}}",swap] \arrow[r,"g_{b_1}^{}"] \tikzpullback & \mc M^{b_1}_{b_2}\arrow[d,swap,"i_{b_1}' "]\arrow[r,"f_{b_1}"]\tikzpullback & {[*/ \mc{G}_{b_1}]}\arrow[d, "i_{b_1}"]\\
\widetilde{\mc M_{b_2}}\arrow[r,"g"]\arrow[d, dashed, shift right =1]\tikzpullback & \mc M_{b_2} \arrow[r,"f"]\arrow[d, dashed, shift right =1]& \Bun_G \\
* \arrow[r]\arrow[u, shift right =1] & {[*/\underline{G_{b_2}(E)}]}\arrow[r] \arrow[u,shift right=1] \arrow[phantom,ur,"\urcorner",very near start] & {[*/ \mc{G}_{b_2}].} \arrow[u,"i_{b_2}", swap]
\end{tikzcd}
\end{equation}
In particular, $\mc M_{b_2}^{b_1}$ is the moduli $v$-stack parameterizing filtered vector bundles $\mc E$ over $X_S$ such that at each geometric point of $S$, $\mc E$ is isomorphic to $\mc E_{b_1}$ and the associated graded of the filtration is isomorphic to $\mc E_{b_2}$; that is to say that the weight filtration of $\mc E_{b_1}$ has pure subquotients corresponding to the semistable direct summands of $\mc E_{b_2}$. Note that the inclusion of $\Bun_G^{b_1}$ factors through the open substack $j_{<b_2}:\Bun_G^{\leq b_2}\setminus \Bun_G^{b_2}\hookrightarrow \Bun_G$, and the pullback of $\mc M_{b_2}$ along $j_{<b_2}$ recovers the punctured cover $\mc M_{b_2}^\circ$.

\end{definition}

The pullback of sheaves on $\mc M_{b_2}^{b_1}$ along $g_{b_1}^{}$ naturally inherit a $\underline{G_{b_2}(E)}$-equivariant structure, inducing $G_{b_2}(E)$-module structures on the $\Lambda$-modules of $R\Gamma(\widetilde{\mc M_{b_2}^{b_1}},g_{b_1}^*(-))$. However, this will not necessarily produce a complex of smooth representations. Then it is required to take the colimit of $K$-fixed vectors for open pro-$p$ subgroups $K\subset G_{b_2}(E)$, which extends to $D(\Lambda)$ by applying the colimit in each degree (recall that taking smooth vectors is exact since $G_{b_2}(E)$ is locally pro-$p$ and due to our assumptions on $\Lambda$, so that there is no need to derive this procedure). In other words, there is a functor \[\left(R\Gamma(\widetilde{\mc M_{b_2}^{b_1}},g_{b_1}^*(-))\right)^\mathrm{sm}:=\varinjlim_K R\Gamma(\widetilde{\mc M_{b_2}^{b_1}}/\underline{K},-):D_\et(\mc M_{b_2}^{b_1})\to D(\Lambda),\]
whose image lies in the category of smooth representations $D_\et([*/\underline{G_{b_2}(E)}])\cong D(G_{b_2}(E))$ after remembering the representation structures. Then intuitively, the following proposition states that there is a base-change isomorphism along the pro-\'etale torsor $g$, up to ``taking smooth subrepresentations''. Note that what we call smooth cohomology here is the usual definition of the cohomology of such a tower of \'{e}tale maps.

\begin{theorem}\label{Smooth-cohomology-of-tower}
There is a natural isomorphism of functors $D_\et([*/\mc G_{b_1}])\to D_{\et}(*)\cong D(\Lambda)$
\[R\Gamma(\widetilde{\mc M_{b_2}},-) \circ (fg)^*R i_{b_1,*} \xrightarrow{\sim} \left(R\Gamma(\widetilde{\mc M_{b_2}^{b_1}},g_{b_1}^*(-))\right)^\mathrm{sm}\circ f_{b_1}^*.\]
Furthermore, the induced smooth $G_{b_2}(E)$-actions on either side agree.
\end{theorem}
\begin{proof}
Lemma \ref{intermediate_etale_covers} allows us to fill in the following diagram, where the morphisms $g^K,\ g_{b_1}^K$ are \'{e}tale and hence cohomologically smooth:
\begin{equation}
\begin{tikzcd}
\widetilde{\mc M^{b_1}_{b_2}}\arrow[d,"\widetilde{i_{b_1}}",swap] \arrow[r,"\psi"] \tikzpullback& \widetilde{\mc M^{b_1}_{b_2}}/\underline{K}\arrow[r,"g_{b_1}^K"]\arrow[d,"i_{b_1}^K",swap]\tikzpullback& \mc M^{b_1}_{b_2}\arrow[d,swap,"i_{b_1}' "]\arrow[r,"f_{b_1}"]\tikzpullback & {[*/ \mc{G}_{b_1}]}\arrow[d, "i_{b_1}
"]\\
\widetilde{\mc M_{b_2}}\arrow[r]\arrow[d, dashed, shift right =1]\tikzpullback & \widetilde{\mc M_{b_2}}/\underline{K} \arrow[d, dashed, shift right =1]\arrow[r,"g^K"]\tikzpullback & \mc M_{b_2} \arrow[r,"f"]\arrow[d, dashed, shift right =1]& \Bun_G \\
* \arrow[r]\arrow[u, shift right =1,"s",swap]&{[*/\underline{K}]} \arrow[r]\arrow[u, shift right =1,"s_K",swap]& {[*/\underline{G_{b_2}(E)}]}\arrow[r] \arrow[u,shift right=1] \arrow[phantom,ur,"\urcorner",very near start] & {[*/ \mc{G}_{b_2}]} \arrow[u,"i_{b_2}", swap]
\end{tikzcd}
\end{equation}

Now since $f$ and $g^K$ are separated, representable in locally spatial diamonds, and cohomologically smooth (\cite[Theorem~V.3.7]{fargues2021geometrization} and \cite[Proposition~24.2]{scholze2017etale}), we can apply smooth base change \cite[Proposition~23.16(ii)]{scholze2017etale} to the two squares on the upper right. 

Since $R\Gamma(\widetilde{\mc M_{b_2}},-)\circ (f g)^* Ri_{b_1,*}$ is naturally isomorphic to the pullback of $i_{b_2}^* Ri_{b_1,*}$ to $*$ (Proposition \ref{strict_hens}), the resulting $G_{b_2}(E)$-representations are already smooth, so in particular it agrees with its smooth subrepresentation.
Then making use of proposition \ref{strict-hens-cor}, we compute for any $A\in D_\et(\Bun_2^{b_1},\Lambda)$:
\begin{align*}
    R\Gamma(\widetilde{\mc M_{b_2}},-)\circ g^*f^*R_{i_{b_1,*}}A &\cong 
    \varinjlim_K \left( s^* g^*f^* R i_{b_1,*} A\right)^K\\
    &\cong 
    \varinjlim_K R\Gamma([*/\underline{K}],-)\circ s_K^* R i_{b_1,*}^K (g_{b_1}^K)^* f_{b_1}^*  A\\
    &\cong \varinjlim_K R\Gamma(\widetilde{\mc M_{b_2}}/\underline{K},-)\circ R i_{b_1,*}^K (g_{b_1}^K)^* f_{b_1}^* A\\
    &\cong \varinjlim_K R\Gamma(\widetilde{\mc M_{b_2}^{b_1}}/\underline{K}, -)\circ  (g_{b_1}^K)^*f_{b_1}^* A,
\end{align*}
where $s$ is the section $*\to \widetilde{\mc M_{b_2}^{b_1}}$ and $s_K$ is the section $[*/\underline{K}]\to \widetilde{\mc M_{b_2}}/\underline{K}$. 

But descent along $\psi:\widetilde{\mc M_{b_2}^{b_1}}\to \widetilde{\mc M_{b_2}^{b_1}}/\underline{K}$ identifies $R\Gamma(\widetilde{\mc M_{b_2}^{b_1}}/\underline{K}, -)\circ  (g_{b_1}^K)^*f_{b_1}^* A$ as the $K$-fixed points of $R\Gamma(\widetilde{\mc M_{b_2}^{b_1}},-)\circ \psi^* (g_{b_1}^K)^*f_{b_1}^* A = R\Gamma(\widetilde{\mc M_{b_2}^{b_1}},-)\circ g_{b_1}^*f_{b_1}^* A$ (note that $K$-group cohomology is simply the functor of $K$-invariants since $K$ is pro-$p$).
\end{proof}

\begin{remark}
In the case of $b_1$ basic and $b_2$ of minimal codimension in the same component, $\widetilde{\mc M_{b_2}^{b_1}}$ is the open complement $\widetilde{\mc M_{b_2}^{\circ}}\subset \widetilde{\mc M_{b_2}}$, which is quasicompact (along with all of its quotients by $\underline{K}$ and $\underline{G_{b_2}(E)}$). Then as arbitrary basechange holds for $*$-pushforwards along qcqs maps of small $v$-stacks \cite[Proposition~17.6]{scholze2017etale}, one concludes in this situation that the functor on the righthand side already produces a complex of smooth representations, without taking smooth subrepresentations. This gives the following description of gluing functors for sheaves supported on $\Bun_G^{\leq b_2}\setminus \Bun_G^{b_2}$ to a sheaf on $\Bun_G^{\leq b_2}$.
\end{remark}
\begin{corollary}
There is a natural isomorphism of functors $D_\et(\Bun_G^{\leq b_2}\setminus \Bun_G^{b_2})\to D_{\et}(*)\cong D(\Lambda)$
\[R\Gamma(\widetilde{\mc M_{b_2}},-) \circ (fg)^*R j_{<b_2,*} \xrightarrow{\sim} R\Gamma(\widetilde{\mc M_{b_2}^{^\circ}},-)\circ (f^\circ g^\circ)^*,\]
where $\widetilde{\mc M_{b_2}^\circ}\xrightarrow{g^\circ}\mc M_{b_2}^\circ \xrightarrow{f^\circ}\Bun_G^{\leq b_2}\setminus \Bun_G^{b_2}$ are defined via pullback along the inclusion into $\Bun_G$.
Furthermore, the induced smooth $G_{b_2}(E)$-actions on either side agree so that the equivalence descends to functors with image in $D_{\et}([*/\underline{G_{b_2}(E)}])$.
\end{corollary}

\subsection{Reduction to compactly supported cohomology}
In this section, we shift focus to compactly supported cohomology. This is because compactly supported cohomology is much easier to compute; for instance, there are the Künneth, basechange, and projection formulas for $!$-pushforward, and the representations that arise as $!$-pushforwards along torsors are more manageable, being compact inductions rather than full inductions. The difficulty is that the torsors that appear are not cohomologically smooth, so that the relationship between cohomology and compactly supported cohomology is not so simple. We will use a version of Hochschild-Serre spectral sequence for compactly supported cohomology to descend to a level that is cohomologically smooth over the base, where we can apply Poincar\'{e} duality.

We can further cover $\widetilde{\mc M_{b_2}^{b_1}}$ by its $\mc G_{b_1}$-torsor $\widetilde{\widetilde{\mc M_{b_2}^{b_1}}}$ as in the following diagram:

\begin{equation}
\begin{tikzcd}
\widetilde{\widetilde{\mc M^{b_1}_{b_2}}}\arrow[d,"p"]\arrow[rr]\tikzpullback & & *\arrow[d] \\
\widetilde{\mc M^{b_1}_{b_2}}\arrow[d] \arrow[r,"g_{b_1}^{}"] \tikzpullback& \mc M^{b_1}_{b_2}\arrow[d]\arrow[r,"f_{b_1}"]\tikzpullback & {[*/ \mc{G}_{b_1}]}\arrow[d, "i_{b_1}
"]\\
\widetilde{\mc M_{b_2}}\arrow[r,"g"]\arrow[d, dashed, shift right =1]\tikzpullback & \mc M_{b_2} \arrow[r,"f"]\arrow[d, dashed, shift right =1]& \Bun_G \\
* \arrow[r]\arrow[u, shift right =1] & {[*/\underline{G_{b_2}(E)}]}\arrow[r] \arrow[u,shift right=1] \arrow[phantom,ur,"\urcorner",very near start] & {[*/ \mc{G}_{b_2}]} \arrow[u,"i_{b_2}", swap]
\end{tikzcd}
\end{equation}

Whenever the unipotent part $\mc G_{b_1}^{>0}$ is nontrivial, it is isomorphic to a positive Banach-Colmez space as before (more generally it is an iterated extension of positive Banach-Colmez spaces). Then since such spaces are cohomologically trivial, we can replace $\mc G_{b_1}$ by $\underline{G_{b_1}(E)}$ by considering the intermediate torsors \[\widetilde{\widetilde{\mc M_{b_2}^{b_1}}}\xrightarrow{p^{>0}} \widetilde{\widetilde{\mc M_{b_2}^{b_1}}}/\mc G_{b_1}^{>0}\xrightarrow{p'}\widetilde{\mc M_{b_2}^{b_1}}\] for $\mc G_{b_1}^{>0}$ and $\underline{G_{b_1}(E)}$, respectively. 

\begin{proposition}
Let $d_1:=\langle 2\rho , \nu_{b_1}\rangle$. There is a natural isomorphism of functors $R\Gamma_c(\widetilde{\widetilde{\mc M_{b_2}^{b_1}}}/\mc G_{b_1}^{>0},-)\cong R\Gamma_c(\widetilde{\widetilde{\mc M_{b_2}^{b_1}}},-)$ up to Tate twist by $d_1$ and cohomological shift by $2d_1$.
\end{proposition}

\begin{proof}
This is a consequence of the descent spectral sequence in the following proposition and the identification of the cohomology and dualizing sheaf $\Lambda(d_1)[2d_1]$ of positive Banach-Colmez spaces in Proposition \ref{Hansen-bc-dualizing}.
\end{proof}
When there is a fixed $\sqrt{q}\in \Lambda$, this shift and twist is accounted for in the definition of the renormalized inclusions $i_{b,*}^{\mathrm{ren}}:\Bun_G^{b}\to \Bun_G$, where the twist $\Lambda(d_1)$ has been identified as the squareroot of the modulus character $\delta_{P_{b}}^{\pm1/2}$ for the parabolic subgroup $P_b$ of $G$ (up to inner twisting). This is one reason why it is much more natural to speak in terms of the renormalized functors.

Now we want to descend compactly supported cohomology along such a torsor, so we use the following proposition. For a locally pro-$p$ group $G$, we take as our definition of smooth homology groups $H_i(G,\sigma):=$ the $-i^\text{th}$ cohomology of the complex $\Lambda \otimes_{\mc H(G)}^\mathbb{L} \sigma$, where $\mc H(G)$ is the Hecke algebra of $G$. Recall that the category of smooth $G$-representations is linearly equivalent to the category of nondegenerate $\mc H(G)$-modules (i.e. modules $M$ such that $\mc H(G)M = M$).

\begin{proposition}\label{homological-Hochschild-Serre}
Let $G$ be a locally pro-$p$ group, and suppose $p:Y\to X$ is a $\underline{G}$-torsor in small $v$-stacks. Then for any sheaf $\mc F\in D_\et(X)$, there is a Hochschild-Serre spectral sequence 
\[E_2^{i,j}= H_{-i}(G,H_c^j(Y,p^*\mc F))\implies H_c^{i+j}(X,\mc F).\]
\end{proposition}
\begin{remark}
This proposition is essentially a consequence of the basechange isomorphism $q^*Rf_!\cong Rg_!p^*$ for
\begin{equation*}
\begin{tikzcd}
Y\arrow[r,"p"]\arrow[d,"g",swap]& X\arrow[d,"f"]\\
*\arrow[r,"q"]& {[*/\underline{G}]}.
\end{tikzcd}
\end{equation*}
Remembering the $G$-structure allows one to descend $Rg_!p^* \mc F$ to a sheaf on $[*/\underline{G}]$, and $q^*$ is the functor that forgets the $G$-structure, yielding an isomorphism as sheaves in $D_\et([*/\underline{G}])$. Then composing with $Rs_!$ for the section $s:[*/\underline{G}]\to *$ gives the result. The main content of the proof is then showing that $Rs_!$ is the functor which computes smooth group homology. \\
Another way of understanding the geometry is to write $p:Y\to X$ as a limit of the intermediate (étale) covers $p^K: Y/\underline{K}\to X$, where the functors $p^K_!\cong p^K_{\natural}$ agree as they are both left adjoint to $Rp^{K,!}$. 
\end{remark}
\begin{proof}
Let $p^K:Y/\underline{K}\to X$ be the intermediate projection for any open pro-$p$ subgroup $K\subset G$. 
We first claim that $\Lambda \cong \Lambda \otimes^\mathbb{L}_{\mc H(G)} Rp_!p^*\Lambda$ is a natural isomorphism. This can be checked $v$-locally on $X$, so without loss of generality $Y= X\times \underline{G}$ with $X$ perfectoid. In particular each $p^K$ is \'etale with fibers $K\backslash G$, and $p = \varprojlim p^K$ is pro-\'etale. Then $p_!^K \Lambda$ is the constant sheaf $\cind_K^G \Lambda\cong \Lambda[K\backslash G]$, so that $p_!\Lambda = \varinjlim_K \Lambda[K\backslash G]= \mc H(G)$, corresponding to the regular representation $(\cind_{1}^{G}\Lambda)^\text{sm}$. Since the fibers of $p$ are $0$-dimensional, we see that $Rp_! \Lambda = p_! \Lambda [0] = \mc H(G)$, hence the claim follows.

For arbitrary $\mc F$, the projection formula for sheaves of $\mc H(G)$-modules yields isomorphisms ($v$-locally)
\[\Lambda \otimes^\mathbb{L}_{\mc H(G)}Rp_!p^*\mc F\cong Rp_! (p^*\Lambda \otimes^\mathbb{L}_{\mc H(G)}p^*\mc F)\cong (Rp_!p^*\Lambda) \otimes^\mathbb{L}_{\mc H(G)}\mc F\cong \mc H(G)\otimes^\mathbb{L}_{\mc H(G)}\mc F\cong \mc F,\]
and composing both sides with $R\Gamma_c$ gives
\[R\Gamma_c(X,-)\circ(\Lambda \otimes^\mathbb{L}_{\mc H(G)}Rp_!p^*\mc F)\cong \Lambda \otimes^\mathbb{L}_{\mc H(G)}R\Gamma_c(Y,p^*\mc F)\cong R\Gamma_c(X,\mc F).\]

\end{proof}

\begin{lemma}\label{equivariant-projection-formula}
Let $p:Y\to X$ be as in Proposition \ref{homological-Hochschild-Serre}. Then there is a cartesian diagram of small $v$-stacks
\begin{center}
\begin{tikzcd}
Y\arrow[r,"g"]\arrow[d,"p",swap] & *\arrow[d,"q"]\\
X\arrow[r,"f",swap] & {[*/\underline{G}]}.
\end{tikzcd}
\end{center}
For any smooth $G$-representation $(\sigma,V)$ corresponding to a sheaf $ [\sigma] \in D_\et([*/\underline{G}])$, there is an isomorphism of underlying $\Lambda$-modules
\[R\Gamma_c(Y,p^*f^*[\sigma])\cong V \otimes^{\mathbb{L}}_\Lambda R\Gamma_c(Y,\Lambda).\]
Furthermore, $Rf_! f^*[\sigma]\cong [\sigma]\otimes^{\mathbb L}_{\Lambda} f_! \Lambda \in D(G)$, identifying the $G$-equivariant structure of $R\Gamma_c(Y, p^* f^*[\sigma])$ obtained from the canonical descent datum via pullback along $q$.
\end{lemma}
\begin{proof}
Since $q^*[\sigma]=V$ as a sheaf of $\Lambda$-modules on $*$, $p^*f^*[\sigma]=g^* V$ is a constant sheaf on $Y$. Then the projection formula along $g$ implies
\begin{align*}
R\Gamma_c(Y,p^*f^*[\sigma])\cong Rg_! g^*V\cong Rg_!(g^* V \otimes^\mathbb{L}\Lambda)\cong V\otimes^\mathbb{L}Rg_!\Lambda\cong V\otimes^\mathbb{L} R\Gamma_c(Y,\Lambda).
\end{align*}
The remaining statements follow by the projection formula for $f$ and the basechange isomorphism $q^*R f_! \xrightarrow{\cong} Rg_! p^*$.
\end{proof}
\begin{remark}\label{pulling-out-reps}
If the underlying $\Lambda$-modules of $\sigma$ or $R\Gamma_c(Y,\Lambda)$ are flat, then the righthand side is represented by the complex $V\otimes_\Lambda R\Gamma_c(Y,\Lambda)$. Since the isomorphism is given by the projection formula, hence functorial in both $V$ and $\Lambda$, such isomorphisms are compatible with pulling back along the geometric $\underline{G}$ action on $Y$, so that the complex on the righthand side descends to $\sigma \otimes_{\Lambda} R\Gamma_c(Y,\Lambda)$ in $D_\et([*/\underline{G}])$ with diagonal $G$-action. This allows us to reduce many of our computations in the following section to the compactly supported cohomology of the trivial module.
\end{remark}

The idea now is to descend the compactly supported cohomology to a cohomologically smooth $v$-stack $\widetilde{\mc M_{b_2}^{\leq b_1}}\to *$, which admits a Poincaré duality on the cohomology of ULA complexes. Consider the following diagram induced by pulling back along the factorization $i_{b_1}:\Bun_2^{b_1}\xrightarrow{i_{b_1}'} \Bun_2^{\leq b_1}\xrightarrow{j_{\leq b_1}} \Bun_2$.

\begin{equation}
\begin{tikzcd}
\widetilde{\mc M^{b_1}_{b_2}}\arrow[d] \arrow[r,"g_{b_1}^{}"] \tikzpullback& \mc M^{b_1}_{b_2}\arrow[d]\arrow[r,"f_{b_1}"]\tikzpullback & {[*/ \mc{G}_{b_1}]}\arrow[d, "i_{b_1}'
"]\\
\widetilde{\mc M^{\leq b_1}_{b_2}}\arrow[d] \arrow[r,"g_{\leq b_1}^{}"] \tikzpullback& \mc M^{\leq b_1}_{b_2}\arrow[d]\arrow[r,"f_{\leq b_1}"]\tikzpullback & {\Bun_G^{\leq b_1}}\arrow[d, "j_{\leq b_1}
"]\\
\widetilde{\mc M_{b_2}}\arrow[r,"g"] & \mc M_{b_2} \arrow[r,"f"]& \Bun_2. \\
\end{tikzcd}
\end{equation}
\begin{lemma}
    $\widetilde{\mc M_{b_2}^{\leq b_1}}\to *$ is cohomologically smooth of dimension $\langle 2\rho ,\nu_{b_2} \rangle$.
\end{lemma}
\begin{proof}
    The map $j_{\leq b_1}$ is open and hence cohomologically smooth of dimension $0$, so that $\mc M_{b_2}^{\leq b_1}\to \mc M_{b_2}$ is cohomologically smooth of dimension $0$. Then the lemma follows since $\mc M_{b_2}\to [*/\underline{G_{b_2}(E)}]$ is cohomologically smooth of dimension $\langle 2\rho, \nu_{b_2} \rangle$ (Theorem \ref{properties-of-M_b}).
\end{proof}
Now since $i_{b_1}$ is the inclusion of a closed substack, $R i_{b_1, !}'=i_{b_1,*}':D_\et(\Bun_G^{b_1})\subset D_{\et}(\Bun_G^{\leq b_1})$ is the inclusion of a triangulated subcategory, so proper basechange allows for the descent of compactly supported cohomology to $\widetilde{\mc M_{b_2}^{\leq b_1}}$ where Verdier duality holds.

%% file: Some_Examples.tex
This section is the computational heart of the paper, where we restrict to $G=\GL_2$. The computation naturally splits into integral and half-integral slope cases. Since the $\mc M_b$ covers can be described in terms of extensions of line bundles, twisting by $\mc O(\pm 1)$ yields equivariantly isomorphic spaces for each integral (resp. half-integral) slope. Thus, we may restrict our attention to the slope $0$ (resp. slope $1/2$) components. Some of the following lemmas may be helpful in understanding the computations of gluing data for arbitrary strata of $\Bun_2$. We only compute the gluing data from basic strata to their immediate specializations, which is already quite involved.
\subsection{Geometric lemmas}
The following lemmas will be useful in analyzing the geometry and cohomology of the spaces that appear. In particular, we consider maps between positive Banach-Colmez spaces for the computation of their cohomology in terms of perfectoid disks. 
\begin{lemma}\label{nonsaturated}
Suppose $\mc E$ is a vector bundle on an absolute Fargues-Fontaine curve $X_C$. An injection $\mc O \hookrightarrow \mc E$ extends to an injection $\mc O(n) \hookrightarrow \mc E$ (for some $1\leq n \leq \deg \mc E$) if and only if it is non-saturated. (Recall that an injective map of vector bundles is defined to be saturated if its cokernel is locally free.)
\end{lemma}

\begin{proof}
Write $\mc C$ for the cokernel of $\mc O \hookrightarrow \mc E$. By assumption, $\mc C$ is not locally free, and hence has a nontrivial torsion submodule. Then there is a skyscraper sheaf $\mc C_{\text{tors}} \hookrightarrow \mc C$ corresponding to a closed effective Cartier divisor of $X_C$ of degree $n$. Pulling back $\mc E \to \mc C$ along this inclusion induces the following short exact sequence:

\begin{center}
\begin{tikzcd}
0 \arrow[r] & \mc O \arrow[r]\arrow[d, equal] & \mc P \arrow[r] \arrow[d]\arrow[phantom,dr,"\ulcorner",very near start]& \mc C_{\text{tors}} \arrow[r]\arrow[d,hookrightarrow] & 0\\
0 \arrow[r] & \mc O \arrow[r] & \mc E \arrow[r] & \mc C \arrow[r] & 0.\\
\end{tikzcd}
\end{center}

Then $\mc P\to \mc E$ is injective by the five lemma, and in particular $\mc P$ is a line bundle as $X_C$ is locally a principal ideal domain \cite[Theorem~I.3.1]{fargues2021geometrization}. As $\mc C_{\text{tors}}$ has degree $n$, additivity of degree implies $\mc P \cong \mc O(n)$.

For the converse, there is a commutative diagram with exact rows
\begin{center}
\begin{tikzcd}
0 \arrow[r] & \mc O \arrow[r]\arrow[d, equal] & \mc O(n) \arrow[r] \arrow[d,hookrightarrow]& \mc C' \arrow[r]\arrow[d,dashed] & 0\\
0 \arrow[r] & \mc O \arrow[r] & \mc E \arrow[r] & \mc C \arrow[r] & 0,\\
\end{tikzcd}
\end{center}
where the right arrow is induced by the universal property of the cokernel. This map is injective by the five lemma, and hence $\mc C$ is not locally free.
\end{proof}
In particular, after pullback to $\spa C$, there is an identification of the torsion of the cokernel of a non-saturated injection $\mc C_{\mathrm{tors}}\cong \bigoplus_{i=1}^n \iota_{x_i,*}C^\sharp$ for certain $x_i\in \Div^1$ which may not be distinct. Pulling back along some $\iota_{x_i,*}C^\sharp$ gives a factorization through a line bundle isomorphic to $\mc O(1)$.
\begin{lemma}\label{BC-homological-disks}
For any positive integer $n$, there is a pro-\'etale torsor $\D^n\to \BC(\mc O(n))_{\spa C}$ with fibres $\underline{E}^{n-1}$, where $\D$ is the perfectoid open unit disk over $\spa C$.
\end{lemma}
\begin{proof}
We identify $\D\xrightarrow{\cong} \BC(\mc O(1))$ using the isomorphism of Proposition \ref{smalldegBC}. Then the idea is to construct covers inductively using extensions
\[0\to \mc O\xrightarrow{\smallvector{f}{g}} \mc O(1)\oplus \mc O(n) \xrightarrow{(\alpha,\beta)} \mc O(n+1) \to 0,\]
with $f,\beta \in \BC(\mc O(1))$ and $g,\alpha \in \BC(\mc O(n))$.
For $n=1$, the long exact sequence on cohomology induces a short exact sequence of group $v$-sheaves,
\[0\to \underline{E}\to \mc BC(\mc O(1))\times \mc BC(\mc O(1)) \to \BC(\mc O(2)) \to 0,\]
yielding a surjection $\D^2\to \BC(\mc O(2))$ with fibres $\underline{E}$. Now for any $n$, we have
\[0\to \underline{E}\to \mc BC(\mc O(1))\times \mc BC(\mc O(n)) \to \BC(\mc O(n+1)) \to 0,\]
where the middle term sits in a short exact sequence
\[0 \to \underline{E}^{n-1}\to \D\times \D^n \to \BC(\mc O(1))\times \BC(\mc O(n))\to 0\]
by induction hypothesis. A careful examination of all of the maps identifies the kernel of $\D^{n+1}\to \BC(\mc O(n+1))$ as $\underline{E}^n$.\\
Alternatively, there is no need for the induction if one considers an extension of the form $0\to \mc O^n \to \mc O(1)^{n+1}\to \mc O(n+1)\to 0.$
\end{proof}

\begin{corollary}
For any positive integer $n$, there is a $v$-cover of the negative Banach-Colmez space $\BC(\mc O(-n)[1])_{\spa C}$ by $\D^n$.
\end{corollary}
\begin{proof}
The long exact sequence associated to an extension
\[0\to \mc O(-n)\to \mc O^2\to \mc O(n)\to 0\]
gives a presentation
\[0\to \underline{E}^2 \to \BC(\mc O(n))\to \BC(\mc O(-n)[1])\to 0.\]
\end{proof}

\begin{lemma}[{\cite[Lemma~6.11]{hamann2023geometric}}]\label{negative_bc_dualizing_sheaf}
    Let $d$ be a positive integer and $f:\BC(\mc O(-d)[1])\to *$ be the structure morphism of a negative Banach-Colmez space. Then $f$ is cohomologically smooth of dimension $d$ with dualizing sheaf $f^!\Lambda \cong \Lambda (d)[2d]$.
\end{lemma}
\begin{proof}
    The idea is to descend the dualizing sheaves of positive Banach-Colmez spaces in Proposition \ref{Hansen-bc-dualizing} along the $\underline{E}^2$-torsor in the previous corollary. But $\underline{E}^2$ acts trivially on $\Lambda(d)[2d]$ as a sheaf on $\BC(\mc O(d))$, so the descent identifies $f^! \Lambda \cong \Lambda(d)[2d]$ as sheaves on $\BC(\mc O(-d)[1])$.
\end{proof}

\subsection{Specialization from the basic point in an integral slope component} 
This section and the following four sections contain the computation of the gluing data for sheaves supported on the basic stratum in slope $0$ to the immediate specialization in $\Bun_2$. We first identify the moduli space $\widetilde{\widetilde{\mc M_{b_2}^{b_1}}}$ in terms of $\underline{\GL_2(E)}$-orbits inside positive Banach-Colmez spaces, then compute its $\underline{\GL_2(E)}$-equivariant cohomology. This ultimately boils down to cohomology of (perfectoid) disks and a comparison theorem of Huber to the \'etale cohomology of Henselian schemes, as well as Kummer theory to understand Frobenius equivariance of the Tate twist in the cohomology of $\Proj^1$.
\subsubsection{Geometry of the covers and the gluing functor for unramified characters}

We consider the open substack consisting of the basic point $\mc O^2 \leftrightarrow b_1 \in B(G)$ and its specialization $ \mc O(-1)\oplus \mc O(1) \leftrightarrow  b_2 \in B(G)$. The corresponding stabilizer group diamonds (up to taking $\pi_0$) are $\mc G_{b_1}\cong \underline{\GL_2(E)}$ and $\underline{G_{b_2}(E)}\cong \underline{E_1}^\times \times \underline{E_2}^\times$, where we use the notation $E_i^\times = E^\times$ to distinguish the first and second factors of the product.
\begin{proposition}\label{Equivariant-iso-cover}
After basechange to $\spa C$, there is a $\mc G_{b_1}\times \underline{G_{b_2}(E)}$-equivariant isomorphism
\[\widetilde{\widetilde{\mc M_{b_2}^{b_1}}}_{,C}\cong \left [\BC(\mc O(1))^2 \setminus \underline{\GL_2(E)}\cdot\Delta(\BC(\mc O(1)))\right]\times \underline{E^{\times}},
\]
where $\BC (\mc O(1))^2$ inherits the natural action of $\underline{G_{b_1}(E)}\cong \underline{\GL_2(E)}$ on $\mc O(1)^2$ via postcomposition, and $\underline{\GL_2(E)}$ acts on the factor of $\underline{E^\times}$ via the determinant. Furthermore, the first factor of $\underline{G_{b_2}(E)}\cong \underline{E_1^\times}\times \underline{E_2^\times}$ acts on $\BC(\mc O(1))^2$ via precomposition, and each $\underline{E_i^\times}$ acts on $\underline{E^\times}$ by the inverse of multiplication.
\end{proposition}
\begin{proof}
Recall that the moduli functor $\widetilde{\widetilde{\mc M_{b_2}^{b_1}}}$ is the $v$-stackification of the groupoid of filtered vector bundles $\mc K\subset \mc E$ over $X_{S,E}$ such that $\mc E$ is isomorphic to $\mc O^2$ and the associated graded bundle $\mc K \oplus \mc C$ is isomorphic to $\mc O(-1)\oplus \mc O(1)$ as graded vector bundles, together with global trivializations $\phi_K,\phi_E,$ and $\phi_C$. The Newton polygon of $\mc E$ is constant in $S$, and hence there exists a global Harder-Narasimhan filtration \cite[Theorem~II.2.19]{fargues2021geometrization}. Since the morphism $\widetilde{\mc M_{b_2}}\to \Bun_2$ is representable in locally spatial diamonds and $\widetilde{\widetilde{\mc M_{b_2}^{b_1}}}$ is the pullback along $*\to \Bun_2^{b_1}\to \Bun_2$, we see that $\widetilde{\widetilde{\mc M_{b_2}^{b_1}}}$ is isomorphic to a sheaf (after basechange to $\spa C$); it becomes a sheaf as soon as we identify isomorphism classes in the fiber categories. Then as both the source and target are $v$-sheaves, it suffices to define the morphism over strictly totally disconnected $T$, where the global Harder-Narasimhan filtrations are split and hence there are diagrams
\begin{center}
\begin{tikzcd}
0 \arrow[r] & \mc K |_{X_T} \arrow[r]\arrow[d,"\phi_K"] & \mc E|_{X_T} \arrow[r]\arrow[d,"\phi_E"]& \mc C |_{X_T} \arrow[d,"\phi_C"]\arrow[r]& 0 \\
&\mc O(-1) \arrow[r,dashed,"\smallvector{f}{g}"]\arrow[d, swap,"\underline{E_1^\times}"]&\mc O^2 \arrow[d, swap,"\underline{GL_2(E)}"]& \mc O(1)\arrow[d, swap,"\underline{E_2^\times}"] & \\
&\mc O(-1) &\mc O^2 & \mc O(1) & 
\end{tikzcd}
\end{center}
inducing sections $f,g\in \BC(\mc O(1))^2(T)$, with $\underline{G_b(E)}$-actions given by pre/postcomposition. Such maps $\smallvector{f}{g}$ are saturated, and all saturated injections are in the image by considering $\mc K=\mc O(-1)$ and $\mc E=\mc O^2$ (necessarily the cokernel is isomorphic to $\mc O(1)$ by additivity of degree). By Lemma \ref{nonsaturated}, the non-saturated injections factor through $\mc O(-1)\hookrightarrow \mc P\cong \mc O$. The factorizations $\mc O\to \mc O^2$ are determined by points of $\BC(\mc O)^2\cong \underline{E}^2$ which has transitive orbit under the $\underline{\GL_2(E)}$-action, so the non-saturated locus (including the zero map $\mc O(-1)\to \mc O^2$) is given by $\underline{\GL_2(E)}\cdot \Delta(\BC(\mc O(1)))$, where $\Delta(\BC(\mc O(1)))$ is the diagonal inside $\BC(\mc O(1))^2$; note that this description does not require fixing any isomorphism $\mc P\cong \mc O$.

Taking determinant bundles yields canonical isomorphisms
$\mc K \otimes \mc C \xrightarrow{\sim}\Lambda^2 \mc E$, which can be composed to isomorphisms
\[\mc O(-1)\otimes \mc O(1)\xrightarrow{\phi_K^\inv \otimes \phi_C^\inv}\mc K \otimes \mc C \xrightarrow{\sim}\Lambda^2 \mc E\xrightarrow{\det \phi_E}\mc O,\]
parameterized by the punctured Banach-Colmez space $\BC(\mc O)\setminus 0 \cong \underline{E^\times}$. Then the $\underline{E_1^\times}$- and $\underline{E_2^\times}$- actions on $\mc O(-1)$ and $\mc O(1)$, resp., induce the inverse of multiplication on the space $\underline{E^\times}$, while the induced $\underline{\GL_2(E)}$-action is multiplication by the determinant. Together with the projection to $\BC(\mc O(1))^2\setminus{\underline{\GL_2(E)}\cdot \Delta (\BC(\mc O(1)))}$, which pins down $\phi_K$ and $\phi_E$, the projection to $\underline{E}^\times$ furthermore pins down the trivialization $\phi_C$, hence the map is an isomorphism of $v$-sheaves.

\end{proof}
\begin{remark}\hfill \\
\begin{itemize}
    \item The same moduli description appears for any integral slope component after twisting extensions of the form 
    \[0 \to \mc O(\lambda -1) \to \mc O(\lambda)^2\cong \mc O(\lambda)\otimes \mc O^2\to \mc O(\lambda +1)\to 0\]
    by the line bundle $\mc O(-\lambda)$. Thus,  $\widetilde{\widetilde{\mc M_{b_2}^{b_1}}}$ for any basic point $b_1$ with integral slope and immediate specialization $b_2$, together with its $\underline{G_{b_i}(E)}$-equivariant data, is described by the space in the proposition. 
    \item We prefer to parameterize trivializations of the determinant bundle, since trivializing the cokernel yields geometric actions on the factor $\underline E^\times$ via the inverse of determinant by $\underline{\GL_2(E)}$ and $e_2/e_1$ by $\underline{E_1^\times} \times \underline{E_2^\times}$, which we find a bit more tricky to identify and compute with (in particular, the compactly supported global sections on $\underline E^\times$ are no longer $\cind_{\SL_2(E)}^{\GL_2(E)}\Lambda$, since the domain of such functions $E^\times \to \Lambda$ is acted upon by the \emph{inverse} of determinant). A careful inspection of the computations in Theorem \ref{cs-coh-tilde-M-trivial-module} shows that the resulting $G_{b_2}(E)$-modules from either description are equal in the end.
    \item The approach of identifying the space $\widetilde{\widetilde{\mc M_{b_2}^{b_1}}}$ as the complement of a positive Banach-Colmez space is informed by the Drinfeld compactification of $\Bun_P$, which is instrumental in studying geometric Eisenstein series (see \cite{hamann2023geometric} and \cite{hamann2024geometriceisensteinseriesi}).
\end{itemize}
\end{remark}

The idea now is to compute the cohomology of $\widetilde{\widetilde{\mc M_{b_2}^{b_1}}}$ using an excision short exact sequence for $Z:= \underline{\GL_2(E)}\cdot \Delta(\BC(\mc O(1)))$ and its open complement $U$ inside $\BC(\mc O(1))^2$. Since our computations are insensitive to basechange to $\spa C$ and tilting, we may replace many of these spaces by perfectoid spaces over $C$.
\begin{lemma}\label{geometry-of-Z}
There is a $\GL_2(E)$-equivariant projection map $Z \setminus 0 \to \underline{\Proj^1(E)}=\underline{B \backslash \GL_2(E)}$ with typical fibre $\BC(\mc O(1))\setminus 0$, where the action on the base is the regular right action. 
\end{lemma}
\begin{remark}\hfill
\begin{itemize}
\item The fibration is a product as abstract $v$-sheaves since the base $\underline{\Proj^1(E)}$ is strictly totally disconnected; however, such an identification is not equivariant for the $\underline{\GL_2(E)}$-action since there is nontrivial monodromy on the induced action on fibers, which looks similar to a factor of automorphy that appears in the theory of classical modular forms. We will not need an explicit description of the monodromy, as we will see in the proof of Proposition \ref{actions-on-Z}; ultimately, one reduces to the $B$-equivariant structure over a distinguished fiber. 
\item The fiber $F$ of $Z\setminus 0\to \underline{\Proj^1(E)}$ over $\infty$ has the following geometric interpretation: by Lemma \ref{nonsaturated}, every geometric point in this fiber corresponds to a diagram
\begin{center}
\begin{tikzcd}
0 \arrow[r] & \mc O(-1) \arrow[r,"f"]\arrow[d, equal] & \mc O \arrow[r] \arrow[d,hookrightarrow,"\smallvector{1}{0}"]& i_* C^\sharp \arrow[r]\arrow[d] & 0\\
0 \arrow[r] & \mc O(-1) \arrow[r,"\smallvector{f}{0}"] & \mc O^2 \arrow[r] & i_*C^\sharp \oplus \mc O \arrow[r] & 0,
\end{tikzcd}
\end{center}
obtained by pulling back along the torsion subsheaf of the cokernel. Restricting the $\mc G_{b_1}$-action on $Z\setminus 0$ to $\underline{B}$ fixes this fiber, so $F$ is naturally a $\underline{B}$-space. The action only depends on the first toral factor of $B$, so the quotient by the $B$-action identifies as the projectivized Banach-Colmez space $\underline{E^\times}\backslash\left( \BC(\mc O(1))\setminus{0}\right )\cong \mathrm{Div}^1$. This can also be thought of as the space of modifications of line bundles $\mc O \hookrightarrow \mc L$ with cocharacter $1\in X_*(E^\times)$. 
\end{itemize}
\end{remark}
\begin{proof}
By Lemma \ref{nonsaturated}, $Z\setminus 0$ is the space parameterizing non-saturated injections $\mc O(-1) \hookrightarrow \mc O^2$, and hence factorizations 
\begin{center}
\begin{tikzcd}
0\arrow[r]&\mc O(-1) \arrow[r,"\smallvector{f}{g}"]&\mc O^2\arrow[r] & \mc C\arrow[r] & 0\\
0 \arrow[r]&\mc O(-1)\arrow[r]\arrow[u, equal] &\mc P \arrow[u, hookrightarrow]\arrow[r]& i_* C^\sharp \arrow[u,hookrightarrow]\arrow[r] & 0.
\end{tikzcd}
\end{center}
Recall that the inclusion of $i_* C^\sharp$ into $\mc C$ is the torsion subsheaf, and the bottom row is obtained by pulling back along this inclusion, so that there have been no choices involved. The choices of identifications $\mc O \xrightarrow{\sim}\mc P$ differ by the action of $\Aut \mc O = \underline{E^\times}$. Then a point of $Z\setminus 0$ determines an injection $\mc O \xrightarrow{\sim}\mc P\to \mc O^2$ modulo the $\underline{E^\times}$-action on the source; i.e. a point of $\underline{\Proj^1(E)}$, but we rather identify with the transpose (row vectors) so that the natural $\GL_2(E)$-action is on the right. For example, the fiber over $[1:-1]$ are the sections $f-g=0$ (i.e. $\Delta(\BC(\mc O(1)))\setminus{0}$), and similarly the fiber over any other point is a translate of $\Delta(\BC(\mc O(1)))\setminus 0$. Proposition \ref{smalldegBC} allows us to replace all occurrences of $\BC(\mc O(1))$ with the perfectoid open unit disk $\D$, where the $0$ point of $\BC(\mc O(1))$ is identified with the origin of $\D$. 

To see that $Z\setminus 0\to \underline{\Proj^1(E)}$ is equivariant, the fibres over a point in  $\underline{\Proj^1(E)}$ inside $Z\setminus 0$ are determined via relations 
\[[a:b]\smallvector{f}{g}=af+bg=0.\]
Then for any $\gamma \in \GL_2(E)$, 
\[0= [a:b]\smallvector{f}{g}= [a:b]\gamma ^\inv \cdot \gamma \smallvector{f}{g},\]
which is to say that the $\gamma$-translate of the point in $Z\setminus 0$ is sent to the $\gamma$-translate of the image in $\underline{\Proj^1(E)}$.
\end{proof}

Now, the abstract cohomology of $Z \cong \left ( \underline{\Proj^1(E)}\times \D^* \right )\cup \ *$ is also computable via an excision short exact sequence. For this, we will need to compute some cohomology of intermediate spaces. The cohomology groups of these spaces are well known, but we include the computations since we will need to determine the induced $G_{b_1}(E)$- and $G_{b_2}(E)$-actions.

\begin{lemma}
The cohomology of the perfectoid open unit disk $\D$ is 
\[H^i(\D,\Lambda)\cong \begin{cases}\Lambda,\quad i=0\\
0,\quad \text{else}.
\end{cases}\]
\end{lemma}
\begin{proof}
First note that perfection for adic spaces in characteristic $p$ does not affect the computation of cohomology, since the map from the perfection to the original space is a universal homeomorphism. Then for the remainder of the proof, the following adic spaces are the usual geometric (nonperfectoid) spaces over $\spa C$. Viewing $\D = \bigcup_n \mathbb{B}_n$ as an increasing union over closed balls $\mathbb{B}_n = \spa C\left<T,T^n/\varpi\right>$, there is a spectral sequence that degenerates immediately into a series of short exact sequences
\[0\to R^1\varprojlim_n H^{i-1}(\mathbb{B}_n,\Lambda)\to H^i(\D,\Lambda)\to \varprojlim_n H^i(\mathbb{B}_n,\Lambda)\to 0,\]
as in Lemma 3.9.2 of \cite{Huber}. The cohomology of $\mathbb{B}_n$ is supported in degree $0$ \cite[Example~0.4.6]{Huber}, so the above short exact sequences imply isomorphisms \[H^i(\D,\Lambda)\cong \varprojlim_n H^i(\mathbb{B}_n,\Lambda),\]
for all $i\neq 1$. 

Since all of the maps $\mathbb{B}_n \to \D$ are inclusions of open subspaces, the induced maps on $H^0$ are given by restriction of the constant sheaf $\Lambda$. Thus the transition maps in the inverse system are the identity on $\Lambda$ (in particular they satisfy the Mittag-Leffler condition), so that the higher inverse limit vanishes in the $i=1$ short exact sequence as well. 
\end{proof}

\begin{lemma}\label{coho-punctured-disk}
The cohomology of the perfectoid punctured open unit disk $\D^*$ is
\[H^i(\D^*,\Lambda)\cong \begin{cases}\Lambda,\quad & i= 0\\
    \Lambda(-1),\quad & i=1\\
0,\quad &\text{else}.
\end{cases}\]
\end{lemma}
\begin{proof}
The idea is to use the decomposition 
\[\D^* = \bigcup_{(\varepsilon,r)\in (0,\infty)}\annu(\varepsilon,r),\]
and the short exact sequences as in the previous lemma, where $\annu(\varepsilon,r)\subset \D$ is the preimage of the interval $(\varepsilon,r)$ for the radius function $\rad:\mathbb{B}\to [0,\infty]$ (note that this radius function is the negative logarithm of the intuitive radius on classical points in equal-characteristic, cf. \cite[Proposition~4.2.6]{berkeley-notes}; i.e. under this normalization, $\rad(0) = \infty$ and the radius of the boundary is $0$). 

We identify each $\mathbb{A}(\varepsilon,r)$ as the intersection of two closed balls $i_\varepsilon:\mathbb{B}_\varepsilon\to \Proj^1$ centered around $\infty$ and $j_r:\mathbb{B}_r'\to \Proj^1$ centered around $0$ inside the adic projective line. The cohomology of $\Proj^1$ is computed in \cite[Example~0.2.5]{Huber} to be a single copy of $\Lambda$ in degrees $0$ and $2$. Then the associated Mayer-Vietoris sequence yields short exact sequences
\[0\to H^0(\Proj^1,\Lambda) \xrightarrow{-\Delta} H^0(\mathbb{B}_\varepsilon,\Lambda)\oplus H^0(\mathbb{B}_r',\Lambda)\xrightarrow{\Sigma} H^0(\mathbb{A}(\varepsilon,r),\Lambda)\to 0,\]
\[0\to H^1(\mathbb{A}(\varepsilon,r),\Lambda)\xrightarrow{\delta} H^2(\Proj^1,\Lambda)\to 0,\]
and
\[0\to H^2(\mathbb{A}(\varepsilon,r),\Lambda)\to 0.\]
In the first sequence, the first map is the inclusion of $\Lambda$ into $\Lambda^{2}$ by the anti-diagonal $(\lambda,-\lambda)$, and the second map is restriction to the annulus followed by the sum, hence $H^0(\mathbb{A}(\varepsilon,r),\Lambda)= \Lambda$. The connecting homomorphism identifies each $H^1(\annu(\varepsilon,r),\Lambda)\cong H^2(\Proj^1,\Lambda)\cong\Lambda(-1)$.

The Mayer-Vietoris sequences above are induced by taking $R\Gamma(\Proj^1,-)$ on the triangles 
\[\Lambda \to Ri_{\varepsilon,*}i_{\varepsilon}^*\Lambda\oplus Rj_{r,*}j_r^*\Lambda \to Ri_{\varepsilon,r,*}i_{\varepsilon,r}^*\Lambda \to \Lambda [1]\in D(\Proj^1_{\et},\Lambda),\]
which are functorial in the following sense:
for each $\delta<\varepsilon$, 
\begin{center}
\begin{tikzcd}
\mathbb{B}_\delta \arrow[dr,"i_\delta",swap]\arrow[r,"h"] & \mathbb{B}_\varepsilon \arrow[d, "i_\varepsilon"]\\
& \Proj^1,
\end{tikzcd}
\end{center}
there are induced adjunction maps
\[Ri_{\varepsilon,*}i_{\varepsilon}^* \Lambda \to Ri_{\varepsilon,*}Rh_*h^* i_{\varepsilon}^* \Lambda \cong Ri_{\delta,*}i_\delta^* \Lambda,\]
and similarly for $j_r$. Then the compatibility of these adjunction maps induce a commutative square 
\begin{center}
\begin{tikzcd}
\Lambda \arrow[d,"1_\Lambda"]\arrow[r] & Ri_{\varepsilon,*}i_{\varepsilon}^* \Lambda\oplus Rj_{r,*}j_r^*\Lambda\arrow[d]\\
\Lambda \arrow[r] & Ri_{\delta,*}i_{\delta}^* \Lambda\oplus Rj_{s,*}j_s^*\Lambda,
\end{tikzcd}
\end{center}
which extends to a morphism of triangles.

As the induced maps on $H^0$ are restriction maps, each map $H^0(\mathbb{A}(\varepsilon',r'),\Lambda)\to H^0(\mathbb{A}(\varepsilon,r),\Lambda)$ is $1_\Lambda$. As the maps between $H^2(\Proj^1,\Lambda)$ for varying annuli are the identity, the induced maps $H^1(\mathbb{A}(\varepsilon',r'),\Lambda)\to H^1(\mathbb{A}(\varepsilon,r),\Lambda)$ are isomorphisms. Then the higher inverse limits vanish and we conclude $H^i(\D^*,\Lambda)\cong \varprojlim_{\varepsilon,r} H^i(\annu(\varepsilon,r),\Lambda)\approx H^i(\annu(\varepsilon,r),\Lambda)$ for every $i$.
\end{proof}
\begin{remark}
The Tate twists that appear in the above computations do not affect the resulting $G_b(E)$-representation structures on cohomology. This is because the geometric Frobenius for $\D^*$ does not act on the individual annuli $\mathbb{A}(\varepsilon,r)$, rather it acts on the inductive system as a whole, as in the proof of the following proposition. Alternatively, the natural isomorphisms $\delta: H^1(\annu(\varepsilon,r), \Lambda) \xrightarrow{\sim} H^2(\Proj^1,\Lambda)\cong \Lambda(-1)$ already yield an isomorphism $H^1(\D^*, \Lambda)\cong \Lambda(-1)$, which identifies the correct Frobenius character anyways. This \emph{considerably} simplifies the proof of the following proposition since the Kummer short exact sequences are no longer necessary. However, we leave in the computation of the characters via Kummer theory since we find the identifications to be informative in identifying the representation structures on cohomology induced by Frobenius equivariance.
\end{remark}

At this point we would like to begin computing the compactly supported cohomology and understand their structure as $\GL_2(E)$-representations.
\begin{proposition}\label{actions-on-Z}
The geometric $\GL_2(E)$-action on $Z\setminus 0$ yields 
\[H^i_c(Z\setminus 0,\Lambda)=\begin{cases}\ind_B^{\GL_2(E)} \Lambda,\quad & i=1,\\ \ind_{B}^{\GL_2(E)}( \abs{-}\otimes 1),\quad & i=2,\\
0,\quad &\text{else},\end{cases}\]
where $\abs{-}\otimes 1$ is the torus character $(t_1,t_2)\mapsto \abs{t_1}\in \Aut(\Lambda)$.
\end{proposition}
\begin{proof}
We will apply the Leray spectral sequence for compactly supported cohomology along the fibration $f:Z\setminus 0\to \underline{\Proj^1(E)}$,
\[H_c^i(\underline{\Proj^1(E)},R^jf_!\Lambda)\implies H_c^{i+j}(Z\setminus 0,\Lambda).\]
As $\Lambda$ is a $\mc G_{b_1}$-equivariant sheaf on $Z\setminus 0$, it is the pullback of the constant sheaf on $[(Z\setminus 0)/\mc G_{b_1}]$, so the basechange formula for $!$-pushforward allows us to consider $Rf_! \Lambda$ to be the pullback of $Rf'_! \Lambda \in D_\et([\underline{\Proj^1(E)}/\mc G_{b_1}])$ as in the diagram
\begin{center}
\begin{tikzcd}
Z\setminus 0 \arrow[r]\arrow[d,"f",swap] & {[(Z\setminus 0)/\underline{\GL_2(E)}]}\arrow[d,"f'"] \\
\underline{\Proj^1(E)}\arrow[r]& {[\underline{\Proj^1(E)}/\underline{\GL_2(E)}]}.
\end{tikzcd}
\end{center}
But the category of sheaves on $[\underline{\Proj^1(E)}/\underline{\GL_2(E)}]= [\underline{B\backslash \GL_2(E)}/\underline{\GL_2(E)}]$ is equivalent to the sheaves on $ [\underline{B}\backslash *]$, where the point $*$ is identified with the point at $\infty\in \underline{\Proj^1(E)}$ (corresponding to the trivial coset $B$), and the equivalence is given by (compact) induction of $B$-representations to $\GL_2(E)$. Then it suffices to compute the compactly supported cohomology of the fiber over $\infty$ as $B$-representations, which extends uniquely to a complex of $\GL_2(E)$-equivariant sheaves on the base $\underline{\Proj^1(E)}$. Since this fiber $\D^*\subset \D^2$ is the inclusion into the first coordinate, the $B$-action only depends on the first coordinate of the diagonal torus inside $B$, where the uniformizer $\pi$ acts by the inverse of geometric Frobenius on $\D^*$ (see Remark \ref{basechange-to-C}).

For the cohomology of $\D^*$, the proof of Lemma \ref{coho-punctured-disk} constructed canonical isomorphisms $H^i(\D^*,\Lambda) \cong  \varprojlim_{\varepsilon,r}H^i(\annu(\varepsilon,r),\Lambda)$. As all the transition maps on $H^0$ are the identity, $H^0(\D^*,\Lambda)=\Lambda$ is the trivial representation. To understand the action on $H^1(\D^*,\Lambda)$ we consider the Kummer theory of $\Lambda = \varinjlim\bigoplus \underline{\Z/n_i}$ as sheaves on $\annu(\varepsilon,r)$. Viewing each $\annu(\varepsilon,r)$ as a rational subset of $\mathbb{B}=\spa C\left<T^{1/p^\infty}\right>$, so that the functions on $\annu(\frac{a}{n},\frac{b}{m})$ are $C\left<T^{1/p^\infty}, T^n/\pi^a,\pi^b/T^m\right>$, the action $\pi= (r\mapsto r^{1/q})$ on $\spa (R,R^+)$-points yields isomorphisms $\phi: \annu(\varepsilon,r)\xrightarrow{\cong}\annu(\frac{\varepsilon}{q},\frac{r}{q})$, where $\phi^*(T)= T^{1/q}$ on the coordinate. Let $\Gamma_{\varepsilon,r}^*$ be the invertible global sections of $\annu(\varepsilon,r)$. Then there are induced morphisms of Kummer short exact sequences

\begin{center}
\begin{tikzcd}
0\arrow[r]& \Gamma_{\frac{\varepsilon}{q},\frac{r}{q}}^*/\Gamma_{\frac{\varepsilon}{q},\frac{r}{q}}^{*n} \arrow[r]\arrow[d,"\phi^*"]& H^1(\annu(\frac{\varepsilon}{q},\frac{r}{q}),\underline{\Z/n})\arrow[r]\arrow[d,"\phi^*"]& \Pic_n\arrow[r]\arrow[d,"\phi^*"] & 0  \\
0\arrow[r]& \Gamma_{\varepsilon,r}^*/\Gamma_{\varepsilon,r}^{*n} \arrow[r]& H^1(\annu(\varepsilon,r),\underline{\Z/n})\arrow[r]& \Pic_n\arrow[r] & 0. 
\end{tikzcd}
\end{center}
Recall that the map $\Gamma^*/\Gamma^{*n}\to H^1(\annu,\underline{\Z/n})$ takes an invertible section $f$ to the line bundle $\mc O$ with linearization $\psi:\mc O\xrightarrow{\sim}\mc O^{\otimes n}, 1\mapsto f^{\otimes n}$. The connecting homomorphism from Lemma \ref{coho-punctured-disk} yields isomorphisms $H^1(\annu, \underline{\Z/n)}\cong \Z/n$; in particular, the coordinate $T$ is mapped to a generator of $H^1(\annu,\underline{\Z/n})$ for each annulus. We fix isomorphisms $H^1(\annu, \underline{\Z/n})\approx \Z/n$ by demanding that this generator is $1\in \Z/n$. Then there are commutative squares
\begin{center}
\begin{tikzcd}
\Gamma_{\frac{\varepsilon}{q},\frac{r}{q}}^*/\Gamma_{\frac{\varepsilon}{q},\frac{r}{q}}^{*n} \arrow[r,"T\mapsto 1"]\arrow[d,swap, "T\mapsto T^{1/q}"]& \Z/n\arrow[d,"\phi^*"]\\
\Gamma_{\varepsilon,r}^*/\Gamma_{\varepsilon,r}^{*n} \arrow[r,"T\mapsto 1"]& \Z/n.
\end{tikzcd}
\end{center}
Thus the $\underline{E^\times}$-action on $H^1(\D^*,\Lambda)\cong \varprojlim H^1(\annu(\varepsilon,r),\Lambda)\approx \Lambda$ is by $\pi=q^\inv=\abs{\pi} \in \Aut(\Lambda)$. 

Computing the compactly supported cohomology $H_c^{2-i}(\D^*,\underline{\Z/n})$ from $H^i(\D^*,\underline{\Z/n})$ using Poincar\'e duality introduces a Tate twist by $\mu_n=\underline{\Z/n}(1)$, which we identify as the trivial sheaf $\underline{\Z/n}$ with (arithmetic) Frobenius acting by $q=\abs{\pi}^\inv$, and then dualizing. This action cancels the character on $H^1(\D^*,\underline{\Z/n}))$ computed above, hence $H_c^1(\D^*,\underline{\Z/n}))=\Z/n$. Finally, $H_c^2(\D^*,\underline{\Z/n})=(\Z/n(1))^\vee= \Z/n (-1)$. Taking a colimit over all $\Z/n\subset \Lambda$ yields $H_c^2(\D^*,\Lambda)=\Lambda(-1)=\abs{-}\otimes 1$ as a representation of $B\subset \GL_2(E)$. 

The abstract cohomology of $\underline{\Proj^1(E)}$ is supported in degree $0$ because the space is strictly totally disconnected after base change to $\spa C$. By the valuative criterion of $\Proj^1$ and universal property of formal schemes for $\spa(\mc O_E,\mc O_E)$, $\Proj^1(E)= \Proj^1(\mc O_E)\cong \varprojlim_n \Proj^1(\mc O_E/ \pi^n \mc O_E)$ is profinite. Then $H^0(\underline{\Proj^1(E)},\Lambda)= C^0(\Proj^1(E))= \varinjlim_n \bigoplus_{\Proj^1(\mc O_E/\pi^n \mc O_E)} \Lambda$ is a colimit of free $\Lambda$-modules, so in particular it is flat. 
Now $\underline{\Proj^1(E)}\cong \underline{B\backslash \GL_2(E)}$ has the regular right-action, so that $H^0(\underline{\Proj^1(E)},[\chi])=\ind_B^{\GL_2(E)}\chi$ for any character $\chi:B\to \Lambda^\times$. Again since $\underline{\Proj^1(E)}$ is strictly totally disconnected, the compactly supported cohomology is supported in degree $0$ and is the full module $\ind_{B}^{\GL_2(E)}\chi$ since each section has profinite support.

Putting all of the pieces together, the second page of the Leray spectral sequence is 
\begin{center}
\begin{tikzpicture}
  \matrix (m) [matrix of math nodes,
    nodes in empty cells,nodes={minimum width=8ex,
    minimum height=4ex,outer sep=-5pt},
    column sep=1ex,row sep=1ex]{
                
               & 0 & 0 & 0 &\\
              2 & \ind_B^{\GL_2(E)}(\abs{-}\otimes 1)& 0 & 0 & \\  
              1 & \ind_B^{\GL_2(E)}\Lambda & 0  & 0 &\\
             j= 0 &  0  & 0 & 0 &\\
    \quad\strut &  i= 0  &  1  &  2 &  \strut \\};
\draw[thick] (m-1-1.east) -- (m-5-1.east) ;
\draw[thick] (m-5-1.north) -- (m-5-5.north) ;
\end{tikzpicture}
\end{center}
and all differentials on all subsequent pages are $0$.
\end{proof}

\begin{proposition}
The compactly supported cohomology of $Z= \underline{\GL_2(E)}\cdot \Delta\BC(\mc O(1))$ as $\GL_2(E)$-modules is 
\[H_c^i(Z,\Lambda)=\begin{cases}  \st, & i=1\\ \ind_{B}^{\GL_2(E)} (\abs{-}\otimes 1),\quad &i=2\\0, & \text{else}. \end{cases}\]

\end{proposition}

\begin{proof}

We compute the excision long exact sequence on the compactly supported cohomology of $j_0:\underline{\Proj^1(E)}\times \D^* \hookrightarrow Z \hookleftarrow *:i_0$ to obtain the cohomology of $Z$:

\begin{center}
\begin{tikzcd}
0 \arrow[r]
& H_c^0(Z,\Lambda) \arrow[r]
\arrow[d, phantom, ""{coordinate, name=W}]
& \Lambda \arrow[dll,
rounded corners,
to path={ -- ([xshift=2ex]\tikztostart.east)
|- (W) [near end]\tikztonodes
-| ([xshift=-2ex]\tikztotarget.west)
-- (\tikztotarget)}] \\
\ind_{B}^{\GL_2(E)} \Lambda \arrow[r]
& H^1_c(Z,\Lambda) \arrow[r]\arrow[d, phantom, ""{coordinate, name=X}]
& 0 \arrow[dll,
rounded corners,
to path={ -- ([xshift=2ex]\tikztostart.east)
|- (X) [near end]\tikztonodes
-| ([xshift=-2ex]\tikztotarget.west)
-- (\tikztotarget)}] \\
\ind_B^{\GL_2(E)}(\abs{-}\otimes 1) \arrow[r]
& H_c^2(Z,\Lambda) \arrow[r]
& 0.
\end{tikzcd}
\end{center}

Note that $H_c^0(Z,\Lambda)=0$ since $H_c^0(Z,\Lambda)\subset H^0(Z,\Lambda)= \Lambda$ as $Z$ is connected, but $Z$ is not quasicompact. Then the map $\Lambda\hookrightarrow \ind_{B}^{\GL_2(E)}\Lambda$ must be the embedding via the constant functions since it is $\GL_2(E)$-equivariant, and hence $H_c^1(Z,\Lambda)=\st$ is the Steinberg representation.
\end{proof}

\begin{proposition}\label{cs-coh-double-tilde-M-trivial-module}
The compactly supported cohomology of $\widetilde{\widetilde{\mc M_{b_2}^{b_1}}}\cong U \times \underline{E^\times}$ as $G_{b_1}(E)$-modules is 
\[H_c^i(\widetilde{\widetilde{\mc M_{b_2}^{b_1}}},\Lambda) =\begin{cases} \st \otimes \cind_{\SL_2(E)}^{\GL_2(E)}\Lambda ,\quad & i=2\\
\left(\ind_{B}^{\GL_2(E)}\abs{-}\otimes 1\right)\otimes \cind_{\SL_2(E)}^{\GL_2(E)} \Lambda ,\quad & i=3\\
\abs{\det} \otimes \cind_{\SL_2(E)}^{\GL_2(E)} \Lambda,\quad & i=4\\ 
0,\quad & \text{else}.
\end{cases}
\]
\end{proposition}
\begin{remark}
    We describe the $G_{b_2}(E)$-module structure more explicitly later in terms of the colimit as $K\to \set{1}\subset \GL_2(E)$ of the modulo $K$ compactly supported cohomology in Proposition \ref{gluing-compact-generators}.
\end{remark}
\begin{proof}
The excision long exact sequence for $j: U\hookrightarrow \D^2 \hookleftarrow Z: i$ induces short exact sequences
\[0\to H_c^0(U,\Lambda)\to 0,\]
\[0 \to H_c^1(U,\Lambda) \to 0,\]
\[0 \to \st \to H_c^2(U,\Lambda)\to 0,\]
\[0\to \ind_B^{\GL_2(E)}(\abs{-}\otimes 1)\to H_c^3(U,\Lambda)\to 0,\]
and 
\[0\to H_c^4(U,\Lambda) \to \Lambda(-2) \to 0.\]

In degree $4$, the $\GL_2(E)$-action on $\Lambda(-2)$ is a character of $\GL_2(E)$, so it necessarily factors through $\det:\GL_2(E)\to E^\times$. Then since $\begin{bmatrix}\pi & 0 \\ 0 & \pi \end{bmatrix}$ acts on $\D^2$ as the inverse of geometric Frobenius, the corresponding action on sheaves is by arithmetic Frobenius, yielding the representation $\abs{\det}:\pi\cdot I\mapsto q^{-2}\in \Aut(\Lambda)$ on $\Lambda(-2)$.

The compactly supported cohomology of $\underline{E^\times}$ is the compactly supported continuous functions $E^\times \to \Lambda$ sitting in degree $0$. Since $\underline{\GL_2(E)}$ acts on the domain of these functions as multiplication by $\det:\GL_2(E)\to \GL_2(E)/\SL_2(E)\cong E^\times$, the $\GL_2(E)$-module $H_c^0(\underline{E^\times},\Lambda)$ is $\cind_{\SL_2(E)}^{\GL_2(E)}\Lambda$. As in the proof of Lemma \ref{tensor-sl2-ind}, $\cind_{\SL_2(E)}^{\GL_2(E)}\Lambda \cong \varinjlim \bigoplus_{q\in Q_i}\Lambda$ is a flat $\Lambda$-module, hence the Künneth formula implies the compactly supported cohomology of $U\times \underline{E^\times}$ is the same as that of $U$ tensored with $\cind_{\SL_2(E)}^{\GL_2(E)}\Lambda$ in each degree.

\end{proof}

\subsection{The case of unramified characters}
We specialize now to the trivial representation on $\Bun_2^1=[*/\underline{\GL_2(E)}]$ and its unramified twists.

\begin{theorem}\label{cs-coh-tilde-M-trivial-module}
The compactly supported cohomology of $\widetilde{\mc M_{b_2}^{b_1}}$ as $E_1^\times \times E_2^\times$-modules is 
\[H_c^i(\widetilde{\mc M_{b_2}^{b_1}},\Lambda) =\begin{cases} 
\Lambda,\quad & i=1\\
\delta_T,\quad & i=4\\
0,\quad & \text{else,}
\end{cases}
\]
where $\delta_T$ is the torus character $(t_1,t_2)\mapsto \abs{t_2\cdot t_1^\inv}\in q^{\Z} \subset \Aut(\Lambda)$.
\end{theorem}
\begin{proof}
We will make use of the spectral sequence 
\begin{align}\label{double-tilde-descent-SS}
E_2^{i,j}=H_{-i}(\GL_2(E), H_c^j(\widetilde{\widetilde{\mc M_{b_2}^{b_1}}},\Lambda))\implies H_c^{i+j}(\widetilde{\mc M_{b_2}^{b_1}},\Lambda)
\end{align}
from Proposition \ref{homological-Hochschild-Serre}. Then we are interested in computing the $\GL_2(E)$-homology of the representations $\st\otimes \cind_{\SL_2(E)}^{\GL_2(E)}\Lambda$, $(\ind_{B}^{\GL_2(E)}\abs{-}\otimes 1)\otimes \cind_{\SL_2(E)}^{\GL_2(E)}\Lambda$, and $\abs{\det}\otimes \cind_{\SL_2(E)}^{\GL_2(E)}\Lambda$. For this we appeal to the homological Hochschild-Serre spectral sequence
\[E_2^{p,q}= H_{-p}(E^\times, H_{-q}(\SL_2(E),\sigma))\implies H_{-p-q}(\GL_2(E),\sigma).\]

First we claim that $H_{i}(\SL_2(E),\cind_{\SL_2(E)}^{\GL_2(E)}\sigma)=\cind_{1}^{E^\times}(H_i(\SL_2(E),\sigma))$ as $\Lambda$-modules. As an $\SL_2(E)$-representation, $\cind_{\SL_2(E)}^{\GL_2(E)}\sigma \cong \varinjlim \bigoplus_{H_k\backslash \GL_2(E)} \sigma$ is a filtered colimit (see Lemma \ref{tensor-sl2-ind} for the notation of $H_k$), which is both exact and preserves projectives (since all of the transition maps are given by $1$'s and $0$'s between finite direct powers of a projective $P$), and hence commutes with $\SL_2(E)$-homology in the Hochschild-Serre spectral sequence. The induced $E^\times$-action can be understood on lifts of the determinant $\GL_2(E)\to \SL_2(E)\backslash \GL_2(E)$, which acts by permuting the $\SL_2(E)$-cosets in the compact induction, so the action on the indices of the direct summands in the filtered colimit corresponds to $\cind_1^{E^\times}(-)$. In general, there will also be an induced $E^\times$-character $\chi_{\det}$ coming solely from $\sigma$ and acting diagonally, so that we have
\[H_{i}(\SL_2(E),\cind_{\SL_2(E)}^{\GL_2(E)}\sigma)= \chi_{\det} \cdot \cind_{1}^{E^\times}(H_i(\SL_2(E),\sigma))\]
as $E^\times$-modules.
\begin{claim}\label{GL2-homology-det-module-deg4}
The $\GL_2(E)$-homology of $\abs{\det}\otimes \cind_{\SL_2(E)}^{\GL_2(E)}\Lambda$ as $E_1^\times \times E_2^\times$-modules is 
$\delta_T$
in degree $0$.
\end{claim}
Now $\abs{\det}\otimes \cind_{\SL_2(E)}^{\GL_2(E)}\Lambda\cong \cind_{\SL_2(E)}^{\GL_2(E)}\Lambda$ as $\SL_2(E)$-modules and the $\SL_2(E)$-(co)homology of $\Lambda$ is a single copy of $\Lambda$ in degree $0$ (c.f. the computation in example \ref{sl2-cohomology-steinberg}, Proposition \ref{dual-of-homology}, and note that $\Lambda$ is self-dual), so that the resulting $E^\times$-module is $\abs{\det}\otimes \cind_{1}^{E^\times}\Lambda$. Proposition \ref{homology-E-times} computes the $E^\times$-homology to be the $E^\times$-coinvariants in degree $0$ and the $\pi^{\Z}$-invariants of the module $\abs{\det}\otimes (\cind_1^{E^\times}\Lambda)_{\mc O_E^\times}\cong \bigoplus_{\pi^{\Z}}\abs{\det}$ in degree $-1$. The latter $\Lambda$-module is zero since $\pi$-invariance implies that any nonzero function is supported on a countable discrete set, which is not compact. For each level $U_k:= 1+\pi^k \mc O_E$, there are isomorphisms
\begin{align*}
\abs{\det}\otimes \cind_1^{E^\times}\Lambda\cong \abs{\det}\otimes \varinjlim \cind_{U_k}^{E^\times}\Lambda & \cong \varinjlim \abs{\det}\otimes \bigoplus_{U_k\backslash E^\times}\Lambda \\ &\cong \varinjlim  \bigoplus_{U_k \backslash E^\times}\abs{\det},
\end{align*}
where the $E^\times$-action on the direct sum is by permuting the factors. Since taking coinvariants commutes with filtered colimits, it suffices to compute the $E^\times$-coinvariants on $\bigoplus_{U_k\backslash E^\times}\abs{\det}$. Any equivalence class of a nonzero function in the coinvariants can be represented by a function $U_k\backslash E^\times \to \Lambda$ that takes a nonzero value at $1\in U_k\backslash E^\times$. If $f_{\overline{z}}$ denotes the indicator function at $\overline{z}\in U_k\backslash E^\times$, then $0=z^\inv \cdot f_1 -f_1= \abs{z}^\inv f_{\overline{z}}-f_1$ in the module of coinvariants. Then as any compactly supported function on $U_k\backslash E^\times$ has finite support, subtracting finitely many multiples of the functions above gives a representative with support $\set{1}\subset U_k\backslash E^\times$. We deduce immediately that the coinvariants of the filtered colimit are $\Lambda$, identified as functions supported on $1\in E^\times$. In particular, the spectral sequence degenerates immediately and yields \[H_*(\GL_2(E),\abs{\det}\otimes\cind_{\SL_2(E)}^{\GL_2(E)}\Lambda)= \Lambda[0]\]
as an abstract $\Lambda$-module. 

Recall that the character $\abs{\det}$ was identified as Frobenius equivariance for the $\mc G_{b_1}$-action on $U\subset \BC(\mc O(1))^2$ and the Tate twist $\Lambda(-2)$. Then the character $\abs{\det}$ induces $\abs{-}_1^{-2}$ because the geometric action of $\underline{E_1^\times} \times \underline{E_2^\times}$ on $U$ is trivial for the second factor, and the first factor acts as $z_1\mapsto \begin{bmatrix}z_1^\inv & \\ & z_1^\inv \end{bmatrix}$. Furthermore, the geometric actions of $\underline{E_i^\times}$ on $\underline{E^\times}$ agree with the inverse of the determinant action of $\GL_2(E)$, so that $z_i\in E_i^\times$ shift the domain of the indicator functions above in the opposite way. 
Then we compute for $z_1\in E_1^\times$ and $z_2\in E_2^\times$:
\[z_1\cdot f_1= \abs{z_1}^{-2}f_{z_1}=\abs{z_1}^\inv f_1,\]
and
\[z_2\cdot f_1= f_{z_2^{}} = \abs{z_2}f_1,\]
which is the character $\delta_T$.

\begin{claim}
The $\GL_2(E)$-homology of $\left(\ind_B^{\GL_2(E)}\abs{-}\otimes 1\right)\otimes \cind_{\SL_2(E)}^{\GL_2(E)}\Lambda$ vanishes in all degrees.
\end{claim}
When restricted to $\SL_2(E)$, the torus characters $\abs{-}\otimes 1$ and $\delta_T^{-1/2}$ agree. Also $\Proj^1(E)\cong B\backslash \SL_2(E)= B\backslash \GL_2(E)$ (recall our convention that $B$ means the standard Borel of either group) and we get an isomorphism of $\SL_2(E)$-representations $\ind_{B}^{\GL_2(E)}\abs{-}\otimes 1\cong \ind_{B}^{\SL_2(E)}\delta_T^{-1/2}$.
Then applying Lemma \ref{tensor-sl2-ind} to $\left(\ind_B^{\GL_2(E)}\abs{-}\otimes 1\right)\otimes \cind_{\SL_2(E)}^{\GL_2(E)}\Lambda$ yields the $\SL_2(E)$-module $\cind_{\SL_2(E)}^{\GL_2(E)}\ind_{B}^{\SL_2(E)}\delta_T^{-1/2}$. But $\ind_{B}^{\SL_2(E)}\delta_T^{-1/2}$ is the unramified principal series representation $\mathrm{PS}(1)$ corresponding to the trivial character, so its $\SL_2(E)$-cohomology vanishes in every degree \cite[Theorem~X.4.3]{borel2000continuous}.

\begin{claim}
The $\GL_2(E)$-homology of $\st\otimes \cind_{\SL_2(E)}^{\GL_2(E)}\Lambda$ as $E_1^\times \times E_2^\times$-modules is 
the trivial representation $\Lambda$ in degree $-1$.
\end{claim}
As before, $\st\otimes \cind_{\SL_2(E)}^{\GL_2(E)}\Lambda\cong \cind_{\SL_2(E)}^{\GL_2(E)}\st$ as $\SL_2(E)$-modules, and so its $\SL_2(E)$-(co)homology is \[\cind_{1}^{E^\times} H_{\bullet}(\SL_2(E),\st)= \cind_{1}^{E^\times}\Lambda[1].\]
Again, we are making use of the computation of the smooth $\SL_2(E)$-cohomology of $\st$ in \ref{sl2-cohomology-steinberg}, smooth homological duality in Proposition \ref{dual-of-homology}, and the fact that both $\Lambda$ and $\st$ are self-dual. Note that the determinant character on $H_{1}(\SL_2(E),\st)=\Lambda$ is trivial, since the first homology group is (the dual of) the module of functions that are constant on $\mc O_E\subset \Proj^1(E)$ and vanish everywhere else, and the induced determinant action is by precomposing functions with the regular right action on $\Proj^1(E)= B\backslash \GL_2(E)=B\backslash \SL_2(E)$. Thus the $\GL_2(E)$-homology is computed by 
\[H_{\bullet}(E^\times, \cind_{1}^{E^\times}\Lambda[1])=\Lambda[1].\]
Since there is no determinant character on the compact-induction to $E^\times$, the identification of $\Lambda=(\cind_{1}^{E^\times}\Lambda)_{E^\times}$ as functions supported on a point of $E^\times$ is not twisted in the sense of the computation in claim \ref{GL2-homology-det-module-deg4}; that is to say that the induced $E_1^\times \times E_2^\times$-action is indeed trivial. 

Putting all of these results together allows us to write the second page of the descent spectral sequence in (\ref{double-tilde-descent-SS}):

\begin{center}
\begin{tikzpicture}
  \matrix (m) [matrix of math nodes,
    nodes in empty cells,nodes={minimum width=8ex,
    minimum height=3ex,outer sep=-5pt},
    column sep=1ex,row sep=1ex]{
                &   0   &   0  &  0   & \delta_T   & 4 \\
               & 0 & 0 & 0 & 0 & 3\\
               &  0 & 0 & \Lambda & 0 & 2\\  
               &  0 &  0  & 0 & 0 & 1\\
               &  0  & 0 &  0  & 0 & 0\\
    \quad\strut &   -3  &  -2  &  - 1  & 0 & \strut \\};
\draw[thick] (m-1-6.west) -- (m-6-6.west) ;
\draw[thick] (m-6-1.north) -- (m-6-6.north) ;
\end{tikzpicture}
\end{center}
Then since all of the differentials are $0$ on every page, we conclude that the compactly supported cohomology of $\widetilde{\mc M_{b_2}^{b_1}}$ is $\Lambda$ in degree $1$ and $\delta_T$ in degree $4$.
\end{proof}

The computation above easily generalizes to representations $\sigma = \chi \circ \det$ for characters $\chi : E^\times \to \Lambda$ since the $\SL_2(E)$-homology is unaffected. 
\begin{corollary}\label{determinant-character-twists}
Let $\chi$ be any character on $E^\times$. Then
\[H_c^i(\widetilde{\mc M_{b_2}^{b_1}},\chi\circ \det) =\begin{cases} 
\chi\otimes \chi,\quad & i=1\\
(\chi\otimes \chi)\cdot \delta_T,\quad & i=4\\
0,\quad & \text{else.}
\end{cases}
\]
In particular, $R\Gamma_c(\widetilde{\mc M_{b_2}^{b_1}},\abs{\det}^k)=\abs{-}^k [-1]\oplus \abs{-}^k\delta_T [-4]$.
\end{corollary}
\begin{proof}
    This follows from the computation of the $\underline{E_1^\times}\times\underline{E_2^\times}$ characters induced on the smooth homology $H_{-i}(E^\times, \chi \otimes \cind_1^{E^\times}\Lambda)=_{\Lambda-\mathrm{Mod}} \Lambda[0]$ for an arbitrary determinant character $\chi$. As in the proof of the theorem, one sees that the module of coinvariants is identified with functions supported on $1\in E^\times$ with 
    \[0=z^\inv f_1 -f_1 = \chi(z^\inv) f_z -f_1.\]
    Then the action for $z_1\in E_1^\times$ and $z_2\in E_2^\times$ can be computed as follows:
    \[z_1\cdot f_1 = f_{z_1}= \chi(z_1) f_1\]
    and
    \[z_2\cdot f_1 = f_{z_2}= \chi(z_2) f_1,\]
    since the actions of $z_i$ only act on the domain of the functions, but the identification of indicator functions is twisted by $\chi$.
    The factor of $\delta_T$ appears in degree 4 since $\chi$ is replaced with the $G_{b_1}(E)$-rep $\chi \cdot \abs{\det}$. Note that the representation $\abs{\det}$ really arises from a Tate twist $\Lambda(-2)$, so that the induced $G_{b_2}(E)$-module structure $\abs{-}_1^{-2}$ is fundamentally different from that from $\chi$, which arises from an invertible sheaf over $[*/\underline{G_{b_1}(E)}]$.
    
\end{proof}

\begin{theorem}\label{Purity-for-trivial-rep}
Let $k$ be an integer (more generally, $k\in \frac{1}{2}\Z$ in case there is a fixed $\sqrt{q}\in \Lambda$). Then 
\[i_{b_2}^* Ri_{b_1,*} \abs{\det}^{k} = \abs{-}^{k}\oplus \abs{-}^{k}\delta_T [-3].\]
In particular, $i_{b_2}^* Ri_{b_1,*} \Lambda = \Lambda[0] \oplus \delta_T [-3]$, which verifies a version of cohomological purity.
\end{theorem}
\begin{proof}
Recall that the map $\mc M_{b_2}\to [*/\underline{G_{b_2}(E)}]$ is cohomologically smooth of dimension $2$ \cite[Proposition~V.3.5]{fargues2021geometrization}. Furthermore, the inclusion $i_{b_1}:[*/\mc G_{b_1}]\hookrightarrow \Bun_2$ is an open substack since $b_1$ is a basic point, hence it is cohomologically smooth of dimension $0$. Then since $\widetilde{\mc M_{b_2}^{b_1}}\to *$ is a composition of pullbacks of these maps along $v$-covers, it is also cohomologically smooth of dimension $2$. Now we can use Poincar\'{e} duality and the compactly supported cohomology from Theorem \ref{cs-coh-tilde-M-trivial-module} to compute the cohomology of $\widetilde{\mc M_{b_2}^{b_1}}$. 

Note that in this case, $\widetilde{\mc M_{b_2}^{b_1}}$ is the punctured negative Banach-Colmez space $\BC(\mc O(-2)[1])$ parameterizing nonsplit extensions of $\mc O(1)$ by $\mc O(-1)$, so that the dualizing sheaf of $\widetilde{\mc M_{b_2}^{b_1}}\to *$ is $\Lambda(2)[4]$ by Lemma \ref{negative_bc_dualizing_sheaf}. By Poincaré duality, we want to compute $H^i(\widetilde{\mc M_{b_2}^{b_1}},\Lambda)\cong H_c^{4-i}(\widetilde{\mc M_{b_2}^{b_1}},\Lambda(2))^\vee,$ and we can use the descent spectral sequence from Proposition \ref{homological-Hochschild-Serre} to obtain $R\Gamma_c(\widetilde{\mc M_{b_1}^{b_2}}, \Lambda(2))$ from the smooth $\GL_2(E)$-homology of $R\Gamma_c(\widetilde{\widetilde{\mc M_{b_1}^{b_2}}}, p^* \Lambda (2))\cong \chi \otimes R\Gamma_c(\widetilde{\widetilde{\mc M_{b_1}^{b_2}}}, \Lambda)$ for some character $\chi_d$ of $\GL_2(E)$ (see remark \ref{pulling-out-reps}). Since $\chi_d = \abs{\det}^i$ for some $i$, the $\SL_2(E)$-homology is unaffected, and it suffices to understand the character for lifts of determinant as in the proof of Corollary \ref{determinant-character-twists}. The element \begin{small}$\begin{bmatrix}
    \pi & 0\\ 0& \pi
\end{bmatrix}$
\end{small} acts as inverse of geometric Frobenius on the $\BC(\mc O(1))^2\setminus Z$-factor and translation by $\pi^2$ on the $\underline{E^\times}$-factor, so that the Tate twist $\Lambda(2)$ induces the automorphism
$q^2= \abs{\pi^2}^\inv$ on compactly supported cohomology of the first factor, and we conclude $\chi_d = \abs{\det}^\inv$ by the Künneth formula. Note that this character is induced by Tate twist and is precisely the inverse character that appears in the compactly supported cohomology of the structure sheaf on $\widetilde{\widetilde{\mc M_{b_2}^{b_1}}}$ in degree $4$; then the result follows by the same arguments in Theorem \ref{cs-coh-tilde-M-trivial-module} and inverting the resulting characters when dualizing. 

The computation for arbitrary characters $\chi:= \chi\circ \det$ follows from the derived adjunction \[R\Gamma \Hom(\Lambda, [\chi])\cong R\Gamma \Hom([\chi]^\inv \otimes f^! \Lambda, f^!\Lambda)\cong \Hom(Rf_!([\chi]^\inv \otimes \Lambda(2)[4]),\Lambda),\]
where $f:\widetilde{\mc M_{b_2}^{b_1}}\to *$ is the structure map. Now the result follows immediately from Corollary \ref{determinant-character-twists}.
\end{proof}

\subsection{Compact generators and orbits of congruence subgroups on the projective line}
We consider the gluing functors for the collection of compact generators $\cind_K^{\GL_2(E)}\Lambda\in D_{\et}([*/\underline{\GL_2(E)}],\Lambda)$. Ultimately, the computations reduce to the group homology of the restrictions to $K$ of the representations computed for the trivial module. This requires knowledge of the orbits for the action of the congruence subgroups $\Gamma_k:=\ker(\SL_2(\mc O_E)\to \SL_2(\mc O_E/\pi^k))$ on $\Proj^1(E)$ by Möbius transformations. By descent, the modules computed in this subsection are also the compactly supported cohomology of $\widetilde{\widetilde{\mc M_{b_2}^{b_1}}}/\underline{K}$.

\begin{lemma}
    The orbits of the action of $\Gamma_k$ on $\Proj^1(E)$ are in bijective correspondence to the points of 
    \[\set{\infty}\cup \set{0} \cup \bigcup_{i=-k+1}^{k-1}\mc O_E^\times/\left(1+\pi^{\min\set{k+i,k-i}}\mc O_E\right).\]
\end{lemma}
\begin{proof}
    The orbits are precisely the double cosets $\Proj^1(E)/\Gamma_k=B\backslash \SL_2(E)/\Gamma_k$. The following is a variation of the computation for $\Gamma_0$ in Example \ref{sl2-cohomology-steinberg}.

The orbit of $\infty$ is computed by considering an arbitrary product
    \[\begin{bmatrix}\alpha & \beta\\ 0 & \alpha^\inv \end{bmatrix}\begin{bmatrix}1 & 0\\ 0 & 1\end{bmatrix}\begin{bmatrix}1+\pi^k a& \pi^k b\\ \pi^k c & 1+\pi^k d\end{bmatrix}=\begin{bmatrix}\alpha (1+a\pi^k)+\pi^k \beta c & \alpha b \pi^k +\beta (1+d\pi^k)\\ \alpha^\inv \pi^k c & \alpha^\inv (1+d\pi^k)\end{bmatrix},\]
    with $\alpha,\beta\in E$ and $a,b,c,d\in \mc O_E$. Whenever $c\neq 0$, choosing $\alpha= \pi^k c$ and $\beta = -(1+a\pi^k)$ yields
    \[\begin{bmatrix}0 & -1\\ 1 & \frac{1+d\pi^k}{\pi^k c}\end{bmatrix}.\] Then since $1+d\pi^k\in \mc O_E^\times$ is a unit, $\frac{1+d\pi^k}{\pi^k c}$ has valuation $\leq -k$. Choosing $d=0$, we see that the orbit of $\infty $ contains all of the points $[1:t]$ where the valuation of $t$ is $\leq -k$.\\
    
    Now we consider the orbit of $[1:t]$ with $t\in \pi^{1-k} \mc O_E$:
    \begin{align*}
    \begin{bmatrix}\alpha & \beta\\ 0 & \alpha^\inv \end{bmatrix}\begin{bmatrix}0 & -1\\ 1 & t\end{bmatrix}\begin{bmatrix}1+\pi^k a & \pi^k b\\ \pi^k c & 1+\pi^k d\end{bmatrix}&=\\
    & \begin{bmatrix}-\pi^k \alpha c +\beta (1+\pi^k a +t\pi^k c) & -\pi^k \alpha d +\beta(\pi^k b+t(1+\pi^kd)) \\ \alpha^\inv (1+\pi^k a+t\pi^k c) & \alpha^\inv (\pi^k b+t(1+\pi^k d))\end{bmatrix},
    \end{align*}
    and choosing $\alpha = (1+\pi^k a+t\pi^k c)\in \mc O_E^\times$ is allowed. Furthermore set $\beta = \pi^k c$, so that the matrix becomes 
    \[\begin{bmatrix}0 & -1\\ 1 & (t+\pi^k(b+td))/(1+\pi^k (a+tc))\end{bmatrix}.\]
    Note that since the only choices so far were made on $\alpha$ and $\beta$, the $\Gamma_k$-orbit will not contain $\infty$ and is hence disjoint. We claim that this orbit is transitive on $\pi^k\mc O_E= [1:\pi^k\mc O_E]\subset \Proj^1(E)$. Choose $c=0$ and $a=0$, so that $d$ must be $0$. Then $b\in \mc O_E$ can be picked freely, so that $[1:t+\pi^k b]=[1:s]$ is in the orbit for any $s\in \pi^k\mc O_E$ as soon as $t\in \pi^k \mc O_E$; in particular, this is the orbit of $0$.
    
    This leaves the subsets $t\in \pi^{i}\mc O_E^\times$ for $i=-k+1,\dots,k-1$. By assumption, the valuation of any translate of $t$ is constant, so there is at least one distinct orbit for each $i$. Write $t= \pi^i s$ with $s\in \mc O_E^\times$, hence the orbit of $t$ contains all of the points of the form 
    \[\pi^i s \cdot \frac{(1+\pi^{k-i}bs^\inv+\pi^k d)}{(1+\pi^{k+i}cs+\pi^k a)}.\] 
    When $i\geq 0$, these orbits correspond to the elements of $\pi^i (\mc O_E^\times /1+\pi^{k-i}\mc O_E)$ by choosing $a=d=c=0$ and observing that $b\in \mc O_E$ is free. Dually, when $i<0$ the orbits correspond to the elements of $\pi^i(\mc O_E^\times / 1+\pi^{k+i}\mc O_E)$ by setting $a=d=b=0$ and now $c\in \mc O_E$ is free.
    Note that as $k$ goes to infinity, the limit of the finitely many discrete orbits is the profinite space $\Proj^1(E)$.
\end{proof}

\begin{proposition}\label{gluing-compact-generators}
Let $K\subset \GL_2(E)$ be a compact open subgroup contained in the maximal compact open $\GL_2(\mc O_E)$ and such that $K\cap \SL_2(\mc O_E)= \Gamma_k$ is the level-$k$ congruence subgroup. Then as $E_1^\times \times E_2^\times$-modules,
\[H_c^i({\widetilde{\mc M_{b_2}^{b_1}}}
,\cind_K^{\GL_2(E)}\Lambda) =\begin{cases} \bigoplus_{\pi^{\Z}}C^0(\mc O_E^\times/K, \Lambda)\otimes \left (C^0(\Proj^1(E)/\Gamma_k, \Lambda)\right )/\Delta, & i=2\\
\bigoplus_{\pi^{\Z}}C^0(\mc O_E^\times/K,\abs{-}_1^\inv \otimes 1) \otimes C^0(\Proj^1(E)/\Gamma_k,\Lambda), & i=3\\
\bigoplus_{\pi^{\Z}}C^0(\mc O_E^\times/K,\abs{-}_1^{-2}\otimes 1), & i=4\\ 
0, & \text{else},
\end{cases}
\]
where we identify the finite set $\SL_2(\mc O_E)\backslash \GL_2(\mc O_E) /K\cong \mc O_E^\times /K$, and $\Delta$ is the diagonal submodule of constant functions on $\Proj^1(E)/\Gamma_k$. The action of $\pi_i\in E_i^\times$ translates the countable $\pi^{\Z}$-factors, the action of $u_i\in \mc O_E^\times \subset E_i^\times$ acts on the domain of functions on the finite discrete space $\mc O_E^\times/K$, while the action of $u\in \mc O_E^\times \subset E_1^\times$ acts on the right tensor factors via the regular right action on the continuous $\Lambda$-valued functions on $\Proj^1(E)$.
\end{proposition}
\begin{remark}
Note that such $K$ form a neighborhood basis of $1\in \GL_2(E)$, and the colimit over all $K$ recovers $H^i_c(\widetilde{\widetilde{\mc M_{b_2}^{b_1}}},\Lambda)$ as an $E_1^\times \times E_2^\times$-module. This allows us to think of $\varinjlim_K H^i_c(\widetilde{\widetilde{ \mc M_{b_2}^{b_1}}}/K, \Lambda)$ in terms of modules of compactly supported functions on $\Proj^1(E) \times E^\times= (\Proj^1 \times \mathbb{G}_m )(E)$.
\end{remark}
\begin{proof}
The Hochschild-Serre descent spectral sequence from the compactly supported cohomology of $\widetilde{\widetilde{\mc M_{b_2}^{b_1}}}$ computes the $\GL_2(E)$-homology of $\cind_K^{\GL_2(E)}\Lambda \otimes^\mathbb{L}_\Lambda R\Gamma_c (\widetilde{\widetilde{\mc M_{b_2}^{b_1}}},\Lambda)$, which is canonically identified with the $K$-homology of the restriction to $K$ of $R\Gamma_c (\widetilde{\widetilde{\mc M_{b_2}^{b_1}}},\Lambda)$. Furthermore, the spectral sequence degenerates immediately since $K$ is profinite with pro-order prime to the torsion of $\Lambda$. Then via the Künneth formula, it suffices to compute the module of $K$-coinvariants for $\st$, $\ind_B^{\GL_2(E)}\abs{-}\otimes 1$, $\abs{\det}$, and $\cind_{\SL_2(E)}^{\GL_2(E)}\Lambda$.

When restricted to $K$, these become $\st, \ind_B^{\GL_2(E)}\Lambda,\Lambda$, and $\bigoplus_{\pi^{\Z}}\cind_{\SL_2(\mc O_E)}^{\GL_2(\mc O_E)}\Lambda$. The latter identification is by rescaling an arbitrary element of $\GL_2(E)$ by $\pi^i$ and \begin{small}$\begin{bmatrix}\pi^j & 0\\ 0&1\end{bmatrix}$\end{small}, so that the direct summands correspond to compactly supported functions on invertible matrices with the same fixed valuation of their determinants. First note that $\SL_2(E)\backslash \GL_2(E) \cong \bigsqcup_{\pi^{\Z}}\SL_2(\mc O_E) \backslash \GL_2(\mc O_E)$. Then since any two compact open subgroups are commensurable and since $K\subset \GL_2(\mc O_E)$ is open, the $K$-coinvariants are a direct sum over $\pi^{\Z}$ with summands identified as the functions on the finite discrete set $(\SL_2(\mc O_E)\backslash \GL_2(\mc O_E))/K$ with $r$ elements. Note that the topological limit of these spaces as $K$ approaches $\set{1}$ is $\bigsqcup_{\pi^{\Z}}\mc O_E^\times\cong E^\times$.

For the $K$-coinvariants of $\st$ and $\ind_B^{\GL_2(E)}\Lambda$, note that $u$ and \begin{small}$\begin{bmatrix}u & 0\\ 0&1\end{bmatrix}$\end{small} for a unit $u\in \mc O_E^\times$ are both in the standard Borel subgroup, so we may instead compute the $\Gamma_k$-coinvariants of $\st$ and $\ind_B^{\SL_2(E)}\Lambda$. The latter module is precisely the functions on $\Proj^1(E)$ that are constant on $\Gamma_k$ orbits, since one identifies the $\Gamma_k$-coinvariants with the $\Lambda$-dual of the $\Gamma_k$-invariants of the contragredient representation. The $\Gamma_k$-coinvariants of the Steinberg module are similarly described as functions on the $\Gamma_k$-orbits in $\Proj^1(E)$, modulo the constant functions.
\end{proof}

\subsection{Unramified principal series representations and homology of the Borel subgroup}
We compute the image of smooth parabolic inductions of unramified characters, which follows from computing the smooth homology of the representations appearing in Proposition \ref{cs-coh-double-tilde-M-trivial-module} after restriction to the Borel subgroup. The computations fundamentally reduce to identifying the Jacquet modules of $\GL_2(E)$-representations, then taking homology with respect to the induced torus action. The images here are the fundamental point of interest for the semi-orthogonal decomposition, and are philosophically a categorical local analogue of constant terms of fourier expansions of Eisenstein series. For the remainder of this section, suppose there exists a fixed $\sqrt{q}\in \Lambda^\times$.

\begin{lemma}\label{homology-of-tensor-with-parabolic-induction}
Suppose $\sigma$ is a smooth $\GL_2(E)$-representation and $\chi$ is a character of $T\subset \GL_2(E)$. There are natural isomorphisms 
\[H_i(\GL_2(E),\left(\ind_B^{\GL_2(E)}\chi\right)\otimes \sigma)\cong H_i(B,\delta_T\otimes \chi\otimes \res^{\GL_2(E)}_B\sigma).\]
as $\Lambda$-modules.
\end{lemma}

\begin{proof}
Consider the following factorization:

\begin{equation*}
\begin{tikzcd}
    {[}{*} /\underline{\GL_2(E)}{]}\arrow[d,"f",swap] & {[}{*}/\underline{B}{]}\arrow[l,"\iota",swap]\arrow[dl,"g"]\\
    {*} & 
\end{tikzcd}
\end{equation*}
The functors we want to compute are $Rf_! (R\iota_!\chi \otimes \sigma)\cong R f_! R\iota_!(\chi \otimes \iota^* \sigma)\cong R g_! (\chi \otimes \iota^* \sigma)$ by the projection formula. Now, $R g_!$ differs from smooth $B$-homology by twisting the representation by the modulus sheaf $\delta_T$ of $B$ (i.e. the sheaf of Haar measures on $B$), $R \iota_!$ identifies with non-normalized parabolic induction, and $Rf_!$ identifies with smooth $\GL_2(E)$-homology since $\GL_2$ is unimodular.

To see that $Rg_!$ differs from smooth $B$-homology by a modulus twist, we recall the derived adjunctions between $D_{\mathrm{lis}}(-,\Lambda)\cong D_{\et}(-,\Lambda)$ categories, $(f_\natural, f^*)$ and $(f_!, f^!)$. Then since $f:[*/\underline{B}]\to *$ is $\ell$-cohomologically smooth, for any complexes $\sigma \in D(B)\cong D_{\et}([*/B])$ and $W\in D(\Lambda)\cong D_\et(*)$,
\[ R\Hom_\Lambda(g_!(\sigma),W) \cong R \Hom_B (\sigma, g^! W)\cong R \Hom_B (\sigma, \delta_T^{-1}\otimes g^*W)\cong R \Hom_\Lambda (g_\natural(\delta_T \otimes \sigma ), W).\]

\end{proof}
\begin{lemma}[Geometric lemma {\cite[\hphantom~9.3]{Bushnell2006}}]\label{Jacquet-of-induced-rep}
    There is a short exact sequence of $T$-representations
    \[0 \to \chi^w \otimes \delta_T^{\inv}\to \mathrm (\ind_B^{\GL_2(E)}\chi)_N \to \chi \to 0,\]
    where $(\chi_1\otimes \chi_2)^w= \chi_2\otimes \chi_1$ denotes the Weyl group action on characters of $T$ and $N\subset B$ is the unipotent subgroup. The map on the right is given by the evaluation at $1\in \GL_2(E)$, $f\mapsto f(1)$, and the kernel is the $N$-coinvariants of the subspace of functions supported on the Bruhat cell $BwN$ for the decomposition $\GL_2(E)=B\sqcup BwN$. In particular, there is a short exact sequence 
    \[0 \to (\chi \otimes \delta_T^{ 1/2})^w\to \mathrm (\ind_B^{\GL_2(E)}(\chi \otimes \delta_T^{-1/2}))_N \to \chi \otimes \delta_T^{-1/2} \to 0.\]
\end{lemma}
A corollary of exactness of $(-)_N$ is that $\st_N\cong \delta_T^{-1}$, which follows after showing the first sequence is split for $\chi= \Lambda$. Furthermore, it is known that the second sequence is split whenever $\chi^w\neq \chi$ (as $T$ is abelian), and is non-split in case $\chi^w = \chi$. 
\begin{lemma}\label{cohomology-PS(1)}
	Let $T_s$ denote the intersection $T\cap \SL_2(E)$ and let $\chi$ be an unramified character of $T$. Then the smooth $T_s$-homology of $\chi \cdot (\ind_B^{\GL_2(E)}\delta_T^{-1/2})_N$ vanishes in all degrees unless $\chi|_{T_s}= \delta_T^{1/2}|_{T_s}=\abs{-}^{-1}$, where the last identification is obtained by the coordinate in the upper-left corner of $T_s\cong E^\times$. In this case, the homology is $\Lambda$ in degrees $0$ and $-1$.  
\end{lemma}
\begin{proof}
	From Proposition \ref{homology-E-times}, the homology of any smooth representation $\sigma$ of $T_s$ can be computed by the homology of the complex $0 \to \sigma_{\mc O_E^\times} \xrightarrow{\pi -1} \sigma_{\mc O_E^\times} \to 0$, which gives the $\pi$-coinvariants in degree $0$ and the $\pi$-invariants in degree $-1$. If $\sigma$ is an unramified character, then the homology vanishes in all degrees except for the trivial character. Then by the geometric lemma and the long exact sequence on homology, the $T_s$-homology groups of $\chi \cdot (\ind_B^{\GL_2(E)}\delta_T^{-1/2})_N$ vanish unless $\chi|_{T_s}=\delta_T^{1/2}|_{T_s}$.\\
	Now suppose that $\chi|_{T_s}=\delta_T^{1/2}|_{T_s}$, and let $\sigma := \chi \cdot (\ind_B^{\GL_2(E)}\delta_T^{-1/2})_N$, which is a non-split extension of the trivial $T$-representation by itself. There is a short exact sequence of complexes of $T_s$-modules
	\begin{center}
	\footnotesize{
	\begin{tikzcd}
		& 0 \arrow[d]& 0\arrow[d] & 0\arrow[d] & \\
		0 \arrow[r]& \Lambda \arrow[d,"0"]\arrow[r]& \sigma \arrow[r]\arrow[d,"\pi-1"]& \Lambda\arrow[r]\arrow[d,"0"] & 0 \\
		0 \arrow[r]& \Lambda\arrow[r]\arrow[d] & \sigma\arrow[r]\arrow[d] & \Lambda\arrow[r]\arrow[d] & 0 \\
		& 0 & 0 & 0 & 
	\end{tikzcd}}	
	\end{center}
	Since $\sigma$ is a non-split extension, the action of $\pi$ must be by an upper triangular matrix, and the connecting homomorphism factors as $b\mapsto \smallvector{a}{b}\mapsto \smallvector{a+ u b}{b}-\smallvector{a}{b}\mapsto \smallvector{u \cdot b}{0}$, which is well-defined for any lift $\smallvector{a}{b}$ of $b$. We claim that $u$ is invertible under our assumptions on $\Lambda$, so that the connecting homomorphism in the long exact sequence is an isomorphism, concluding the proof. For any $f\in \ind_B^{\GL_2(E)}\delta_T^{-1/2}$ such that $f(1)=b\in \Lambda$ and $t\in T$, observe that the function $t\cdot f -\delta_T^{-1/2}(t)f $ vanishes at $1\in \GL_2(E)$ and in particular vanishes on the Bruhat cell $B$. Then after taking $N$-coinvariants, $t\cdot f -\delta_T^{-1/2}f $ is in the kernel of the short exact sequence from Lemma \ref{Jacquet-of-induced-rep}, which is isomorphic to the representation $\delta_T^{-1/2}$ via a Weyl-twisted averaging map $g(x) \mapsto g_N(x):=\int_N g(xwn)$ as a function $x\in T\to \Lambda$, after fixing a Haar measure on $N$ (note that choosing a different Haar measure will rescale $u$). Then $u$ can be computed by explicitly integrating $(t\cdot f -\delta_T^{-1/2}f)(xwn)$ over $N\cong E$ and using smoothness of $f$, reducing to integrating over a certain ball $\pi^i \mc O_E$.
    This is a standard computation, and the result is well-known to be nonzero and invertible for banal torsion coefficients.
\end{proof}
\begin{remark}
	The computation of the scalar $u$ appearing in the proof of Lemma \ref{cohomology-PS(1)} is classically done by analytic continuation of intertwining integrals (for example, in unpublished notes of Casselman \cite[Proposition 9.4.5(a)]{casselman1995introduction}). These are $\GL_2(E)$-equivariant operators $\ind_{B}^{\GL_2(E)}\delta_T^{-1/2+s} \to \ind_{B}^{\GL_2(E)} \delta_T^{-1/2-s}$ which yield isomorphisms for irreducible parabolic inductions when $s\neq \pm 1/2,0$ is a complex parameter (interpolating the root lattice of the $\SL_2$ torus). Then under a certain normalization as $s$ goes to $0$, $u$ is $2(1-q^{-1})$. Classically, this term simultaneously arises as a local factor at an unramified place of a special value of the global $L$-function arising from the constant term of a certain Eisenstein series. In our setting, $u$ can be interpreted as a unipotent monodromy-action of a torus $T$ on the local system corresponding to the normalized Jacquet module of the principal series representation $\mathrm{PS}(1)$ along the Bruhat-stratification of $\underline{\Proj^1(E)}= {\set{\infty}}\bigcup \underline{B \backslash BwB}$. More generally, there should be a strong connection between nontrivial images of gluing functors and degenerations of analytic functions built from intertwining operators in Casselman's basis.
\end{remark}
\begin{proposition}\label{non-LS-type}
    Let $\chi=\chi_1 \otimes \chi_2$ be an unramified character of $T\subset B$. Then $H_c^i(\widetilde{\mc M_{b_2}^{b_1}}, \ind_B^{\GL_2(E)}\chi\cdot \delta_T^{-1/2})$ vanishes for all $i$ unless $\chi_1 \cdot \chi_2^\inv \cong 1$ or $\abs{-}^{\pm 1}$ as characters $E^\times \to \Lambda$.
    
    Furthermore, if $\chi = \abs{-}_1^a \otimes \abs{-}_2^b: (z_1,z_2)\mapsto \abs{z_1}^a \cdot \abs{z_2}^b$ with $a,b\in \frac{1}{2}\Z$, then \newline
    for $a+1=b$,
    \[H_c^i(\widetilde{\mc M_{b_2}^{b_1}},\ind_B^{\GL_2(E)}(\abs{-}^a\otimes \abs{-}^{a+1})\cdot \delta_T^{-1/2}) =\begin{cases} 
    \abs{-}^{a+1/2} & i=1,2\\
    0,\quad & \text{else,}
    \end{cases}
    \]
    for $a=b$,
    \[H_c^i(\widetilde{\mc M_{b_2}^{b_1}},\ind_B^{\GL_2(E)}(\abs{-}^a\otimes \abs{-}^a)\cdot \delta_T^{-1/2}) =\begin{cases} 
    \abs{-}^a\delta_T^{1/2} & i=2,3\\
    0,\quad & \text{else,}
    \end{cases}
    \]
    and for $a-1=b$,
    \[H_c^i(\widetilde{\mc M_{b_2}^{b_1}},\ind_B^{\GL_2(E)}(\abs{-}^{a}\otimes \abs{-}^{a-1})\cdot \delta_T^{-1/2}) =\begin{cases} 
    	\abs{-}^{a-1/2}\delta_T & i=3,4\\
    	0,\quad & \text{else,}
    \end{cases}
    \]
    and the cohomology vanishes for any other pair $(a,b)$.
\end{proposition}
\begin{remark}
    The characters $\chi$ such that $\chi_1\cdot \chi_2^\inv$ is not $1$ or $\abs{-}^{\pm 1}$ are called \emph{generic} (see \cite[Example~3.9]{hamann2023geometric}). Hamann proves in a much more general setting that $Ri_{b_1,*}$ and $Ri_{b_1,!}$ agree on (normalized) parabolically induced representations for generic characters \cite[Proposition~11.13]{hamann2023geometric}, so that the image of the gluing functors vanish for generic $\chi$. 
\end{remark}
\begin{proof}
	We outline the proof in the following steps: the $\GL_2(E)$-homology of $(\ind_B^{\GL_2(E)}\chi\cdot \delta_T^{-1/2} )\otimes \sigma $ is computed by the $B$-homology of $\sigma$ twisted by the modulus character of $B$ and $\chi$; we apply this to $\sigma =\rho\otimes_\Lambda \cind_{\SL_2(E)}^{\GL_2(E)}\Lambda$. We then reduce to the homology of the standard Borel $B_s\subset \SL_2(E)$ since every module is tensored with $\cind_{\SL_2(E)}^{\GL_2(E)} \Lambda$.
	Then we filter $B_s$ by its unipotent subgroup to reduce to the torus homology of the (non-normalized) Jacquet module, which vanishes except for non-generic characters $\chi$ by the previous lemma.
	
    Lemma \ref{homology-of-tensor-with-parabolic-induction} reduces the computation to the smooth $B$-homology of the modules from Theorem \ref{cs-coh-double-tilde-M-trivial-module} restricted to the Borel subgroup and twisted by $\chi$ and the modulus character $\delta_T$. Let $B_s$ denote the standard Borel subgroup of $\SL_2(E)$, so that there is a short exact sequence $1\to B_s \to B\xrightarrow{\det}E^\times \to 1$. We write $\chi_s:=\chi|_{B_s}=\chi_1\cdot \chi_2^\inv $ for the character induced by restriction to $B_s$ and identifying the split 1-torus in $B_s$ via the first coordinate. Then we want to compute the smooth group homology 
    \[H_{-i}(E^\times, H_{-j}(B_s, \delta_T^{1/2}\otimes \chi_s \otimes \res_{B_s}^{\GL_2(E)}\rho)\otimes \cind_1^{E^\times}\Lambda)\]
    for $\rho = \st,\ \ind_B^{\GL_2(E)}(\abs{-}\otimes 1),\ \abs{\det}$ (corresponding to the compactly supported cohomology of $\widetilde{\widetilde{\mc M_{b_2}^{b_1}}}$ in degrees 2,3,4 respectively), and we perform these computations in parallel. The underlying $\Lambda$-modules are formally isomorphic to $H_{-j}(B_s,\delta_T^{1/2}\otimes \chi_s\otimes \res_{B_s}^{\GL_2(E)}\rho)$ when $i=0$ and vanish whenever $i\neq 0$, but we will still need to determine the induced determinant characters on $B_s$-homology groups in order to understand the correct $G_{b_2}(E)$-actions. We will see shortly that this corresponds to the action of the first toral component on the unnormalized Jacquet module.
    
    There is no contribution to $B_s$-homology from the unipotent subgroup outside degree 0, so there are further isomorphisms \[H_{-j}(B_s,\delta_T^{1/2}\otimes \chi_s \otimes \res_{B_s}^{\GL_2(E)}\rho)\cong H_{-j}(T_s, \left (\delta_T^{1/2}\otimes \chi_s\otimes \res_{B_s}^{\GL_2(E)}\rho \right )_N)\] as $E^\times$-modules. Note that for any character $\chi: B \to \Aut(\Lambda)$ that is trivial on $N$, there is an equivariant isomorphism $(\chi \otimes \rho)_N \cong \chi|_T \otimes (\rho)_N$ as $T$-representations.  Applying Lemma \ref{Jacquet-of-induced-rep}, we see that $(\delta_T^{ 1/2}\cdot \chi_s\otimes \res_{B_s}^{\GL_2(E)}\rho)_N= \chi_s \cdot \delta_T^{ 1/2 -1}, \chi_s\cdot \delta_T^{ 1/2}\cdot (\ind_{B}^{\GL_2(E)}\delta_T^{-1/2})_N, \chi_s\cdot \delta_T^{ 1/2}$, respectively (note the twist by $\delta_T^{-1}$ in the first module arising from the Jacquet module of the Steinberg representation). Now fix an isomorphism $T_s\cong E^\times$ via the first coordinate, so that the restriction to $T_s$ can be identified with the representations $\chi_s\cdot |-|, \chi_s \cdot |-|^{-1} \cdot (\ind_B^{\GL_2(E)} \delta_T^{-1/2})_N, \chi_s\cdot |-|^{-1}$, respectively. Recall that $(\ind_B^{\GL_2(E)} \delta_T^{-1/2})_N$ is a nonsplit extension of $\delta_T^{-1/2}$ by itself. 
    
    Applying Proposition \ref{homology-E-times} computes the smooth $T_s$-homology in terms of $\mc O_{E}^\times,\pi^{\Z}\subset T_s$-(co)invariants, giving the second page of the compactly-supported Hochschild-Serre spectral sequence $E_2^{p,q}= H_{-p}(B, \chi \otimes H_c^q(\widetilde{\widetilde{\mc M_{b_2}^{b_1}}},\Lambda)) \implies H_c^{p+q}(\widetilde{\mc M_{b_2}^{b_1}}, \ind_B^{\GL_2(E)} \chi)$:

\begin{center}
\begin{tikzpicture}
  \matrix (m) [matrix of math nodes,
    nodes in empty cells,nodes={minimum width=8ex,
    minimum height=5ex,outer sep=-5pt},
    column sep=0ex,row sep=1ex]{
               & 0 & ((\chi_s\cdot |-|^{-1})_{\mc O_E^\times})^{\pi^{\Z}} & (\chi_s\cdot |-|^{-1})_{E^\times} & 4\\  
               & 0 & ((\chi_s\cdot \abs{-}^{-1}\cdot (\ind_{B}^{\GL_2(E)}\delta_T^{-1/2})_N)_{\mc O_E^\times})^{\pi^{\Z}} & ((\chi_s\cdot \abs{-}^{-1}\cdot (\ind_{B}^{\GL_2(E)}\delta_T^{-1/2})_N)_{E^\times} & 3\\
                 &   0  & ((\chi_s\cdot |-|)_{\mc O_E^\times})^{\pi^{\Z}} & (\chi_s\cdot |-|)_{E^\times} & 2 \\
                &0 & 0 & 0 & 1\\
               & 0& 0  & 0 & \quad q=0\\
    \strut  &  -2  &  - 1  & p=0 & \strut \\};
\draw[thick] (m-1-5.west) -- (m-6-5.west) ;
\draw[thick] (m-6-1.north) -- (m-6-5.north) ;
\end{tikzpicture}
\end{center}
By Lemma \ref{cohomology-PS(1)}, we see that the $q=3$ row vanishes unless $\chi_s$ is trivial, in which case each of the terms are free of rank $1$. Likewise, the $q=2$ resp. $q=4$ row vanishes unless $\chi_s = \abs{-}^{-1}$ resp. $\chi_s =\abs{-}$, and in these cases the nonvanishing terms are free of rank $1$.

Furthermore, the spectral sequence exists in the category of smooth $G_{b_2}(E)$-modules, whose action is identified by the induced determinant actions as in the proof of Corollary \ref{determinant-character-twists}. The induced determinant characters $\chi_{\det}$ are well-defined on lifts of the determinant, so we only need to analyze the  action of $\begin{bmatrix}
    z & 0 \\ 0 & 1
\end{bmatrix}$
on $H_j(B_s, (\delta_T^{1/2}\otimes \chi_s \otimes \res_{B_s}^{\GL_2(E)}\rho))$.
Then in each degree, the $E^\times$-coinvariants are identified with those functions supported on $1\in E^\times$, such that $z\cdot f_1=\chi_{\det}(z)f_{z_1^\inv}$. For $q=2$, we have $(\delta_T^{1/2}\otimes \chi \otimes \st)_N \cong \delta_T^{1/2}\otimes \chi \otimes \delta_T^{-1}$, which induces $\chi_{\det}= \abs{-}^{a+1/2}$. For $q=3$, we have $(\delta_T^{1/2}\otimes \chi \otimes \abs{-}^{1/2}\ind_B^{\GL_2(E)}(\delta_T^{-1/2}))_N$ 
inducing $\chi_{\det}=\abs{-}^{a+1/2}$. 
For $q=4$, we have $(\delta_T^{1/2}\otimes \chi \otimes \abs{\det})_N$ inducing $\chi_{\det}=\abs{-}^{a+1/2}$.

Recall that the representations $\rho = \st, \ind_B^{G}(\abs{-}\otimes 1),\abs{\det}$ arise from the Tate twists $\Lambda, \Lambda(-1), \Lambda(-2)$, respectively. 
Then as in the proof of Corollary \ref{determinant-character-twists},
we compute for $z_1\in E_1^\times$ and $z_2\in E_2^\times$ in each nonvanishing degree:
\begin{align*}
    q=2 :\quad & z_1\cdot f_1= f_{z_1} &=\abs{z_1}^{a+1/2} f_1,\\
    & z_2\cdot f_1= f_{z_2^{}} &= \abs{z_2}^{a+1/2}f_1,\\
    q=3 :\quad & z_1\cdot f_1= \abs{z_1}^{-1}f_{z_1} =\abs{z_1}^{a+1/2-1} f_1 &= \abs{z_1}^{a-1/2} f_1,\\
    & z_2\cdot f_1= f_{z_2^{}} = \abs{z_2}^{a+1/2}f_1 &=\abs{z_2}^{a+1/2}f_1,\\
    q=4 :\quad & z_1\cdot f_1= \abs{z_1}^{-2}f_{z_1}=\abs{z_1}^{a+1/2-2} f_1 &=\abs{z_1}^{a-1/2-1} f_1,\\
    & z_2\cdot f_1= f_{z_2^{}}  = \abs{z_2}^{a+1/2}f_1 &=\abs{z_2}^{a-1/2+1}f_1,
\end{align*}
inducing the characters $\abs{-}^{a+1/2}, \abs{-}^a \delta_T^{1/2},$ and $\abs{-}^{a-1/2}\delta_T$, respectively. 
\end{proof}
\begin{corollary}\label{gluing-functors-for-unram-principal-series}
    Let $m,n\in \frac{1}{2}\Z$. Then 
    $i_{b_2}^* Ri_{b_1,*} \ind_B^{\GL_2(E)}\left(\abs{\det}^m\otimes \delta_T^{n}\otimes \delta_T^{-1/2}\right)$ vanishes unless $n\in \set{-1/2,0,1/2},$ where we have
    \begin{align*}
    H^k(i_{b_2}^* Ri_{b_1,*} \ind_B^{\GL_2(E)}\left(\abs{\det}^m\otimes \delta_T^{-1/2}\otimes \delta_T^{-1/2}\right)) &=\begin{cases} 
    \abs{-}^m\delta_T & k=2,3\\
    0,\quad & \text{else,}
    \end{cases}\\
    H^k(i_{b_2}^* Ri_{b_1,*} \ind_B^{\GL_2(E)}\left(\abs{\det}^m \otimes \delta_T^{-1/2}\right)) &=\begin{cases} 
    \abs{-}^m\delta_T^{1/2} & k=1,2\\
    0,\quad & \text{else,}
    \end{cases}\\
    H^k(i_{b_2}^* Ri_{b_1,*} \ind_B^{\GL_2(E)}\left(\abs{\det}^m\otimes \delta_T^{1/2}\otimes \delta_T^{-1/2}\right)) &=\begin{cases} 
    \abs{-}^m & k=0,1\\
    0,\quad & \text{else.}
    \end{cases}
    \end{align*}
\end{corollary}
\begin{proof}
    This follows from Verdier duality, as in the proof of Theorem \ref{Purity-for-trivial-rep}. By the same trick, we can commute characters $\abs{\det}^m$, so that the computation reduces to $m=0$. Similarly, the dualizing sheaf of $\widetilde{\mc M_{b_2}^{b_1}}$ is $\Lambda(2)[4]$, so that we are reduced to computing \[H_c^{4-k}(\widetilde{\mc M_{b_2}^{b_1}},\Lambda(2)\otimes [\ind_B^{\GL_2(E)}\delta_T^{n-1/2}])^\vee.\]
    Matching $n={-1/2,0,1/2}$ to the three cases in the previous proposition, we see that $m=0$ corresponds to $a=-1/2,0,1/2$, respectively. Note that twisting by $\Lambda(2)$ induces a determinant twist $\abs{\det}^\inv$ on $\GL_2(E)$-reps, so that the $B_s$-homology groups are unchanged. Then we only have to redo the identification of the induced $E_1^\times \times E_2^\times$ characters. Now, the twisted determinant reps are all $\chi_{\det} = |-|^{a+1/2-1}=|-|^{a-1/2}$, and the total Tate twists in each case correspond to $\Lambda(2),\Lambda(1),\Lambda$, respectively. Thus, we have in each non-vanishing degree
    \begin{align*}
    q=2 :\quad & z_1\cdot f_1= \abs{z_1}^{2}f_{z_1}=\abs{z_1}^{a-1/2+2} f_1 &=\abs{z_1}^{a+1/2+1} f_1,\\
    & z_2\cdot f_1= f_{z_2^{}}  = \abs{z_2}^{a-1/2}f_1 &=\abs{z_2}^{a+1/2-1}f_1,\\
    q=3 :\quad & z_1\cdot f_1= \abs{z_1}^{1}f_{z_1} =\abs{z_1}^{a-1/2+1} f_1 &= \abs{z_1}^{a+1/2} f_1,\\
    & z_2\cdot f_1= f_{z_2^{}} = \abs{z_2}^{a-1/2}f_1 &=\abs{z_2}^{a-1/2}f_1,\\
    q=4 :\quad & z_1\cdot f_1= f_{z_1} &=\abs{z_1}^{a-1/2} f_1,\\
    & z_2\cdot f_1= f_{z_2^{}} &= \abs{z_2}^{a-1/2}f_1,
\end{align*}
so that setting $a=-1/2,0,1/2$ induces the characters $\delta_T^\inv, \delta_T^{-1/2},\Lambda$, respectively, dualizing to $\delta_T,\delta_T^{1/2}$, and $\Lambda$. 
\end{proof}

\subsection{Vanishing of cuspidal representations, Kirillov model, and Iwahori-fixed vectors}
We compute directly that the $*$-inclusion of cuspidal representations vanish when restricted to a higher stratum. There is a formal argument in much greater generality following \cite{fargues2021geometrization} using properties of excursion operators and compatibility with parabolic inductions. This vanishing property is sometimes referred to as the \emph{cleanliness} of cuspidal parameters.

Let $M\subset B \subset \GL_2(E)$ be the subgroup of matrices $\begin{bmatrix}t&n\\ &1 \end{bmatrix}$, called the \textit{mirabolic subgroup}.
\begin{lemma}[{\cite[Section~36-37]{Bushnell2006}}]
Let $\vartheta$ be a fixed non-trivial character of the group $N$ of upper triangular unipotent matrices in $\GL_2(E)$. Let $\sigma$ be a smooth irreducible cuspidal representation of $\GL_2(E)$.
\begin{enumerate}
    \item There is an $M$-isomorphism $\res^{\GL_2(E)}_M \sigma \xrightarrow{\sim} \cind_N^M \vartheta$.
    \item There is a unique homomorphism $\sigma_{\mc K}: \GL_2(E) \to \Aut_\Lambda (\cind_N^M \vartheta)$ such that
    \begin{enumerate}
        \item $\sigma_{\mc K}|_M$ restricts to the natural action of $M$ on $\cind_N^M \vartheta$, and
        \item $\sigma_{\mc K}\cong \sigma$.
    \end{enumerate}
\end{enumerate}
The representation $\sigma_{\mc K}$ is called the \emph{Kirillov model of $\sigma$}.
\end{lemma}
For any character $\chi: E^\times \to \Lambda$, there is a twisted indicator function 
    \begin{align*}
    f_{\chi,k}(x):=
    \begin{cases}
        \chi(x),\quad x\in \pi^k \mc O_E^{\times},\\ 0,\quad \text{else}.
    \end{cases}
    \end{align*}
Note that such functions generate $\cind_{1}^{E^\times}\Lambda$ as a $\Lambda$-module when $\chi$ and $k$ are allowed to vary.
\begin{lemma}[{\cite[\hphantom~37.1.1,~Theorem~37.3]{Bushnell2006}}]
    Let $\sigma_{\mc K}$ be the Kirollov model of a smooth irreducible cuspidal representation $\sigma$. Then the action of $\GL_2(E)$ on $\sigma_{\mc K}\cong_{\Lambda} \cind_{1}^{E^\times}\Lambda$ is partially described by
    \begin{align*}
        \begin{bmatrix}
            t & 0 \\ 0 & 1
        \end{bmatrix} \cdot f: x \mapsto f(tx), \quad
        \begin{bmatrix}
            1 & n \\ 0 & 1
        \end{bmatrix} \cdot f : x \mapsto \vartheta(xn)f(x), \quad
        \begin{bmatrix}
            z & 0 \\ 0 & z
        \end{bmatrix} \cdot f = \omega_{\sigma}(z)f,
    \end{align*}
    where $\omega_{\sigma}$ is the central character of $\sigma$. 
\end{lemma}
\begin{lemma}
    Let $\sigma$ be a smooth irreducible cuspidal representation of $\GL_2(E)$. Then $H_i(\GL_2(E), \sigma)=0$ for every $i$. Additionally, $H_i(\SL_2(E),\sigma)=0$ for every $i$.
\end{lemma}
\begin{proof}
    It suffices to compute the smooth group homology for the restriction to $\SL_2(E)$ via the Hochschild-Serre spectral sequence. Under our assumptions on $\Lambda$, the contragredient functor $\sigma \mapsto \sigma^*$ is exact and yields an isomorphism $\sigma \hookrightarrow (\sigma^*)^*$ when restricted to irreducible representations; in particular, taking contragredients preserves irreducible supercuspidal representations. Then by Proposition \ref{dual-of-homology}, it suffices to show vanishing of smooth $\SL_2(E)$-cohomology of $\sigma$. This is computed to be $\sigma^{\SL_2(E)}$ in degree $0$ and $\sigma^{\Gamma_0}/(\sigma^U+\sigma^{U'})$ in degree $1$ in Theorem \ref{sl2-cohomology}. By assumption, there is a nonzero element $x_0\in E$ such that $\vartheta(x_0)\neq 1$, and we write $x_0= u_0 \pi^{k_0}$ for some unit $u_0 \in \mc O_E^\times$. Then one computes
    \begin{align*}
        \begin{bmatrix}
        1 & \pi^{2k_0-k} \\ 0 & 1    
        \end{bmatrix} \cdot f_{\chi,k}(\pi^{k-k_0}x_0)= \vartheta(x_0) f_{\chi,k}(\pi^{k-k_0}x_0)\neq f_{\chi,k}(\pi^{k-k_0}x_0),
    \end{align*}
    so that $\sigma^{\SL_2(E)}=0$.

    The vanishing of the first cohomology follows from \cite[Proposition~2.4]{Casselman1980TheUP}; more precisely, if a smooth admissible module has a nonzero fixed vector by the Iwahori subgroup $\Gamma_0$, then it is not supercuspidal. 
\end{proof}
\begin{proposition}\label{vanishing-of-cuspidals}
    Let $\sigma$ be a smooth irreducible cuspidal representation of $\GL_2(E)$. Then the cohomology $H^i_c(\widetilde{\mc M_{b_2}^{b_1}},\sigma)$ vanishes in every degree. In particular, $i_{b_2}^* R i_{b_1,*}\sigma=0$.
\end{proposition}
\begin{proof}
    Remark \ref{pulling-out-reps} allows us to restrict attention to the smooth group homology of $\sigma \otimes H^i_c(\widetilde{\widetilde{\mc M_{b_2}^{b_1}}},\Lambda)$. In each degree, there is a tensor factor of $\cind_{\SL_2(E)}^{\GL_2(E)}\Lambda$ appearing, so the Hochschild-Serre spectral sequence for $E^\times \cong \GL_2(E)/\SL_2(E)$ reduces to showing the vanishing of the $\SL_2(E)$-homology of $\sigma \otimes \st$, $\sigma \otimes \ind_B^{\SL_2(E)}\delta_T^{-1/2}$, and $\sigma$ as in the computations for Theorem \ref{cs-coh-tilde-M-trivial-module}.

    The vanishing of $\SL_2(E)$-homology of $\sigma$ is the previous lemma. Furthermore, $H_i(\SL_2(E),\sigma \otimes \ind_{B}^{\SL_2(E)}\delta_T^{-1/2})\cong H_i(B, \delta_T^{-1/2}\cdot \sigma)$, and the spectral sequence for $1\to N\to B\to T\to 1$ reduces to showing that $H_i(N, \sigma)=0$. But the Kirillov model $\sigma_{\mc K}= \cind_N^M \vartheta$ implies that these homology groups are contained in a direct sum of $H_i(N,\vartheta)=0$, since $\vartheta$ was chosen to be a nontrivial character of the unipotent group $N$.

    For the vanishing of cohomology of $\sigma \otimes \st$, note that the underlying $\Lambda$-module structure of $\sigma_{\mc K}\cong_{\Lambda} \cind_1^{E^\times} \Lambda$ is flat, so that there is a short exact sequence
    \[0\to \sigma \to \sigma \otimes \ind_{B}^{\SL_2(E)}\Lambda \to \sigma \otimes \st \to 0.\]
    Then the long exact sequence on smooth homology concludes the computation.
\end{proof}

\subsection{Specialization from the basic point in a half-integral slope component}
Let $b_1\leftrightarrow \mc O(1/2)$ be the basic point of the slope $1/2$-component, with specialization $b_2\leftrightarrow \mc O \oplus \mc O(1)$. 
\begin{proposition}\label{double-uniformized-slope-1/2-cover}
After basechange to $\spa C$, there is an equivariant isomorphism
\[\widetilde{\widetilde{\mc M_{b_2}^{b_1}}}\cong (\BC(\mc O(1/2))\setminus{0})\times \underline{E^\times},\]
where $\BC(\mc O(1/2))$ inherits the natural action of $\underline{G_{b_1}(E)}\cong \underline{D^\times}$ on $\mc O(1/2)$ via postcomposition, and $\underline{D^\times}$ acts on the factor of $\underline{E^\times}$ via the reduced norm. Furthermore, the first factor of $\underline{G_{b_2}(E)}\cong \underline{E_1^\times}\times \underline{E_2^\times}$ acts on $\BC(\mc O(1/2))$ via precomposition, and each $\underline{E_i^\times}$ acts on $\underline{E^\times}$ by the inverse of multiplication. 
\end{proposition}
\begin{proof}
The identification is exactly the same as in Proposition \ref{Equivariant-iso-cover}, except one has to consider diagrams of the form 
\begin{center}
\begin{tikzcd}
0 \arrow[r] & \mc K |_{X_T} \arrow[r]\arrow[d,"\phi_K"] & \mc E|_{X_T} \arrow[r]\arrow[d,"\phi_E"]& \mc C |_{X_T} \arrow[d,"\phi_C"]\arrow[r]& 0 \\
&\mc O \arrow[r,dashed,"f"]\arrow[d, swap,"\underline{E_1^\times}"]&\mc O(1/2) \arrow[d, swap,"\underline{D^\times}"]& \mc O(1)\arrow[d, swap,"\underline{E_2^\times}"] & \\
&\mc O &\mc O(1/2) & \mc O(1). & 
\end{tikzcd}
\end{center}
Additionally, Lemma \ref{nonsaturated} implies that the nonsaturated injections $\mc O\hookrightarrow \mc O(1/2)$ extend to an injection $\mc O(1)\hookrightarrow \mc O(1/2)$, but $\Hom(\mc O(1),\mc O(1/2))=0$. Thus, the nonzero sections $f\in \BC(\mc O(1/2))$ parameterize nonsplit extensions $0\to \mc O\to \mc O(1/2)\to \mc C\cong \mc O(1)\to 0$, and fixing an isomorphism of the determinant $\Lambda^2 \mc O(1/2)\cong \mc O(1)$ furthermore identifies $\mc C$ with $\mc O(1)$.
\end{proof}
This characterization already yields an interesting description of the gluing functor in terms of the compactly supported cohomology of a quotient of the Lubin-Tate tower at infinite level (although we will not need this).
\begin{example}[{\cite[Example~II.3.12]{fargues2021geometrization}}]\label{Lubin-Tate-tower}
The infinite level Lubin-Tate tower $\mathrm{LT}_\infty$, basechanged to some $\spa C^\sharp$, is the moduli of modifications $\mc O_{X_S}^2\hookrightarrow \mc O_{X_S}(1/2)$ at the point of the curve determined by the untilt $C^\sharp$. In particular, one deduces a $\underline{D^\times}$-equivariant isomorphism
\[\mathrm{LT}_\infty^\flat/\begin{pmatrix}1 & \underline{E}\\ 0 & \underline{E^\times}\end{pmatrix}\cong (\BC(\mc O(1/2)\setminus \set{0})\times_k \spa C.\]
\end{example}
This quotient of the Lubin-Tate tower only keeps track of the first coordinate of the modification, $\mc O \hookrightarrow \mc O(1/2)$. Furthermore, \cite[Example~II.3.12]{fargues2021geometrization} identifies $\widetilde{\mc M_{b_2}^{b_1}}\cong \BC(\mc O(-1)[1])\setminus{0}$ (after basechange to $\spa C$) as a quotient $\widetilde{\mc M_{b_2}^{b_1}}\cong (\BC(\mc O(1/2))\setminus{0})/\underline{\SL_1(D)}$, where $\SL_1(D):=\ker(\Nrd: D^\times \to E^\times)$. Then after fixing an untilt $C^\sharp$ and global section of $\mc O(1)$, the long exact sequence of the associated divisor short exact sequence $0\to \mc O(-1)\to \mc O\to i_*C^\sharp\to 0$ yields an isomorphism $\BC(\mc O(-1)[1])\setminus \set{0}\times_k\spa C\cong (\Omega_{C^\sharp})^{\diamond}/\underline{E}$, where $\Omega = \mathbb{A}^1_E\setminus \underline{E}$ is the Drinfeld upper half plane. 

We remark that it is possible to compute $H^i_c(\widetilde{\mc M_{b_2}^{b_1}},\Lambda)$ from this description, but we provide a different computation, similar to the integral slope components, that we feel better demonstrates the $\underline{G_{b_2}(E)}$-module structure.
\begin{proposition}\label{double-uniformized-half-slope-cs-coh}
The compactly supported cohomology of $\widetilde{\widetilde{\mc M_{b_2}^{b_1}}}$ as $D^\times$-modules is 
\[H_c^i(\widetilde{\widetilde{\mc M_{b_2}^{b_1}}},\Lambda) =\begin{cases} 
\Lambda\otimes \cind_{\SL_1(D)}^{D^\times}\Lambda,\quad & i=1\\
\abs{\Nrd}^{}\otimes \cind_{\SL_1(D)}^{D^\times}\Lambda,\quad & i=2\\
0,\quad & \text{else.}
\end{cases}
\]
\end{proposition}
\begin{proof}
Abstractly, the compactly supported cohomology of $\BC(\mc O(1/2))\setminus{0}$ agrees with that of $\D^*$, so that the $\Lambda$-module structure is $H^1_c(\BC(\mc O(1/2))\setminus 0,\Lambda)= \Lambda$, $H^2_c(\BC(\mc O(1/2))\setminus 0,\Lambda)= \Lambda(-1)$, and 0 in every other degree. Then they are characters of $D^\times$, which necessarily factor through $\Nrd:D^\times \to E^\times$ \cite[53.5~Proposition]{Bushnell2006}.
Note that the action of $\pi\in E^\times \subset D^\times$ carries over to the square of inverse geometric Frobenius on $\D^*$.
Then exactly as in the proof of Proposition \ref{actions-on-Z}, we see that $H^1_c(\BC(\mc O(1/2))\setminus 0,\Lambda)= \Lambda$ is the trivial module, and $H^2_c(\BC(\mc O(1/2))\setminus 0,\Lambda)$ is $\pi \mapsto q^{-2} = \abs{\Nrd \pi}^{}\in \Aut(\Lambda)$. Also as in the slope $0$ case, the compactly supported cohomology of $\underline{E^\times}$ is the compactly supported functions $E^\times \to \Lambda$ in degree $0$, and the $D^\times$-module is $\cind_{\SL_1(D)}^{D^\times}\Lambda$ since $\Nrd:D^\times \to E^\times$ extends to the geometric action $\underline{D^\times}\to \Aut(\underline{E^\times})$. Then the proof is finished by applying the Künneth formula.
\end{proof}
\begin{theorem}
The compactly supported cohomology of $\widetilde{\mc M_{b_2}^{b_1}}$ as $E_1^\times \times E_2^\times$-modules is 
\[H_c^i(\widetilde{\mc M_{b_2}^{b_1}},\Lambda) =\begin{cases} 
\Lambda,\quad & i=1\\
\delta_T,\quad & i=2\\
0,\quad & \text{else.}
\end{cases}
\]
\end{theorem}
\begin{proof}
We'll apply the spectral sequence
\begin{align}
E_2^{i,j}=H_{-i}(D^\times, H_c^j(\widetilde{\widetilde{\mc M_{b_2}^{b_1}}},\Lambda))\implies H_c^{i+j}(\widetilde{\mc M_{b_2}^{b_1}},\Lambda)
\end{align}
from Proposition \ref{homological-Hochschild-Serre}, as well as the homological Hochschild-Serre spectral sequence 
\[E_2^{p,q}= H_{-p}(E^\times, H_{-q}(\SL_1(D),\sigma))\implies H_{-p-q}(D^\times,\sigma).\]
Both of the modules appearing in the cohomology of $\widetilde{\widetilde{\mc M_{b_2}^{b_1}}}$ restrict to $\bigoplus_{E^\times}\Lambda$ as $\SL_1(D)$-modules, so that the $\SL_1(D)$-homology reduces to that of the trivial module. Now since the only split torus inside $D^\times$ is the inclusion of the center $E^\times \subset D^\times$, the reduced Bruhat-Tits building for $\SL_1(D)$ is $0$-dimensional, hence $\SL_1(D)$-group homology is the module of coinvariants supported in degree $0$.

Now the computation is reduced to the $E^\times$-homology of $\cind_1^{E^\times} \Lambda$ and $\abs{\Nrd}^{}\cind_{1}^{E^\times}\Lambda$; more generally we compute the homology of $\abs{\Nrd}^k \cind_1^{E^\times}\Lambda$. By proposition \ref{homology-E-times}, this is the module of $E^\times$-coinvariants in degree $0$ and $(\abs{-}^k \otimes \cind_1^{E^\times}\Lambda_{\mc O_{E}^{\times}})^{\pi^{\Z}}$ in degree $-1$. 
As in the slope $0$ computations, the first homology is $(\bigoplus_{\pi^{\Z}}\abs{-}^k)^{\pi^{\Z}}=0$, since any fixed vector would have to be supported on all of $\pi^{\Z}$, which is not compact. The $0$th homology groups are the module of functions supported on $\set{1}\subset E^\times$ with the relation and $f_{z}= \abs{\Nrd}^{k}f_1$ (where $f_z$ is the indicator function for $z\in E^\times$). 
The geometric action of $\underline{E_1^\times}\times \underline{E_2^\times}$ takes $\pi_1\in E_1^\times$ to the inverse action of $\pi \in E^\times \subset D^\times$ on $\BC(\mc O(1/2))$, inducing the character $\pi_1\mapsto \abs{\Nrd \pi^\inv}^k=\abs{\pi_1}^{-2k}$, and $z_i\in E_i^\times$ rotates the domain of the functions $E^\times \to \Lambda$ in the opposite way as $\Nrd:D^\times \to E^\times \subset \Aut(\underline{E^\times})$. 
Then the $E^\times$-homology of $\abs{\Nrd}^k\cind_1^{E^\times}\Lambda$ is $\delta_T^k [0]$, which follows from the computations
\[z_1\cdot f_1= \abs{z_1}^{-2k}f_{z_1}=\abs{z_1}^{-k}f_1,\]
and
\[z_2\cdot f_1= f_{z_2^{}} = \abs{z_2}^k f_1.\]
\end{proof}
\begin{corollary}
Let $\chi$ be an unramified character. Then
\[H_c^i(\widetilde{\mc M_{b_2}^{b_1}},\chi\circ \Nrd) =\begin{cases} 
\chi \otimes \chi,\quad & i=1\\
(\chi \otimes \chi) \cdot \delta_T,\quad & i=2\\
0,\quad & \text{else.}
\end{cases}
\]
In particular, $H^1_c(\widetilde{\mc M_{b_2}^{b_1}},\abs{\Nrd}^k)=\abs{-}^k$ and $H^2_c(\widetilde{\mc M_{b_2}^{b_1}},\abs{\Nrd}^k)=\abs{-}^k\delta_T$.
\end{corollary}
\begin{theorem}
The cohomology of $\abs{\Nrd}^k$ as a sheaf on $\widetilde{\mc M_{b_2}^{b_1}}$ as $E_1^\times \times E_2^\times$-modules is 
\[H^i(\widetilde{\mc M_{b_2}^{b_1}},\abs{\Nrd}^k) =\begin{cases} 
\abs{-}^k,\quad & i=0\\
\abs{-}^k\delta_T,\quad & i=1\\
0,\quad & \text{else.}
\end{cases}
\]
\end{theorem}

The remaining irreducible smooth $D^\times$-modules are finite dimensional and compactly induced from a compact-mod-center open subgroup $K\subset D^\times$, and the images of the gluing functors reduce to $K\cap \SL_1(D)$-orbits of smooth characters of $K$.

%% file: Smooth_Reps.tex
All results appearing in this appendix are known, although they perhaps do not appear in the literature in this precise context. In particular, the theory of smooth representations is usually developed with discrete complex coefficients. Since we work with discrete $\Lambda$-modules, the invertibility of $p$ and self-injectivity of $\Lambda$ ensures that all of the theorems that we need carry over.

\subsection{(Co)homology of smooth \texorpdfstring{$\SL_2(E)$}{SL2(E)}-representations via the Bruhat-Tits building}
We would like to understand the smooth cohomology of an irreducible smooth representation of $\SL_2(E)$. In \cite{schneider1997representation}, Schneider and Stuhler defined (co)sheaves on the (reduced) Bruhat-Tits building in order to construct projective homological resolutions of certain complex smooth representations of the $E$-points of a reductive group $G$. Namely, these representations are required to be generated by their fixed vectors under some compact open level subgroup of $G$. Furthermore these resolutions exist in the category of smooth representations with fixed central character $\chi$. We restrict our attention to such resolutions for the trivial module $\Lambda$ over $\SL_2(E)$, so that both of the previous conditions can be safely ignored; in particular, this resolution is precisely the cellular chain complex of the Bruhat-Tits building with coefficients in $\Lambda$. Note that a large part of what follows in this paragraph can be found in \cite{StuhlerSchneider+1993+19+32} and an $\ell = p$ analogue in  \cite{fust2021continuous}, but we include our own exposition for completeness. 

Recall that the Bruhat-Tits building for $\SL_2(E)$, denoted $\BT$, is a $1$-dimensional simplicial complex (equipped with its locally Euclidean metric) with $0$-cells corresponding to similarity classes of rank $2$ $\mc O_E$-lattices (i.e. $L\sim L'$ iff $\exists \lambda \in E^\times$ such that $L=\lambda L'$). There are $1$-cells between similarity classes if and only if there are representatives $L_0$ and $L_1$ such that there are proper inclusions $L_1 \subset L_0 \subset \pi^\inv L_1$, and $\SL_2(E)$ acts on the building by $L\mapsto gL\subset E^2$, which is well-defined up to similarity and preserves edges. For example, there is a distinguished point $[\mc O_E \oplus \mc O_E]$ with an edge connecting to $[\mc O_E \oplus \pi \mc O_E]$. A $1$-simplex is called a \emph{chamber} of an apartment (i.e. infinite geodesic) containing it, and the example from the previous sentence will be our distinguished chamber $C=C_0 \cup C_1$, where $C_0$ is the set of $0$-facets in $C$ and $C_1$ is the $1$-facet. 

While the $\SL_2(E)$-action is not transitive on $0$-cells (since for example $[\mc O_E\oplus \pi \mc O_E]= {\begin{bmatrix}{1}& \\ & {\pi}\end{bmatrix}}[\mc O_E\oplus \mc O_E]$), it is in fact transitive on chambers. In this way, the Bruhat-Tits building is roughly a nonarchimedean analogue of the real symmetric space of $\SL_2$. Then the action splits the $0$-cells into two disjoint classes corresponding to the translates of $[\mc O_E\oplus \mc O_E]$ and $[\mc O_E \oplus \pi \mc O_E]$, which have stabilizer subgroups \[U:=\SL_2(\mc O_E) \text{ and } U':=\begin{bmatrix}1 & \\ & \pi \end{bmatrix}\SL_2(\mc O_E)\begin{bmatrix}1 & \\ & \pi^\inv \end{bmatrix}= \begin{bmatrix}\mc O_E & \pi^\inv \mc O_E\\ \pi \mc O_E & \mc O_E \end{bmatrix}\cap \SL_2(E),\] respectively. The stabilizer of $C$ is \[\Gamma_0 := \set{\begin{bmatrix}* & *\\ & *\end{bmatrix}\mod \pi }\cap \SL_2(\mc O_E) = U\cap U'.\]
It will also be useful to denote these stabilizer subgroups by $P_F$ (for parahoric subgroup), where $F$ is the fixed facet. Since all of these groups are compact open subgroups of $\SL_2(E)$, the quotients are all discrete, so that cellular chains with $\Lambda$ coefficients agree with elements of the compact-induction from the stabilizer groups.
\begin{proposition}
There is an $\SL_2(E)$-equivariant isomorphism from the cellular chain complex of $\BT$ to 
\[\cdots \to 0 \to \cind_{\Gamma_0}^{\SL_2(E)}\Lambda \xrightarrow{\delta} \cind_{U}^{\SL_2(E)}\Lambda\oplus \cind_{U'}^{\SL_2(E)}\Lambda\xrightarrow{\deg} \Lambda \to 0\to \cdots,\]
with $\Lambda$ sitting in degree $0$.
\end{proposition}
\begin{proof}
Let $([L_0],[L_1])$ be a simple $1$-chain. By transitivity of the $\SL_2(E)$-action on chambers, we can write $([L_0],[L_1]) = g^\inv([\mc O_E\oplus \mc O_E],[\mc O_E,\pi\mc O_E])$ for $g\in \SL_2(E)$, which is unique on the level of $\Gamma_0$-cosets. Then the $\Lambda$-linear map sending $([L_0],[L_1])$ to the indicator function $f_{\Gamma_0\cdot g}\in \cind_{\Gamma_0}^{\SL_2(E)}\Lambda$ is a linear isomorphism by the remark before the proposition. The map is equivariant as
\begin{align*}
h([L_0],[L_1]) & = h g^\inv ([\mc O_E\oplus \mc O_E],[\mc O_E\oplus \pi \mc O_E])\\
&= (gh^\inv)^\inv ([\mc O_E\oplus \mc O_E],[\mc O_E\oplus \pi \mc O_E])\\
& \mapsto f_{\Gamma_0\cdot gh^\inv}\\
&= h\cdot f_{\Gamma_0 \cdot g}.
\end{align*}

There are similar isomorphisms identifying the $0$-chains generated by $\SL_2(E)$-translates of $[\mc O_E\oplus \mc O_E]$ or $[\mc O_E\oplus \pi \mc O_E]$ with $\cind_{U}^{\SL_2(E)}\Lambda$ or $\cind_{U'}^{\SL_2(E)}\Lambda$, respectively, so that the full group of $0$-chains is their direct sum. Then we conclude that $\delta$ is the map extending $f_{\Gamma_0\cdot g}\mapsto f_{U\cdot g}- f_{U'\cdot g}$, which is well-defined since $\Gamma_0=U\cap U'$.
\end{proof}

It is known that $\BT$ is contractible, so that the complex above gives an exact resolution of $\Lambda$. Now fix a smooth $\SL_2(E)$-representation $\sigma$ and an injective resolution $\sigma \hookrightarrow I^0 \to I^1 \to \cdots$. Frobenius reciprocity yields functorial isomorphisms for any open subgroup $H\subset \SL_2(E)$
\[\Hom_{\SL_2(E)}(\cind_H^{\SL_2(E)}\Lambda, \sigma)\cong \Hom_H(\Lambda, \sigma)\cong \sigma^H.\]
In particular, $\Ext_{\SL_2(E)}^q(\cind_{P_F}^{\SL_2(E)}\Lambda,\sigma)\cong H^q(P_F,\sigma)$.
Then there is a double chain complex $C^{p,q}= \Hom_{\SL_2(E)}(\bigoplus_{F\in C_p}\cind_{P_F}^{\SL_2(E)}\Lambda,I^q)$ with associated (vertical) spectral sequence 
\begin{align}\label{sl2-SS}
E_2^{p,q}=H^p\left(\bigoplus_{F\in C_\bullet}H^q(P_F,\sigma)\right)\implies \Ext_{\SL_2(E)}^{p+q}(\Lambda,\sigma)= H^{p+q}(\SL_2(E),\sigma).
\end{align}
Under a mild hypothesis on the torsion of $\Lambda$, this spectral sequence degenerates on the second page. 
\begin{lemma}\label{suff-large-tors-1}
Suppose that the torsion of $\Lambda$ is coprime to $(q-1)^2(q+1)$. Then for any facet $F\in C$, the higher cohomology groups $H^q(P_F,\sigma)$ vanish.
\end{lemma}
\begin{proof}
There is a filtration 
\[1\to U^{(1)}\to U \xrightarrow{\text{mod } \pi} \SL_2(\F_q)\to 1,\]
where $U^{(1)}$ is the pro-$p$ level-$1$ subgroup of elements congruent to the identity mod $\pi$. Then the associated Lyndon-Hochschild-Serre spectral sequence \[H^i\left(U/U^{(1)},H^j(U^{(1)},\sigma)\right)\Rightarrow H^{i+j}(U,\sigma)\]
degenerates immediately as $\sigma$ does not contain $p$-torsion. Then for each $i$ there is an isomorphism $H^i(\SL_2(\F_q),\sigma^{U^{(1)}})\cong H^i(U,\sigma)$. These higher cohomology groups vanish as long as the torsion of $\sigma$ does not divide the order of $\SL_2(\F_q)$, which is $q(q^2-1)$. The exact same argument works for $U'$ since it is isomorphic to $U$.

Similarly, we consider the following filtration of $\Gamma_0$ which induces the inclusion of the standard unipotent matrices into the standard Borel subgroup modulo $\pi$:
\[1\to \Gamma_1 \to \Gamma_0 \to \F_q^\times \times \F_q^\times\to 1.\]
Writing $\Gamma_1= \set{\begin{bmatrix}1+\pi \mc O_E & \mc O_E\\ \pi \mc O_E & 1+\pi \mc O_E \end{bmatrix}}$ shows that $\Gamma_1$ is indeed pro-$p$, so that the $\Gamma_0$-cohomology of $\sigma$ is reduced to the $\F_q^\times \times \F_q^\times$-cohomology as above. Then the higher cohomology groups vanish under the assumption on the torsion of $\Lambda$.
\end{proof}

\begin{theorem}\label{sl2-cohomology}
Let $\sigma$ be a smooth $\SL_2(E)$-representation with coefficients in $\Lambda$, where the torsion of $\Lambda$ is coprime to $(q-1)^2(q+1)$.
Then 
\[H^i(\SL_2(E),\sigma)= \begin{cases}\sigma^{\SL_2(E)},\quad & i=0\\ \sigma^{\Gamma_0}/(\sigma^{U}+\sigma^{U'}),\quad & i=1\\ 0,\quad &\text{else.}\end{cases}\]
\end{theorem}
\begin{remark}
    Note that the only representations that potentially have nonzero cohomology are those $\sigma$ with $\sigma^{\Gamma_0}\neq 0$, i.e. the representations with nonzero Iwahori fixed vectors.
\end{remark}

\begin{proof}
The double complex giving rise to the spectral sequence (\ref{sl2-SS}) can be written in the first quadrant as follows:
\begin{small}
\begin{center}
\begin{tikzcd}[row sep= small, column sep = small]
 & \vdots & \vdots &\\
0 \arrow[r] & \Hom_{\SL_2(E)}(\bigoplus\limits_{F\in C_0}\cind_{P_F}^{\SL_2(E)}\Lambda,I^1 )\arrow[r]\arrow[u] & \Hom_{\SL_2(E)}(\cind_{\Gamma_0}^{\SL_2(E)}\Lambda,I^1 )\arrow[r] \arrow[u]& 0\\
0 \arrow[r] & \Hom_{\SL_2(E)}(\bigoplus\limits_{F\in C_0}\cind_{P_F}^{\SL_2(E)}\Lambda,I^0 )\arrow[r,"\delta_0^*"]\arrow[u,"\partial^0_*",swap] &\Hom_{\SL_2(E)}(\cind_{\Gamma_0}^{\SL_2(E)}\Lambda,I^0 )\arrow[r] \arrow[u]& 0\\
 & 0\arrow[u] & 0 \arrow[u]
\end{tikzcd}
\end{center}
\end{small}
Lemma \ref{suff-large-tors-1} implies vanishing of the vertical cohomology groups above the $0^\text{th}$ row, so that the entire $E_2^{\bullet,\bullet}$-page can be understood by the induced map $\overline{\delta_0^*}$ on the $E_1^{\bullet,\bullet}$-page. A vertical $0$-cohomology class of $\Hom(\bigoplus_F\cind_{P_F}^{\SL_2(E)}\Lambda,-)$ is in the kernel of $\partial^0_*$, and hence is a map $\varphi$ factoring through $\sigma \subset I^0$:
\begin{center}
\begin{tikzcd}
\cind_U^{\SL_2(E)}\Lambda \oplus \cind_{U'}^{\SL_2(E)}\Lambda \arrow[r,"\varphi"]\arrow[dr,dashed, "\overline{\varphi}",swap] & I^0\\
& \sigma. \arrow[u,hookrightarrow]
\end{tikzcd}
\end{center}
Then the vertical cohomology class $[\varphi]$ corresponds to $\overline{\varphi}$ via the isomorphism \[\ker \partial_{*}^0\cong \Hom_{\SL_2(E)}(\bigoplus_F \cind_{P_F}^{\SL_2(E)}\Lambda,\sigma).\]
This means $\overline{\delta_0^*}(\overline{\varphi})$ is simply the precomposition of $\overline{\varphi} \circ \delta :\cind_{\Gamma_0}^{\SL_2(E)}\Lambda\to \sigma$. Recall the series of isomorphisms
\[
\begin{matrix}
\Hom_{\SL_2(E)}(\cind_H^{\SL_2(E)}\Lambda, \sigma) &\xrightarrow{\cong}& \Hom_H(\Lambda, \sigma)&\xrightarrow{\cong}& \sigma^H,\\
\overline{\varphi} &\mapsto& \overline{\varphi}\circ (1\mapsto f_H) &\mapsto& \overline{\varphi}(f_H),
\end{matrix}
\]
where $f_H$ is the indicator function on the trivial coset $H$.
We conclude that the bottom row of the $E_1^{\bullet,\bullet}$-page is the sequence 
\[\cdots \to 0 \to \sigma^{U}\oplus \sigma^{U'}\xrightarrow{\delta^*}\sigma^{\Gamma_0}\to 0\to \cdots,\]
with $\delta^*(v,w)= v-w$ being the difference map (since $\overline{\varphi}\circ \delta(f_{\Gamma_0})=\overline{\varphi}((f_U,-f_{U'}))=\overline{\varphi_1}(f_U)-\overline{\varphi_2}(f_{U'})$).

Since $\delta^*$ is the only nontrivial differential on the first page, we have $H^p(\SL_2(E),\sigma)\cong H^p(0 \to \sigma^{U}\oplus \sigma^{U'}\xrightarrow{\delta^*}\sigma^{\Gamma_0}\to 0)$.
\end{proof}

\begin{example}\label{sl2-cohomology-steinberg}
We use Theorem \ref{sl2-cohomology} to compute the smooth $\SL_2(E)$-cohomology of $\Lambda$, $\ind_B^{\SL_2(E)}\Lambda$, and the Steinberg representation $\st$. Recall there is a split short exact sequence 
\[0\to \Lambda \to \ind_B^{\SL_2(E)}\Lambda\to \st\to 0,\]
where the map on the left is the inclusion by the diagonal.
The theorem immediately implies that the cohomology of $\Lambda$ is $\Lambda$ in degree $0$, and that the degree $0$ cohomology of $\st$ vanishes. Then the long exact sequence gives isomorphisms $H^0(\SL_2(E),\ind_{B}^{\SL_2(E)}\Lambda)\cong \Lambda$ and $H^1(\SL_2(E),\ind_{B}^{\SL_2(E)}\Lambda)\cong H^1(\SL_2(E),\st)$. We proceed to compute the $U,\ U',$ and $\Gamma_0$-orbits acting on $\ind_{B}^{\SL_2(E)}\Lambda$. First, we choose representatives of $\Proj^1(E)=B\backslash \SL_2(E)$ by identifying
\[\begin{bmatrix}0 & -1\\ 1 & t\end{bmatrix}\leftrightarrow [1:t],\quad \begin{bmatrix}1 & 0\\ 0 & 1\end{bmatrix}\leftrightarrow [0:1]=\infty.\]
\begin{itemize}
    \item $\Gamma_0:$ We claim that there are exactly two orbits represented by $\infty$ and $0$. The orbit of $\infty$ is computed by considering an arbitrary product
    \[\begin{bmatrix}\alpha & \beta\\ 0 & \alpha^\inv \end{bmatrix}\begin{bmatrix}1 & 0\\ 0 & 1\end{bmatrix}\begin{bmatrix}a & b\\ \pi c & d\end{bmatrix}=\begin{bmatrix}\alpha a+\pi \beta c & \alpha b +\beta d\\ \alpha^\inv \pi c & \alpha^\inv d\end{bmatrix},\]
    with $\alpha,\beta\in E$ and $a,b,c,d\in \mc O_E$. Whenever $c\neq 0$, choosing $\alpha= \pi c$ and $\beta = -a$ yields
    \[\begin{bmatrix}0 & b\pi c -ad\\ 1 & \frac{d}{\pi c}\end{bmatrix}=\begin{bmatrix}0 & -1\\ 1 & \frac{d}{\pi c}\end{bmatrix}.\] Then since $d\in \mc O_E^\times$ is a unit, $\frac{d}{\pi c}$ has strictly negative valuation. Choosing $c$ to be a non-negative power of $\pi$, we see that the orbit of $\infty $ contains all of the points $[1:t]$ where the valuation of $t$ is strictly negative.\\
    Now we consider the orbit of $[1:t]$ with $t\in \mc O_E$:
    \[\begin{bmatrix}\alpha & \beta\\ 0 & \alpha^\inv \end{bmatrix}\begin{bmatrix}0 & -1\\ 1 & t\end{bmatrix}\begin{bmatrix}a & b\\ \pi c & d\end{bmatrix}=\begin{bmatrix}-\pi \alpha c +\beta (a+t\pi c) & -\alpha d +\beta(b+td) \\ \alpha^\inv (a+t\pi c) & \alpha^\inv (b+td)\end{bmatrix},\]
    and choosing $\alpha = (a+t\pi c)\in \mc O_E^\times$ is allowed since $a\in \mc O_E^\times$. Furthermore set $\beta = \pi c$, so that the matrix becomes 
    \[\begin{bmatrix}0 & -ad -t\pi c d +\pi cd +\pi ctd\\ 1 & (b+td)/(a+t\pi c)\end{bmatrix}=\begin{bmatrix}0 & -1\\ 1 & (b+td)/(a+t\pi c)\end{bmatrix}.\]
    Note that since the only choices so far were made on $\alpha$ and $\beta$, the $\Gamma_0$-orbit will not contain $\infty$ and is hence disjoint. We claim that this orbit is transitive on $\mc O_E= [1:\mc O_E]\subset \Proj^1(E)$. Choose $c=0$ and $a=1$, so that $d$ must be $1$. Then $b\in \mc O_E$ can be picked freely, so that $[1:b+t]=[1:s]$ is in the orbit for any $s\in \mc O_E$.
    \item $U=\SL_2(\mc O_E):$ Similarly we compute 
    \[\begin{bmatrix}\alpha & \beta\\ 0 & \alpha^\inv \end{bmatrix}\begin{bmatrix}1 & 0\\ 0 & 1\end{bmatrix}\begin{bmatrix}a & b\\ c & d\end{bmatrix}=\begin{bmatrix}\alpha a+ \beta c & \alpha b +\beta d\\ \alpha^\inv c & \alpha^\inv d\end{bmatrix},\]
    with $\alpha,\beta\in E$ and $a,b,c,d\in \mc O_E$. Whenever $c\neq 0$, fixing $\alpha = c$ and $\beta = -a$ gives the matrix
    \[\begin{bmatrix}0 & -1\\ 1 & d/c\end{bmatrix},\]
    so that for any $t\in E$, setting $d= tc$ with both $c,d$ integral and $b= (ad-1)/c$ shows that the $U$-action is transitive.
    \item $U'=\begin{bmatrix}\mc O_E & \pi^\inv \mc O_E\\ \pi \mc O_E & \mc O_E\end{bmatrix}\cap \SL_2(E):$ As above, the orbit of $\infty$ can be computed by 
    \[\begin{bmatrix}\alpha & \beta\\ 0 & \alpha^\inv \end{bmatrix}\begin{bmatrix}1 & 0\\ 0 & 1\end{bmatrix}\begin{bmatrix}a & \pi^\inv b\\ \pi c & d\end{bmatrix}=\begin{bmatrix}\alpha a+ \pi \beta c & \alpha \pi^\inv b +\beta d\\ \alpha^\inv \pi c & \alpha^\inv d\end{bmatrix}.\]
    As in the previous two cases, suppose $c\neq 0$ and set $\alpha =\pi c$, $\beta = -a$, reducing to the form
    \[\begin{bmatrix}0 & -1\\ 1 & d/\pi c\end{bmatrix}.\]
    It cannot be the case that both $d,c\in \mc O_E$ are non-units. If $c\in \mc O_E^\times,$ without loss of generality $c=1$ so that $\frac{d}{\pi c}= \frac{1}{\pi}d$ is an arbitrary element of $E$ with valuation $\geq -1$. Otherwise, $c\in \pi \mc O_E$ and $d$ is a unit (without loss of generality $d=1$), which gives $\frac{d}{\pi c}= \frac{1}{\pi^2 c'}$ where $c'\in \mc O_E$ is arbitrary. That is, every element of $E$ with valuation $<-1$ is also in the orbit of $\infty$, so that the $U'$-action is also transitive. 
\end{itemize}
Then the $\Gamma_0$-fixed points are the locally constant functions on $\Proj^1(E)$ that are constant on the two orbits, hence $(\ind_B^{\SL_2(E)}\Lambda)^{\Gamma_0}\cong \Lambda^2$ (note that such functions are smooth vectors, so they exist in the smooth subrepresentation). Furthermore, the $U$- and $U'$-fixed points are the constant functions $\Lambda\smallvector{1}{1}\subset \Lambda^2$. 
Now, the short exact sequence $0\to \Lambda\to \ind_{B}^{\SL_2(E)} \Lambda \to \st \to 0$ is sent under $R\Hom(\Lambda,-)$ to the sequence
\begin{center}
\begin{tikzcd}[ampersand replacement=\&]
	0  \arrow[r]\& \Lambda^2 \arrow[r,"\mathrm{id}"]\arrow[d,"{(1,-1)}",swap] \& \Lambda^2 \arrow[r]\arrow[d,"{ \begin{pmatrix} 1 & -1 \\ 1 & -1\end{pmatrix} }"]\& \st^{U}\oplus \st^{U'} \arrow[d]\arrow[r]\& 0 \\
	0 \arrow[r] \& \Lambda \arrow[r,"{\smallvector{1}{1}}"] \& \Lambda^2 \arrow[r] \& \st^{\Gamma_0}\arrow[r]\& 0,
\end{tikzcd}
\end{center}
inducing the short exact sequence
\[0\to \Lambda\smallvector{1}{1} \xrightarrow{\mathrm{id}} \Lambda\smallvector{1}{1} \to H^0(\SL_2(E),\st)\to 0 \to \Lambda^2/  \smallvector{1}{1}\Lambda \to H^1(\SL_2(E),\st) \to 0.  \]
We conclude that $H^1(\SL_2(E),\ind_B^{\SL_2(E)}\Lambda)\cong \Lambda^2/\smallvector{1}{1}\cdot \Lambda \cong \Lambda$ has elements represented by functions that are constant on $\mc O_E\subset \Proj^1(E)$ and $0$ on the complement.
\end{example}

The resolution of Schneider and Stuhler can also be used to compute smooth $\SL_2(E)$-homology groups. We remark that there is a similar spectral sequence as above filtering the $\SL_2(E)$-homology groups by the homology of $U,U',$ and $\Gamma$-modules. However in our case, there is already a quite strong duality statement relating the smooth dual of homology to the cohomology of the smooth dual. 
\begin{definition}
Suppose $G$ is a locally pro-$p$ group. Let $\mc S(G)$ be the algebra of compactly supported locally constant $\Lambda$-valued functions on $G$, with pointwise addition and multiplication given by convolution
\[(f*f')(g)=\int_{h\in G}f(gh^\inv)f'(h) d\mu_G.\] Let $\mc H(G)$ be the algebra of compactly supported locally constant distributions on $\mc S(G)$, with multiplication given by convolution of distributions. $\mc H(G)$ is called the \textit{Hecke algebra} of $G$.
\end{definition}
\begin{remark}\hfill
\begin{itemize}
    \item Fixing a Haar measure $\mu_G$ gives an isomorphism $\mc S(G)\xrightarrow{ \sim} \mc H(G)$ via multiplication by $\mu_G$.
    \item There is a $\Lambda$-linear equivalence of categories between smooth $G$-representations and non-degenerate modules over $\mc H(G)$.
    \item Since $G$ is locally pro-$p$ and $\Lambda$ is discrete, the locally constant functions are precisely the continuous maps. Then $C^0(G,\Lambda)= \varinjlim_K C^0(K\backslash G, \Lambda)$, and compact support is equivalent to having finite support at each level. Then there is a $G$-equivariant isomorphism $\mc S(G)\cong \varinjlim_K \Lambda[K\backslash G]$.
\end{itemize}
\end{remark}

The following proposition is a special case of \cite[Proposition~VII.4.2]{fargues2021geometrization}.
\begin{proposition}\label{dual-of-homology}
Let $\sigma$ be a smooth representation of a locally pro-$p$ group $G$ with coefficients in $\Lambda$. Then for all $i$ there is an isomorphism
\[H_{i}(G,\sigma)^\vee \cong H^{i}(G,\sigma^\vee).\]
\end{proposition}
\begin{proof}
This arises from an adjunction in the derived category of solid $G$-modules (i.e. it is closed symmetric monoidal) and the fact that the derived category of smooth representations embeds fully faithfully (as $\Lambda$ is discrete). 
Since taking smooth vectors is exact and takes injectives to smooth injectives under our assumptions on $\Lambda$ and $G$ \cite[Proposition~X.1.5]{borel2000continuous}, the adjunction takes the form
\[R\Hom_{\Lambda}(\sigma \otimes_{\mc H(G)}^{\mathbb L}\Lambda,\Lambda)^\text{sm}\cong R\Hom_{\mc H(G)}(\Lambda, R\Hom_{\Lambda}(\sigma, \Lambda)^\text{sm}).\]
Furthermore, taking smooth duals is also exact \cite[Proposition~2.10]{Bushnell2006} and $\Lambda$ is self-injective, inducing
\[\Hom_{\Lambda}(\sigma \otimes_{\mc H(G)}^{\mathbb L} \Lambda, \Lambda)^\text{sm}\cong R\Hom_{\mc H(G)}(\Lambda, \sigma^\vee).\]
Exactness of $(-)^\vee$ means that the cohomology of the lefthand side is the same as the smooth dual of the smooth $G$-homology of $\sigma$.
\end{proof}

\subsection{Compact induction from \texorpdfstring{$\SL_2(E)$}{SL2(E)} to \texorpdfstring{$\GL_2(E)$}{GL2(E)}}
We record some lemmas regarding compact induction along the determinant short exact sequence of $\GL_2(E)$. We also compute the smooth homology of representations of a one dimensional split torus, which appears in the Hochschild-Serre spectral sequence associated to $\SL_2(E)\subset \GL_2(E)$.
\begin{lemma}\label{tensor-sl2-ind}
Let $\sigma$ be a smooth representation of $\GL_2(E)$. Then there is an isomorphism 
\[\sigma \otimes \cind_{\SL_2(E)}^{\GL_2(E)}\Lambda\xrightarrow{\sim} \cind_{\SL_2(E)}^{\GL_2(E)}\sigma\]
as $\SL_2(E)$-modules.
\end{lemma}
\begin{proof}

There is a decreasing filtration $\GL_2(E)\supset H_1\supset \cdots $ where $H_i:= \det^\inv(1+\pi^i\mc O_E)$. Then each $H_i$ is an open subgroup of $\GL_2(E)$ (hence the coset space is discrete) with $\SL_2(E)=\bigcap_i H_i$; write $Q_i:= H_i\backslash \GL_2(E)$ so that $E^\times \cong \SL_2(E)\backslash \GL_2(E)\cong \varprojlim Q_i$. We claim $\cind_{\SL_2(E)}^{\GL_2(E)}\sigma \cong \varinjlim \cind_{H_i}^{\GL_2(E)}\sigma$. The idea is that compact support inside $E^\times$ of a fixed function is profinite, so such a function in $C^0(E^\times,V)\cong \varinjlim C^0(H_i \backslash \GL_2(E),V)$ is determined by a function on some finite level with finite (i.e. compact modulo $H_i$) support. Since there are isomorphisms of $\SL_2(E)$-representations 
\[\sigma\otimes \cind_{H_i}^{\GL_2(E)}\Lambda\cong \sigma \otimes \bigoplus_{Q_i}\Lambda \cong \bigoplus_{Q_i}\sigma\cong \cind_{H_i}^{\GL_2(E)}\sigma,\] 
we see that $\sigma \otimes \cind_{\SL_2(E)}^{\GL_2(E)}\Lambda \cong \cind_{\SL_2(E)}^{\GL_2(E)}\sigma$ as $\SL_2(E)$-representations. 
\end{proof}

\begin{lemma}
For any smooth $\SL_2(E)$-representation $\sigma$, for every $i$ there is an isomorphism on smooth homology
\[H_i(\GL_2(E),\cind_{\SL_2(E)}^{\GL_2(E)} \sigma)\cong H_i(\SL_2(E),\sigma).\]
\end{lemma}
\begin{remark}
This is dual to Shapiro's lemma $H^i(\GL_2(E),\ind_{\SL_2(E)}^{\GL_2(E)}\sigma )\cong H^i(\SL_2(E),\sigma)$.
\end{remark}

\begin{proof}
Since both $\SL_2(E)$ and $\GL_2(E)$ are unimodular and $\SL_2(E)\subset \GL_2(E)$ is closed, there is a natural isomorphism of functors \cite[116]{Renard2010ReprsentationsDG}
\[\cind_{\SL_2(E)}^{\GL_2(E)}\circ (-\otimes \delta_{\SL_2(E)\backslash \GL_2(E)})=\cind_{\SL_2(E)}^{\GL_2(E)}\xrightarrow{\sim} \mc H(\GL_2(E))\otimes_{\mc H(\SL_2(E))}-.\]
The same is true for $1\subset \GL_2(E)$, so that \[\mc H(\GL_2(E))\cong \cind_{1}^{\GL_2(E)}\Lambda \cong \cind_{\SL_2(E)}^{\GL_2(E)}\mc H(\SL_2(E))\cong_{\SL_2(E)}\varinjlim \bigoplus_{Q_i}\mc H(\SL_2(E))\]
is a flat $\mc H(\SL_2(E))$-module. Then in the derived category of $\Lambda$-modules:
\begin{align*}
\Lambda \otimes_{\mc H(\GL_2(E))}^\mathbb{L}\cind_{\SL_2(E)}^{\GL_2(E)}\sigma &\cong \Lambda \otimes_{\mc H(\GL_2(E))}^\mathbb{L}(\mc H(\GL_2(E))\otimes_{\mc H(\SL_2(E))}\sigma) \\
&\cong \Lambda \otimes_{\mc H(\GL_2(E))}^\mathbb{L}(\mc H(\GL_2(E))\otimes^\mathbb{L}_{\mc H(\SL_2(E))}\sigma)\\
&\cong \left(\Lambda\otimes_{\mc H(\GL_2(E))}^\mathbb{L}\mc H(\GL_2(E)) \right) \otimes_{\mc H(\SL_2(E))}^\mathbb{L} \sigma \\
&\cong \Lambda \otimes_{\mc H(\SL_2(E))}^\mathbb{L} \sigma.
\end{align*}

\end{proof}

We denote by $\pi^{\Z}:= E^\times /\mc O_E^\times$ the rank $1$ free abelian group generated by the uniformizer of $E$. In particular, $\pi^{\Z}$ is discrete. 
\begin{proposition}\label{homology-E-times}
Let $\sigma$ be a smooth $E^\times$-module. If the torsion of $\Lambda$ is sufficiently large, then 
\[H_i(E^\times, \sigma)\cong
\begin{cases}
\sigma_{E^\times},\quad & i=0\\ (\sigma_{\mc O_E^\times})^{\pi^{\Z}},\quad & i=1\\ 0,\quad & \text{else.}
\end{cases}\]
\end{proposition}
\begin{proof}
Consider the homological Hochschild-Serre spectral sequence 
\[E_2^{i,j}= H_{-i}(\pi^{\Z}, H_{-j}(\mc O_E^\times, \sigma))\implies H_{-i-j}(E^\times, \sigma).\]
The unit group $\mc O_E^\times$ is the direct product of a pro-$p$ group with $\Z/(q-1)\Z$, so another Hochschild-Serre spectral sequence shows that $\sigma$ is acyclic as an $\mc O_E^\times$-module (as long as the torsion of $\Lambda$ and $q-1$ have no divisors in common). 

Then the spectral sequence is just the $\pi^{\Z}$-homology of the co-fixed points $\sigma_{\mc O_E^\times}$. Since $\pi^{\Z}$ is discrete, its category of representations is equivalent to modules over the group ring $\Lambda[\pi^{\Z}]$. Then taking the standard free resolution of the trivial module
\[0\to (\pi -1)\Lambda[\pi^{\Z}]\xrightarrow{\iota} \Lambda[\pi^{\Z}]\xrightarrow{\Sigma} \Lambda \to 0\]
gives the group homology of $\sigma_{\mc O_E^\times}$ as the homology of the complex
\[0 \to (\pi-1)\sigma_{\mc O_E^\times}\xrightarrow{\iota\otimes 1}\sigma_{\mc O_E^\times} \to 0,\]
which is easily seen to be the $\pi^{\Z}$ coinvariants in degree $0$ and invariants in degree $1$.
\end{proof}